\documentclass[10pt]{amsart}
\usepackage[utf8]{inputenc}
\usepackage[T1]{fontenc}
\usepackage[english]{babel}
\title[Stability for Martin boundaries in rel. hyp. groups]
{Stability phenomena for Martin boundaries of relatively hyperbolic groups}
\author{Matthieu Dussaule}
\author{Ilya Gekhtman}

\date{}
\usepackage{authblk}
\usepackage{silence}
\usepackage{comment}
\usepackage{amsmath}
\usepackage{amssymb}
\usepackage{dsfont}
\usepackage{amsfonts}
\usepackage{amsthm}
\usepackage{tikz}
\usepackage{graphicx}
\usepackage{fancyhdr}
\usepackage{caption}
\usepackage{enumerate}
\usepackage{mathtools}
\usepackage[colorlinks=true,pdfstartview=FitH,bookmarks=false]{hyperref}
\hypersetup{urlcolor=black,linkcolor=black,citecolor=black,colorlinks=true}

\newcommand\N{\mathbb{N}}
\newcommand\Z{\mathbb{Z}}

\newcommand\R{\mathbb{R}}

\theoremstyle{definition}
\newtheorem{hyp}{Assumption}

\theoremstyle{plain}
\newtheorem{definition}{Definition}[section]
\newtheorem{proposition}[definition]{Proposition}
\newtheorem{corollary}[definition]{Corollary}
\newtheorem{theorem}[definition]{Theorem}
\newtheorem{lemma}[definition]{Lemma}
\newtheorem*{thm*}{Theorem}
\newtheorem*{prop*}{Proposition}
\newtheorem*{lem*}{Lemma}

\theoremstyle{remark}

\newtheorem*{rem*}{Remark}

\DeclareMathOperator{\Cay}{Cay}

\DeclareMathOperator{\supp}{supp}
\DeclareMathOperator{\Tr}{Tr}

\DeclareMathOperator{\Leb}{Leb}
\DeclareMathOperator{\reg}{reg}

\usepackage{etoolbox}
\apptocmd{\sloppy}{\hbadness 10000\relax}{}{}
\apptocmd{\sloppy}{\vbadness 10000\relax}{}{}

\begin{document}

\begin{abstract}
    Let $\Gamma$ be a relatively hyperbolic group and let $\mu$ be an admissible symmetric finitely supported probability measure on $\Gamma$.
    We extend Floyd-Ancona type inequalities from \cite{GGPY} up to the spectral radius $R$ of $\mu$.
    We use them to find the precise homeomorphism type of the $r$-Martin boundary, which describes $r$-harmonic functions, for every $r\leq R$.
    We also define a notion of spectral degeneracy along parabolic subgroups which is crucial to describe the homeomorphism type of the $R$-Martin boundary.
    Finally, we give a criterion for (strong) stability of the Martin boundary in the sense of Picardello and Woess \cite{PicardelloWoess} in terms of spectral degeneracy.
    We then prove that this criterion is always satisfied in small rank,
    so that in particular,
    the Martin boundary of an admissible symmetric finitely supported probability measure on a geometrically finite Kleinian group of dimension at most 5 is always strongly stable.
\end{abstract}

\maketitle

\section{Introduction}\label{Sectionintroduction}
\subsection{Martin boundaries and stability}
We consider a finitely generated group $\Gamma$ together with a probability measure $\mu$.
The $\mu$-random walk starting at $e$ on $\Gamma$ is defined as
$X_n=g_1...g_n$,
where the $g_k$ are independent, identically distributed according to $\mu$ random variables on $\Gamma$.
We always assume that the random walk is admissible: the support of $\mu$ generates the group $\Gamma$ as a semi-group.
In many situations, one can understand the asymptotic behaviour of $X_n$ in terms of geometric properties at large scale of $\Gamma$.

One way to do so is to try to compare geometric boundaries of the group, encoding how geodesics behave at infinity, with probabilistic boundaries, encoding how the random walk behaves asymptotically.
In the groups we will study here, the random walk will always be transient, meaning that it almost surely goes to infinity.
One can then define the \textit{Martin boundary} as follows.

Define the Green function $G(\cdot,\cdot)$ by
$$G(x,y)=\sum_{n\geq 0}\mu^{*n}(x^{-1}y),$$
where $\mu^{*n}$ is the $n$th convolution power of $\mu$.
Precisely,
$$\mu^{*n}=\sum_{z_0=x,z_1,...,z_n=y}\mu(z_0^{-1}z_1)\mu(z_1^{-1}z_2)...\mu(z_{n-1}^{-1}z_n).$$

Define then the Martin kernel $K(\cdot,\cdot)$ by
$$K(x,y)=\frac{G(x,y)}{G(e,y)}.$$
We endow $\Gamma$ with the discrete topology.
The Martin compactification of $\Gamma$ and $\mu$ is the smallest compact metrizable set $X$ such that $\Gamma$ embeds as a dense and open set in $X$ and such that the function $K(\cdot,\cdot)$ extends as a continuous function on $\Gamma \times X$.
In other words, a sequence $g_n$ in $\Gamma$ converges to a point in the Martin compactification if and only if
for every $g\in \Gamma$, $K(g,g_n)$ converges to some limit.
The Martin boundary is the complement of $\Gamma$ in $X$.
The Martin compactification and the Martin boundary always abstractly exist and they are unique up to homeomorphism, see \cite{Kaimanovichboundaries} or \cite{Sawyer}.
We will denote the Martin boundary by $\partial_{\mu}\Gamma$.

The random walk almost surely converges to a point in the Martin boundary (see \cite{Sawyer}).
Letting $X_{\infty}$ be the corresponding random variable in $\partial_{\mu}\Gamma$, we can consider the law of $X_{\infty}$.
This yields a probability measure $\nu$ on $\partial_{\mu}\Gamma$ which is called the \textit{harmonic measure}.
The Martin boundary, together with $\nu$, is a model for the so-called \textit{Poisson boundary}.
We will not be interested in the Poisson boundary in the following, so we do not give more details and instead refer to \cite{Vershik}, \cite{KaimanovichVershik} or \cite{Kaimanovichannals}.

\medskip
More generally, one can define the weighted Green function at $r$ by
$$G_r(x,y)=\sum_{n\geq 0}r^n\mu^{*n}(x^{-1}y).$$
Let $R$ be the radius of convergence of this power series, which we call the \textit{spectral radius} of $\mu$.
All groups we will study in this paper are non-amenable.
It thus follows from a result of Guivarc'h (see \cite[p.~85, remark~b)]{Guivarch}) that for every $x,y$, $G_{R}(x,y)<+\infty$ (see also \cite[Theorem~7.8]{Woess-book} for a stronger statement).
For $r\leq R$, we then define similarly the $r$-Martin kernel by
$$K_r(x,y)=\frac{G_r(x,y)}{G_r(e,y)}$$
and the $r$-Martin compactification and boundary.
As before, a sequence $g_n$ in $\Gamma$ converges to a point in the $r$-Martin compactification if and only if
for every $g\in \Gamma$, $K_r(g,g_n)$ converges to some limit.
The $r$-Martin compactification also always abstractly exist, see \cite{Woess-book} and is unique up to homeomorphism.
We will denote the $r$-Martin boundary by $\partial_{r\mu}\Gamma$.

One important aspect of the $r$-Martin boundary is that it gives a description of $r$-harmonic positive functions in the following sense.
Recall that a function $f:\Gamma \to \R$ is called $r$-harmonic if for all $x\in \Gamma$
$$\sum_{y\in \Gamma}\mu(x^{-1}y)f(y)=rf(x).$$
This can be written as $Pf=rf$, where $P$ is the Markov operator associated with $\mu$.
Every $r$-harmonic positive function can be represented as an integral over the Martin boundary.
Precisely, for every such function $f$, there exists a probability measure $\nu_f$ on $\partial_{r\mu}\Gamma$ such that for all $x\in \Gamma$,
$$f(x)=\int_{\partial_{r\mu}\Gamma}K(x,\xi)\nu_f(d\xi).$$
In general, the measure $\nu_f$ is not unique.
To obtain uniqueness, we restrict our attention to the minimal boundary that we now define.

An $r$-harmonic positive function $f$ is called minimal if for every other $r$-harmonic positive function $g$ satisfying $g\leq Cf$ for some constant $C$, we have
$g=C'f$ for some constant $C'$.
The $r$-minimal Martin boundary is then the set
$$\partial_{r\mu}^{m}\Gamma=\{\xi \in \partial_{r\mu}\Gamma, K_r(\cdot,\xi) \text{ is minimal harmonic}\}.$$

Let $f$ be an $r$-harmonic positive function.
Then, one can choose $\nu_f$ giving full measure to $\partial_{r\mu}^m\Gamma$ and in this case, $\nu_f$ is unique.
This is the so-called Martin representation Theorem.

\medskip
In \cite{PicardelloWoess}, Picardello and Woess define the stability of the Martin boundary of a random walk (or more generally of a Markov chain) as follows (see also \cite[IV.28.A]{Woess-book}).

\begin{definition}
Let $\Gamma$ be a finitely generated group and $\mu$ be a probability measure on $\Gamma$.
Let $R$ be the spectral radius of $\mu$.
Say that the Martin boundary is stable if the following conditions hold.
\begin{enumerate}
    \item For every $0<r_1,r_2< R$, the $r_1$ and $r_2$-Martin compactifications are homeomorphic, that is, $K_{r_1}(\cdot, y_n)$ converges pointwise if and only if $K_{r_2}(\cdot,y_n)$ converges pointwise.
    For simplicity we then write then $\partial_{\mu}\Gamma$ for the $r$-Martin boundary whenever $0<r<R$.
    \item The identity on $\Gamma$ extends to a continuous and equivariant surjective map $\phi_{\mu}$ from $\Gamma\cup \partial_{\mu}\Gamma$ to $\Gamma \cup \partial_{R\mu}\Gamma$.
    We then write $K_R(x,\xi)=K_R(x,\phi_{\mu}(\xi))$ for $\xi\in \partial_\mu\Gamma$.
    \item The map $(r,x,\xi)\in (0,R]\times \Gamma \times \partial_\mu\Gamma \mapsto K_r(x,\xi)$ is continuous in the three variables $(r,x,\xi)$.
\end{enumerate}
Also say that the Martin boundary is strongly stable if the first condition holds for every $0<r_1,r_2\leq R$ so that in particular, the map $\phi_{\mu}$ induces a homeomorphism from the $r$-Martin boundary to the $R$-Martin boundary.
\end{definition}

In general, identifying the homeomorphism type of the Martin boundary is difficult
and there are few example for which we know whether the Martin boundary is stable.
Often, all the $r$-Martin compactifications are homeomorphic for $r<R$ but not at $r=R$.

Let us give some examples.
We start with the classical case of finitely supported random walks in $\Z^d$.
It is well known there are no non-constant positive $R$-harmonic function and so the $R$-Martin boundary is reduced to a point.
Actually to ensure the existence of the $R$-Martin boundary, we need to ensure that the random walk is $R$-transient, that is $G_R(x,y)$ is finite, so we need to assume that $d\geq 3$ (see \cite[\S 7.B]{Woess-book} for more details).
On the other hand, the $r$-Martin boundary is homeomorphic to a $d-1$-dimensional sphere for $r<R$, as proved by Ney and Spitzer in some particular case \cite{NeySpitzer}.
We refer to \cite[\S 25.B]{Woess-book} for a complete discussion.
In particular, the Martin boundary is stable but not strongly stable.
Moreover, in this situation, the $r$-Martin boundary is minimal for every $r\leq R$.


For nilpotent groups, we do not know in general the homeomorphism type of the $r$-Martin boundary. However, according to results of Margulis \cite{Margulis}, we know that for finitely supported random walks, it is not reduced to a point for $r<R$, whereas the $R$-Martin boundary is reduced to a point. We do not know if the $r$-Martin boundary is minimal when $r< R$.

The situation in hyperbolic groups is different.
Gou\"ezel \cite{Gouezel-local} proved that the Martin boundary of a finitely supported symmetric probability measure on a nonelementary Gromov-hyperbolic group is strongly stable.
For every $r\leq R$, the $r$-Martin boundary  is always minimal and coincides with the Gromov boundary.
This was already proved by Gou\"ezel and Lalley \cite{GouezelLalley} for co-compact Fuchsian groups.

\subsection{Random walks on relatively hyperbolic groups}
In this paper, we are interested in symmetric finitely supported random walks on relatively hyperbolic groups.
We will give a precise definition of these groups below.
Archetypal examples are given by free products and finite co-volume (or more generally geometrically finite) Kleinian groups.
We are particularly interested in the case where the parabolic subgroups (which are the conjugates of the free factors for the particular case of a free product) are virtually abelian,
although a lot of our results still hold for any relatively hyperbolic group.

The Martin boundary of such a group has a strong relation with the so-called Bowditch boundary, the Gromov boundary of a Gromov hyperbolic space on which the group acts geometrically finitely, see \cite{GGPY}.
We will give a precise definition of the Bowditch boundary below.
In the case of a free product, it is constructed by gluing one point at infinity for every coset of every free factor to the set of infinite words.
If the free factors are one-ended, the Bowditch boundary coincides with the set of ends.
The relation between the set of ends and the Martin boundary was explored by Woess in \cite{Woess-ends}.
In the case of a finite co-volume Kleinian group, the Bowditch boundary coincides with the ideal boundary of the hyperbolic space $\mathbb{H}^n$.

The 1-Martin boundary of a finitely supported probability measure on a relatively hyperbolic group is described in \cite{DGGP}.
It is obtained from the Bowditch boundary by blowing up each parabolic point into a sphere of dimension one less than the rank of its stabilizer.
Roughly speaking, those spheres appear as the Martin boundary of the induced random walk on the virtually abelian parabolic subgroups.
The term "induced random walk" refers to the sub-Markov chain corresponding to first return transition kernel to the parabolic subgroup, see Section~\ref{Sectionspectraldegeneracy}.
In other words, the 1-Martin boundary roughly consists of the Martin boundary of the parabolic subgroups, glued to the conical limit points.
In this paper, we will give a precise description of the homeomorphism type of the $r$-Martin boundary for every $r\leq R$.
We will also prove that stability of the Martin boundary depends on the behaviour of the measure $\mu$ along parabolic subgroups.

\medskip
Before giving a precise statement of our results, let us give a brief explanation of what happens in the particular case of adapted random walks on a free product.
Let $\Gamma=\Gamma_1*\Gamma_2$ be a free product of two finitely generated groups $\Gamma_1$ and $\Gamma_2$.
Then, $\Gamma$ is hyperbolic relative to the conjugates of $\Gamma_1$ and $\Gamma_2$.
Recall that a probability measure $\mu$ on $\Gamma$ is called adapted if it can be written as
$$\mu=\alpha \mu_1+(1-\alpha)\mu_2,$$
where $\mu_i$ is a probability measure on $\Gamma_i$ and where $\alpha>0$.
Let $G_r$ be the Green function for $\mu$ on $\Gamma$ and let $G_r^{(i)}$ be the Green function for $\mu_i$ on $\Gamma_i$.
Then, in \cite{Woess-Martin}, Woess shows that for every $r$, there exists $\rho_i(r)$ such that for all $x,y\in \Gamma_i$,
 $G_r(x,y)=G_{\rho_i(r)}^{(i)}(x,y)$.
Moreover, letting $R_i$ be the spectral radius of $\mu_i$
we have $\rho_i(r)<R_i$ whenever $r<R$.
Call $\mu$ spectrally degenerate along $\Gamma_i$ if $\rho_i(R)=R_i$.
In other words, $\mu$ is spectrally degenerate along $\Gamma_i$ if when $\mu$ reaches it spectral radius $R$, then $\rho_i$ reaches the spectral radius $R_i$ of $\mu_i$.

It is then proved in \cite{Woess-Martin} that the $r$-Martin boundary of $\Gamma$ consists of a copy of the $\rho_i(r)$-Martin boundary of $(\Gamma_i,\mu_i)$ for each coset $\gamma \Gamma_i$ glued to the set of infinite words, see \cite{Woess-Martin} for a precise statement.
In particular, if the groups $\Gamma_i$ are virtually abelian, then the Martin boundary of $(\Gamma,\mu)$ can only be strongly stable if the random walk is not spectrally degenerate along $\Gamma_1$ and $\Gamma_2$.

\medskip
We will generalize Woess's results to any relatively hyperbolic group with virtually abelian parabolic subgroups.
We thus need to define a notion of spectral degeneracy along parabolic subgroups which mimics the one above.
We will give a precise formulation in Section~\ref{Sectionspectraldegeneracy}.
Let us give a rough explanation here.
Consider a relatively hyperbolic group $\Gamma$ and let $H$ be a parabolic subgroup.
The main difference with the above situation is that there is no a priori probability measure like $\mu_i$ on $H$.
We introduce the first return transition kernel $p_r$ on $H$ associated with $r\mu$, $r\leq R$.
Then, the spectral radius of this induced kernel $p_r$ is at least 1.
We say that $\mu$ is spectrally degenerate along $H$ if the spectral radius of $p_R$ is exactly 1.
This is another way of evaluating if the induced walk on $H$ reaches its spectral radius as the original $\mu$-random walk on $\Gamma$ reaches its own spectral radius.
We equivalently say that the parabolic limit point fixed by $H$ is spectrally degenerate.

\medskip
We will identify the $R$-Martin boundary with a geometric boundary that we now informally describe, see Section~\ref{Sectionspectraldegeneracy} for a precise construction.
We start with the Bowditch boundary.
If $\xi$ is a parabolic limit point of $\Gamma$ whose stabilizer is a virtually abelian group of rank $d$ and if $\xi$ is not spectrally degenerate, we replace $\xi$ by a sphere of dimension $d-1$.
We do not replace spectrally degenerate parabolic limit points.

We will also identify the $r$-Martin boundary, $r<R$, with a geometric boundary.
In this case, we need to replace every parabolic limit point, spectrally degenerate or not, with a sphere of the appropriate dimension.
We will use the following terminology.
\begin{enumerate}
    \item Whenever $r<R$ we call a boundary obtained by replacing every parabolic fixed point in the Bowditch boundary with a sphere of dimension one less than the rank of its stabilizer a \textit{$r$-geometric boundary}.
    \item We call a boundary obtained by replacing every spectrally non-degenerate parabolic limit point in the Bowditch boundary with a sphere of the appropriate dimension a \textit{$R$-geometric boundary}.
\end{enumerate}
Note that the $R$-geometric boundary depends on the measure $\mu$, whereas for $r<R$ the $r$-geometric boundary does not.
Also note that the identity map on $\Gamma$ extends to a continuous equivariant  surjection from the $r$-geometric compactification to the $R$-geometric compactification, which consists of collapsing the added spheres onto one point.
Our main result is the following identification of the homeomorphism type of the Martin boundary.

\begin{theorem}\label{maintheorem}
Let $\Gamma$ be a non-elementary relatively hyperbolic group with respect to a collection of virtually abelian subgroups.
Let $\mu$ be a symmetric probability measure on $\Gamma$ whose finite support generates $\Gamma$ and let $R$ be its spectral radius.
Then for every $r\leq R$, the identity of $\Gamma$ extends to an equivariant homeomorphism from the $r$-geometric boundary to the $r$-Martin boundary.

In particular, the identity of $\Gamma$ extends to a continuous equivariant surjection $\phi_{\mu}$ from the $r$-Martin compactification to the $R$-Martin compactification.
\end{theorem}

We also prove that the Martin boundary is strongly stable when the measure is not spectrally degenerate.

\begin{theorem}\label{mainthmstability}
Let $\Gamma$ be a non-elementary relatively hyperbolic group with respect to a collection of virtually abelian subgroups.
Let $\mu$ be a symmetric probability measure on $\Gamma$ whose finite support generates $\Gamma$ and let $R$ be its spectral radius.
Assume that $\mu$ is spectrally non-degenerate.
Let $\partial_\mu\Gamma$ denote the $r$-Martin boundary for $0<r\leq R$.
The map $(r,x,\xi) \mapsto K_r(x,\xi)$ is continuous in $(x,r,\xi)\in (0,R]\times \Gamma\times \Gamma \cup \partial_\mu\Gamma$.
\end{theorem}

Note that the third variable $\xi$ is allowed to vary in the whole compactification $\Gamma \cup \partial_{\mu}\Gamma$ and not only in the boundary $\partial_\mu\Gamma$.
We can also prove that the $r$-Martin boundary is always minimal.

\begin{theorem}\label{minimality}
Let $\Gamma$ be a non-elementary relatively hyperbolic group with respect to a collection of virtually abelian subgroups.
Let $\mu$ be a symmetric probability measure on $\Gamma$ whose finite support generates $\Gamma$ and let $R$ be its spectral radius.
Then, for every $0<r\leq R$, the $r$-Martin boundary is minimal.
\end{theorem}

We will also show that the random walk cannot be spectrally degenerate at parabolic subgroups of small rank.
More precisely, let $H$ be a parabolic subgroup which is virtually abelian of rank $d$.
If $d\leq 4$, then the random walk is not spectrally degenerate at $H$.
In particular, the Martin boundary is stable in small dimension.
This implies the following.

\begin{theorem}\label{stabilitylowdimension}
Let $\Gamma$ be the fundamental group of a geometrically finite hyperbolic manifold of dimension $d\leq 5$.
Let $\mu$ be a symmetric probability measure whose finite support generates $\Gamma$.
Let $R$ be its spectral radius.
Then, for every $r\leq R$, the $r$-Martin boundary coincides with the CAT(0) boundary of $\Gamma$.
In particular, the Martin boundary is strongly stable.
\end{theorem}

We emphasize that Theorem~\ref{maintheorem} is optimal and cannot be extended in the following sense :
there exist relatively hyperbolic groups with spectrally degenerate measures. Indeed, even for adapted measures on free products of abelian groups where one factor has rank at least $5$, Candellero and Gilch \cite{Candellero-Gilch} construct examples of both spectrally degenerate and spectrally non-degenerate measures, so the Martin boundary may or may not be strongly stable.
Moreover, the Candellero-Gilch construction might be adapted to construct spectrally degenerate measures on any relatively hyperbolic group with respect to virtually abelian groups of large enough rank.

\subsection{Relative form of Ancona inequalities up to the spectral radius}
Let us briefly explain how we prove these results.
Whenever $f$ and $g$ are two functions satisfying that there exists a constant $C$ such that $\frac{1}{C}f\leq g \leq C g$, we will write
$f\asymp g$. If the constant depends on some parameter, we will not use this notation to avoid confusion, except if the parameter is clear from the context.
Also, whenever $f\leq Cg$ for some constant $C$, we will write $f\lesssim g$.

Ancona \cite{Ancona} proved that the Martin boundary of a finitely supported admissible random walk on a hyperbolic group coincides with the Gromov boundary. To prove this, he used the following deviation inequalities.
If $x,y,z$ are three points in $\Gamma$ such that $y$ is on a geodesic $[x,z]$ from $x$ to $z$, then we have
\begin{equation}\label{weakAnconahyperbolic}
    G(x,z)\asymp G(x,y)G(y,z).
\end{equation}
The implicit constant only depends on the hyperbolicity parameters of the group.
We can restate these inequalities saying that the random walk must pass within a bounded distance of $y$ with high probability.

To prove that the Martin boundary is stable, Gou\"ezel \cite{Gouezel-local} showed that these deviation inequalities still hold at the spectral radius.
For every $r\leq R$, we have
\begin{equation}\label{weakAnconahyperbolic'}
    G_r(x,z)\asymp G_r(x,y)G_r(y,z).
\end{equation}
where the implicit constant does not depend on $r$.
He actually proved a stronger version of them, which allowed him to prove H\"older regularity of the Martin kernels.
More precisely, he proved that whenever two geodesics $[x,y]$ and $[x',y']$ fellow travel for a time at least $n$, then for all $r\leq R$, we have
\begin{equation}\label{strongAnconahyperbolic}
    \left | 1-\frac{G_r(x,y)G_r(x',y')}{G_r(x',y)G_r(x,y')}\right |\leq C \rho^n
\end{equation}
for some constants $C\geq 0$ and $0<\rho<1$ independent of $r$.
We will call~(\ref{weakAnconahyperbolic'}) weak Ancona inequalities and~(\ref{strongAnconahyperbolic}) strong Ancona inequalities up to the spectral radius in what follows.

\medskip
For relatively hyperbolic groups, such inequalities are not expected to hold in general, for if $y$ is deep in a parabolic subgroup, there is no reason that the random walk passes within a bounded distance of $y$.

However, we can prove analogous inequalities if we restrict ourselves to so-called transition points. Say a point $y$ on a geodesic $\alpha=[x,z]$ is an $(\epsilon,\eta)$-transition point of $\alpha$ if the length-$\eta$ interval of $\alpha$ around $y$ is not contained in the $\epsilon$-neighborhood of a single parabolic subgroup.
We show that the inequalities~(\ref{weakAnconahyperbolic'}) hold whenever $y$ is an  $(\epsilon,\eta)$-transition point on a geodesic connecting $x,z$.
Similarly, we prove~(\ref{strongAnconahyperbolic}) holds whenever there exist $n$ distinct $(\epsilon,\eta)$-transition points $z_i\in [x,y]$, $z'_i\in [x',y']$ with $d(z_i,z'_i)<D$.
In both cases the implied constants are independent of $r\leq R$.


In fact, we prove a more general version of these inequalities, using Floyd functions. Fix a function $f(n)=\lambda^n$, $0<\lambda<1$.
The Floyd distance viewed from $y\in \Gamma$, which we denote by $\delta^{f}_{y}(.,.)$ is the path metric on $\Gamma$ obtained from rescaling the Cayley graph by declaring the length of an edge $\tau$ to be $f((d(\tau,y))$. 
The connection with transition points on word geodesics in relatively hyperbolic groups is as follows: if $\Gamma$ is relatively hyperbolic, and $f$ a suitable exponentially decaying function, then for each $\epsilon,\eta,D>0$ there is a $\delta>0$ such that whenever $x,y,z\in \Gamma$ and $y$ is $D$ close to an $(\epsilon,\eta)$-transition point of $[x,z]$ we have $\delta^{f}_{y}(x,z)>\delta$. 
In hyperbolic groups, all points on a geodesic are transition points and this condition on the Floyd distance is equivalent to saying that $y$ is within a bounded distance of $[x,z]$.

We say that a probability measure $\mu$ on a finitely generated group $\Gamma$ satisfies the weak Floyd-Ancona inequalities up to the spectral radius if for every $\delta>0$, and every $x,y,z\in \Gamma$ with $\delta^{f}_{y}(x,z)>\delta$ the inequality~(\ref{weakAnconahyperbolic'}) holds.
We say that $\mu$ satisfies the strong Floyd-Ancona inequalities up to the spectral radius if the inequality~(\ref{strongAnconahyperbolic}) holds whenever
there is a sequence of $n$ distinct points $z_1,...,z_n$ such that for every $i$, $\delta^{f}_{z_i}(x,y)>\delta$ and  $\delta^{f}_{z_i}(x',y')>\delta$.
In both cases, the implied constants are required to depend only on $\delta>0$ but not on $r\leq R$.
Adapting the techniques of Gou\"ezel in \cite{Gouezel-local}, we prove the following.

\begin{theorem}\label{maintheoremAncona}[Theorem~\ref{thmweakAncona}, Theorem~\ref{thmstrongAncona}]
Let $\mu$ be a symmetric measure on a finitely generated group $\Gamma$, whose finite support generates $\Gamma$.
Then, $\mu$ satisfies both weak and strong Floyd-Ancona inequalities up to the spectral radius with respect to any exponential Floyd function.
\end{theorem}

We prove that these inequalities hold for admissible symmetric finitely supported measures on any finitely generated groups. 
However, the statement will only be non-vacuous if the Floyd boundary, which is the complement of the group $\Gamma$ inside its completion for the Floyd distance, is infinite.
The only known examples of finitely generated groups with infinite Floyd boundary are nonelementary relatively hyperbolic groups.

For $r=1$, the weak Floyd-Ancona inequalities were proved in \cite{GGPY}. The strong Floyd-Ancona inequalities are new even for $r=1$.

The results concerning the identification of the Martin boundary in this paper use only the weak Floyd-Ancona inequalities at the spectral radius. However, the strong inequalities are important for other reasons. In the setting of hyperbolic groups, these strong versions of Ancona inequalities for $r=1$ were 
used by many authors including \cite{INO}, \cite{GMM}, where they are the key ingredient to prove 
H\"older continuity of the Martin kernel on the Gromov boundary. The uniformity of these inequalities up to the spectral radius was established in \cite{Gouezel-local}, \cite{GouezelLalley} and used to prove a local limit theorem for random walks on hyperbolic groups.
The Floyd-Ancona inequalities proved here are used by the first author in \cite{DussauleLLT2} to establish a weaker H\"older-type property and use it to prove a local limit theorem for some classes of relatively hyperbolic groups.



\subsection{Organization of the paper}\label{Sectionorganization}
In Section~\ref{Sectionspectraldegeneracy}, we first introduce relatively hyperbolic groups.
We then properly define the notion of spectral degeneracy for a probability measure on such a group.
We also define the geometric boundaries we will identify the Martin boundaries with and propose a construction using arguments of Dahmani \cite{Dahmani}.

\medskip

In Section~\ref{Sectiondeviationinequalities}, we prove deviation inequalities.
In particular, we prove Theorem~\ref{maintheoremAncona}, generalizing Gou\"ezel approach in \cite{Gouezel-local}.
We first prove that the probability of going from $x$ to $z$, missing a ball centered on a point $y$ satisfying $\delta_y^f(x,z)\geq \delta$, decays super-exponentially fast in the radius of the ball.
We can then conclude that weak Floyd-Ancona inequalities hold, as in \cite{GGPY}.
To prove that strong Floyd-Ancona inequalities also hold, we have to be more precise and we first prove a refined version of weak inequalities, see Proposition~\ref{strengthenedweakAncona}.

\medskip

Section~\ref{SectiondescriptionMartin} is devoted to the proof of our main theorem.
We first prove that if a sequence converges to a point in the geometric boundary, then it converges to a point in the Martin boundary.
If the limit point in the boundary is a conical limit point, the result is given by \cite{GGPY}, once weak Floyd-Ancona inequalities are established.
We thus only have to prove the result when the limit is a parabolic limit point.
We deal separately with the degenerate and non-degenerate case.

Essentially, the strategy goes as follows.
When the parabolic limit point is spectrally degenerate, we prove that the Martin boundary of the induced chain on the corresponding parabolic subgroup is reduced to a point.
We then have to prove that the Martin boundary of this induced chain embeds into the Martin boundary of the whole group.
This is a difficult task in general, but this follows here from the fact that the Martin boundary of the induced chain on any fixed neighborhood of the parabolic subgroup is also reduced to a point.
We can actually prove this when the parabolic subgroups are virtually nilpotent and we do not need to assume there that they are virtually abelian.

When the parabolic limit point is not spectrally degenerate, we use the results of \cite{Dussaule} to prove that the Martin boundary of the induced chain on the corresponding parabolic subgroup is a Euclidean sphere at infinity.
We again have to prove that this Martin boundary embeds in the whole Martin boundary.
This was proved in \cite{DGGP} for the 1-Martin boundary and the same strategy works here.
However, some crucial technical details have to be changed so we rewrite the entire proof.

\medskip

In Section~\ref{Sectionstability}, we prove Theorem~\ref{mainthmstability} and Theorem~\ref{minimality}.
The proof of minimality is similar to the proof of minimality for the 1-Martin boundary in \cite{GGPY}.

To prove continuity of $(r,x,\xi)\mapsto K_r(x,\xi)$, we again deal separately with conical limit points and parabolic limit points $\xi$.
In the first case, continuity is a direct consequence of strong Ancona inequalities.
Our proof also shows that the map $(r,x,\xi)\mapsto K_r(x,\xi)$ is continuous for every point in the Martin boundary, when $\Gamma$ is hyperbolic.
To the authors' knowledge, this is not stated anywhere, although this is implicit in \cite{Gouezel-local}, where there is a description of the $r$-Martin boundary for every $r\leq R$.

For parabolic limit points, we need to prove a form of continuity in $r$ of the first return kernel to a parabolic limit group associated with $r\mu$.
This uses in turn results of Section~\ref{SectiondescriptionMartin}.
We conclude the proof of Theorem~\ref{mainthmstability} using the results of the first author in \cite{Dussaule}.
\medskip

Finally, in Section~\ref{Sectionspectrallydegeneratemeasures}, we prove that the Martin boundary is strongly stable in small dimension.
Precisely, we prove that if $H$ is a parabolic subgroup which is virtually abelian of rank $d$ and if $\mu$ is spectrally degenerate along $H$, then $d\geq 5$.
Roughly speaking, we prove this in two steps.
We first show that the derivative at 1 of the Green function along a parabolic subgroup is always finite.
We then show that if $\mu$ is spectrally degenerate along $H$, then the induced transition kernel $p$ on $H$ satisfies a local limit theorem of the form $p^{(n)}(e)\sim C n^{-d/2}$, where $p^{(n)}$ is the $n$th convolution power of $p$.
This is a version of the classical local limit theorem in $\Z^d$ that we prove using results in \cite{Dussaule}, see also \cite{Woess-book}.
Since the Green function along a parabolic subgroup is given by $t\mapsto \sum_{n\geq 0}t^np^{(n)}(e)$, the first derivative is finite if and only if $n\geq 5$.

\section{Relatively hyperbolic groups and boundaries}\label{Sectionspectraldegeneracy}
\subsection{Some background on the Floyd distance}\label{Sectionrelativelyhyperbolicgroups}
Let $\Gamma$ be a finitely generated group. The action of $\Gamma$ on a compact Hausdorff space $T$ is called a convergence action if the induced action on triples of distinct points of $T$ is properly discontinuous. 
Suppose $\Gamma\curvearrowright T$ is a convergence action. The set of accumulation points  $\Lambda \Gamma$ of any orbit $\Gamma \cdot x\ (x\in T)$ is called the {\it limit set} of the action. As long as $\Lambda \Gamma$ has more than two points, it is uncountable and is the unique minimal closed $\Gamma$-invariant subset of $T$.
The action is then said to be non-elementary. In this case, the orbit of every point in $\Lambda \Gamma$ is infinite. 
The action is {\it minimal} if $\Lambda \Gamma=T$.

A point $\zeta\in\Lambda \Gamma$  is called {\it conical} if there exists a sequence $g_{n}$ of $\Gamma$ and distinct points $\alpha,\beta \in \Lambda \Gamma$ such that
$g_{n}\zeta \to \alpha$ and $g_{n}\eta \to \beta$ for all $\eta \in  T\setminus\{ \zeta\}$.
The point $\zeta\in\Lambda \Gamma$ is called bounded parabolic if it is the unique fixed point of its stabilizer in $\Gamma$, which is infinite and acts co-compactly on $\Lambda \Gamma \setminus \{\zeta\}$.
The stabilizers of bounded parabolic points are called maximal parabolic subgroups.
The convergence action $\Gamma \curvearrowright  T$ is called geometrically finite if every point of $\Lambda \Gamma \subset T$ is either conical or bounded parabolic.
Since $\Gamma$ is assumed to be finitely generated, every maximal parabolic subgroup is finitely generated too (see \cite[Main~Theorem~(d)]{Gerasimov2}).
Then, by Yaman's results \cite{Yaman}, it follows that for every  minimal geometrically finite action $\Gamma \curvearrowright T$ there exists a proper geodesic Gromov hyperbolic space $X$ on which $\Gamma$ acts properly discontinuously by isometries and a $\Gamma$-equivariant homeomorphism $T \to \partial X$.

Suppose now $\Omega$ is a collection of subgroups of $\Gamma$. We say that $\Gamma$ is \textit{hyperbolic relative to $\Omega$} if there exists some compactum $T$ on which $\Gamma$ acts minimally and geometrically finitely and such that the maximal parabolic subgroups are the elements of $\Omega$.
Such a compactum is then unique up to $\Gamma$-equivariant homeomorphism \cite{Bowditch} and is called the Bowditch boundary of $(\Gamma, \Omega)$. 
The group $\Gamma$ is said to be \textit{non-elementary relatively hyperbolic} if it admits a non-elementary geometrically finite convergence action on some infinite compactum.

\medskip

Let $\Gamma$ be a finitely generated group and let $S$ be a finite generating set.
Consider then a function $f:\mathbb{R}^{+}\to \mathbb{R}^{+}$ satisfying the two following conditions: $\sum_{n\geq 0}f(n)$ is finite and there exists a $\lambda\in (0,1)$ such that $1\geq f(n+1)/f(n)\geq \lambda$ for all $n\in\mathbb{N}$. The function $f$ is called a {\it rescaling function} or a \textit{Floyd-function}.
When $f$ is of the form $f(n)=\lambda^{n}$, we say that $f$ is an exponential Floyd function.

Pick a basepoint $o\in \Gamma$ and rescale the Cayley graph $\Cay(\Gamma,S)$ by declaring the length of an edge $\tau$ to be $f(d(o,\tau))$. The induced short-path metric on $\Cay(\Gamma,S)$  is called the {\it Floyd metric} with respect to the basepoint $o$ and Floyd function $f$ and denoted by $\delta^{f}_{o}(.,.)$.
Its Cauchy completion, whose topology does not depend on the basepoint, is called the Floyd compactification $\overline{\Gamma}_{f}$ and $\partial_{f}\Gamma= \overline{\Gamma}_{f} \setminus \Gamma$ is called the Floyd boundary.
Karlsson \cite{Karlsson} proved that the Floyd boundary is either uncountably infinite or reduced to 0, 1 or 2 points.

If $\Gamma$ is non-elementary relatively hyperbolic and if the Floyd function $f$ is not decreasing exponentially too fast, Gerasimov proved that the Floyd boundary is infinite.
One can take $f$ to be of the form $f(n)=\lambda^{n}$ for suitable $\lambda$.
Moreover, there is a continuous $\Gamma$-equivariant surjection ({\it Floyd map}) from the Floyd boundary to the Bowditch boundary \cite[Map theorem]{Gerasimov}.
Furthermore,  Gerasimov and Potyagailo \cite[Theorem~A]{GePoJEMS} proved that the pre-image  of any conical point by this map is a singleton and the pre-image of a parabolic fixed point $p$ is the limit set for the action of its stabilizer $\Gamma_p$ on  $\partial_f\Gamma$.
In particular if $\Gamma_p$ is an amenable non-virtually cyclic group then its limit set on the Floyd boundary is a point. Consequently, when $\Gamma$ is hyperbolic relative to a collection of infinite amenable subgroups none of which are virtually cyclic, the Floyd boundary is homeomorphic to the Bowditch boundary.
This is in particular the case if the parabolic subgroups are virtually abelian.

We will mainly be interested in relatively hyperbolic groups in this paper.
However, the deviation inequalities that we will prove in Section~\ref{Sectiondeviationinequalities} will be formulated with the Floyd distance and will hold in any finitely generated group.

\medskip
To conclude this section, we recall the precise definition of transition points.

\begin{definition}\label{deftransitionusual}
If $\alpha$ is a (finite or infinite) geodesic in $\Cay(\Gamma,S)$ 
, a point $p\in \alpha$ is said to be {\it $(\epsilon,\eta)$-deep} if there is a $g \in \Gamma$, $P\in \Omega$ such that the part of $\alpha$ containing the points at distance at most $\eta$ from $p$
is contained in the $\epsilon$-neighborhood of $gP$. Otherwise, $p\in \alpha$ is called an \textit{$(\epsilon,\eta)$-transition point} of $\alpha$.
\end{definition}

The following result relates transition points to the Floyd metric.

\begin{proposition}\label{Floydgeo}\cite[Corollary 5.10]{GePoCrelle}
For every $\epsilon>0$, $\eta>0$ and $D>0$ there exists $\delta>0$ such that
if $y$ is within word distance $D$ of an $(\epsilon,\eta)$-transition point of a word geodesic from $x$ to $z$ then $\delta^{f}_{y}(x,z)>\delta$.
\end{proposition}

\subsection{Spectral degeneracy and geometric boundaries}
Let $\Gamma$ be a relatively hyperbolic group with respect to $\Omega$ and let $\mu$ be a probability measure on $\Gamma$.
We now define properly the notion of spectral degeneracy along parabolic subgroups, and then proceed to define the geometric boundaries associated with $\mu$.

We choose a set $\Omega_0$ of representatives of conjugacy classes of elements of $\Omega$.
Such a set $\Omega_0$ is finite according to \cite[Proposition~6.15]{Bowditch}.
Let $H\in \Omega_0$.
We denote by $p_H$ the first return kernel to $H$.
If $h,h'\in H$, then
$p_H(h,h')$ is the probability that the $\mu$-random walk, starting at $h$, eventually comes back to $H$ and that its first return to $H$ is at $h'$.
In other words,
$$p_H(h,h')=\mathbb{P}_h(\exists n\geq 1, X_n=h',X_1,...,X_{n-1}\notin H).$$
More generally, for $r\in [0,R]$, we denote by $p_{H,r}$ the first return transition kernel to $H$ for $r\mu$.
Precisely, if $h,h'\in H$, then
$$p_{H,r}(h,h')=\sum_{n\geq 0}\sum_{\underset{\notin H}{g_1,...,g_{n-1}}}r^n\mu(h^{-1}g_1)\mu(g_1^{-1}g_2)...\mu(g_{n-2}^{-1}g_{n-1})\mu(g_{n-1}^{-1}h').$$

We then denote by $p_{H,r}^{(n)}$ the convolution powers of this transition kernel, by $G_{H,r}(h,h'|t)$ the associated Green function, evaluated at $t$ and by $R_H(r)$ the associated spectral radius, that is, the radius of convergence of $t\mapsto G_{H,r}(h,h'|t)$.
For simplicity, write $R_H=R_H(R)$.

Since $\Gamma$ is assumed to be non-elementary, it contains a free group and hence is non-amenable.
As explained in the introduction, $G_{R}(e,e)$ is finite, so that
$G_{k,R}(1)<+\infty$.
In particular, for every $r\leq R$, $R_H(r)\geq 1$.

\begin{definition}
We say that the random walk is \textit{spectrally degenerate along $H$} if $R_H(R)=1$.
In that case we also say that the measure $\mu$ is spectrally degenerate along $H$.
We say that the random walk (or the measure $\mu$) is spectrally non-degenerate if it is not spectrally degenerate along $H$.
\end{definition}

The following lemma implies that this definition does not depend on the choice of $\Omega_0$.

\begin{lemma}\label{lemmaspectraldegeneracyindependentOmega0}
Let $\Omega_0$ and $\Omega_1$ be two finite set of representatives of conjugacy classes of elements of $\Omega$.
Let $H_0\in \Omega_0$ and $H_1\in \Omega_1$ be such that $H_0=gH_1g^{-1}$ for some $g$.
Then, $\mu$ is spectrally degenerate along $H_0$ if and only if it is spectrally degenerate along $H_1$.
\end{lemma}

To avoid a lengthy argument, we do not prove this lemma a priori and we will not use it in the following.
It will be a consequence of our main result anyway, since the $R$-Martin boundary does not depend on the choice of $\Omega_0$.

\begin{definition}
Let $\xi$ be a parabolic limit point in the Bowditch boundary.
We say that $\xi$ is spectrally degenerate (with respect to $\mu$) if $\mu$ is spectrally degenerate along $H$,
where $H\in \Omega_0$ is a conjugate of the stabilizer of $\xi$.
\end{definition}

We now associate with $\mu$ two geometric boundaries.
The first one does not depend on $\mu$.
It consists of the Bowditch boundary blown-up at every parabolic point.
The second one only depends on the set of parabolic subgroups along which $\mu$ is spectrally degenerate, and is obtained from the Bowditch boundary by blowing up only the degenerate parabolic limit points.

We assume  that the parabolic subgroups are virtually abelian.
Let $H$ be such a parabolic subgroup.
Then, there exists a finite index subgroup of $H$ which is isomorphic to $\Z^d$.
Any section $H/\Z^d\to H$ gives an identification between $H$ and $\Z^d\times E$, where $E$ is a finite set.
An element $h\in H$ will be denoted by $(x,k)$, where $x\in \Z^d$ and $k\in E$.
We define the geometric boundary $\partial H$ of $H$ as follows.
A sequence $h_n=(x_n,k_n)$ of elements of $H$ converges to a point in $\partial H$ if and only if $x_n$ tends to infinity and
$\frac{x_n}{\|x_n\|}$ converges to some $\theta \in \mathbb{S}^{d-1}$, where $\|x\|$ is the Euclidean norm of $x\in \Z^d$.
In other words, we glue a sphere at infinity on $\Z^d\times E$.
Notice that this definition is independent of the coordinate $k_n$.
Also notice that $H$ acts co-compactly on a CAT(0) space because it is virtually abelian, see e.g.\ \cite[Remark~7.3~(2)]{BridsonHaefliger}.
The boundary $\partial H$ coincides with the CAT(0) boundary of $H$, i.e. the visual boundary of this CAT(0) space, see \cite[Section~3.2]{DGGP} for more details.

Fix a set of representatives $\Omega_0$ of conjugacy classes of parabolic subgroups.
Following Dahmani \cite{Dahmani}, given a boundary for every parabolic subgroup $H\in\Omega_0$, one can construct a boundary for $\Gamma$.
This boundary basically consists of gluing the boundaries of parabolic subgroups to the set of conical limit points, see \cite[Theorem~3.1]{Dahmani}.
We also give more details in \cite[Section~7]{DGGP} and give a formulation in terms of projections on parabolic subgroups.
We now use this construction of Dahmani.
We need to choose a boundary for every parabolic subgroup in $\Omega_0$.
We choose either the boundary $\partial H$ defined above, or the boundary reduced to one point.
More precisely, we define a first geometric boundary, choosing $\partial H$ for every parabolic subgroups.
We then define a second geometric boundary associated with a special class of parabolic subgroups $\Omega_1\subset \Omega_0$, choosing the one point-boundary for parabolic subgroups in $\Omega_1$ and $\partial H$ for parabolic subgroups $H\in \Omega_0\setminus \Omega_1$.

\begin{definition}[First geometric boundary]\label{deffirstgeometricboundary}
Let $\Gamma$ be a relatively hyperbolic group with respect to a collection of virtually abelian subgroups $\Omega$.
We define a first geometric compactification to be any metrizable compact space $X$ such that $\Gamma$ is a dense and open subset of $X$ and such that the following holds.
A sequence $g_n$ of elements of $\Gamma$ converges to some point $\xi\in X\setminus \Gamma$ if and only if
\begin{enumerate}
    \item either $g_n$ converges to a conical limit point,
    \item or there exists a parabolic subgroup $H\in \Omega_0$ and there exists $g\in \Gamma$ such that the projection $\pi_{gH}(g_n)$ of $g_n$ on $gH$ satisfies that $g^{-1}\pi_{gH}(g_n)$ converges to a point in the geometric boundary $\partial H$ of $H$.
\end{enumerate}
We call the complement $X\setminus \Gamma$ of $\Gamma$ in $X$ the first geometric boundary.
\end{definition}

\begin{definition}[Second geometric boundary]\label{defsecondgeometricboundary}
Let $\Gamma$ be a relatively hyperbolic group with respect to a collection of virtually abelian subgroups $\Omega$.
Choose a finite set of representatives of conjugacy classes $\Omega_0\subset \Omega$ and let $\Omega_1\subset \Omega_0$.
We define a second geometric compactification (associated with $\Omega_1$) to be any metrizable compact space $X$ such that $\Gamma$ is a dense and open subset of $X$ and such that the following holds.
A sequence $g_n$ of elements of $\Gamma$ converges to some point $\xi\in X\setminus \Gamma$ if and only if
\begin{enumerate}
    \item either $g_n$ converges to a conical limit point,
    \item or $g_n$ converges to a parabolic limit point whose stabilizer is conjugate to a parabolic subgroup in $\Omega_1$,
    \item or there exists a parabolic subgroup $H\in \Omega_0\setminus \Omega_1$ and there exists $g\in \Gamma$ such that the projection $\pi_{gH}(g_n)$ of $g_n$ on $gH$ satisfies that $g^{-1}\pi_{gH}(g_n)$ converges to a point in the geometric boundary of $H$.
\end{enumerate}
We call the complement $X\setminus \Gamma$ of $\Gamma$ in $X$ the second geometric boundary.
\end{definition}

Note that since $\Gamma$ is asked to be dense in the compactifications, their topology is uniquely defined by the given characterization of converging sequences.
To simplify notations in the following, we will use the following terminology.

\begin{definition}
Let $\Gamma$ be a relatively hyperbolic group with respect to a collection of virtually abelian subgroups $\Omega$.
Let $\mu$ be a probability measure on $\Gamma$ whose support generates $\Gamma$ as a semi-group.
Let $R$ be its spectral radius and let $r<R$.
Denote by $\Omega_{\mu}\subset \Omega_0$ the set of spectrally degenerate parabolic subgroups.
For $r<R$, we define the $r$-geometric compactification as the first geometric compactification.
We also define the $R$-geometric compactification as the second geometric compactification, with $\Omega_1=\Omega_{\mu}$.
\end{definition}

We will denote by $\partial_r\Gamma$ the $r$-geometric boundary, for $r\leq R$.
It follows from the definitions that the identity of $\Gamma$ extends to a continuous equivariant surjective map $\phi_{\mu}$ from the $r$-geometric compactification to the $R$-geometric compactification.
The induced map on the boundaries is obtained collapsing the geometric boundaries of the spectrally degenerate parabolic subgroups to one point.
Also, for $r<R$, the $r$-geometric boundary coincides with the PBU-boundary defined in \cite{DGGP}.

\section{Deviation inequalities}\label{Sectiondeviationinequalities}
The goal of this section is to prove weak and strong Floyd-Ancona inequalities.
\subsection{Bounds for sums of Green functions over spheres}
We first prove the following proposition.
We consider a finitely generated group.
We fix a generating set and we write $d$ the associated word distance.
Also, we write $S_k$ for the sphere of radius $k$ with respect to the distance $d$.

\begin{proposition}\label{doubleGreenfiniteonspheres}
Let $\Gamma$ be a group with nontrivial Floyd boundary and $\mu$ a symmetric probability measure on $\Gamma$ whose finite support generates $\Gamma$.
Let $R$ be its spectral radius.
The quantity
$$\sum_{x\in \Gamma,d(e,x)=k}G^{2}_{r}(e,x)$$ is bounded independently of $k$ and of $r\leq R$.
\end{proposition}

We need some preliminary results.
We refer to \cite{YangSCC} for definitions and background for statistically convex co-compact group actions and contracting elements. For our purposes, we will only use that such actions  $\Gamma \curvearrowright (X,d)$  include the actions of groups with infinite Floyd boundary on their Cayley graphs \cite[Lemma 7.2]{Yanggrowthtight}.
The following result is due to Yang \cite{YangSCC}, who calls it the "Extension Lemma".
\begin{lemma} \label{extensionlemma}\cite[Lemma 2.14]{YangSCC}
Let $\Gamma \curvearrowright X$ be a proper co-compact action with a contracting element with respect to a metric $d$ on $X$. Let $o\in X$ be a basepoint.
Then there is a $C>0$ with the following property. For any $g, h \in \Gamma$ there is a contracting element $w\in G$ with $d(o,w\cdot o)<C$  such that 
$$|d(o,gwh\cdot o)-d(o,g\cdot o)-d(o,h\cdot o)|<C.$$
Moreover, every geodesic from $o$ to $gwh\cdot o$ passes within $C$ of $g\cdot o$ and $gw\cdot o$.
\end{lemma}

As explained, Yang proved that if $\Gamma$ is a group with nontrivial Floyd boundary, then its action on any Cayley graph has a contracting element, so we can reformulate the Extension Lemma as follows in our situation.

\begin{lemma}\label{extensionlemmaFloyd}
Let $\Gamma$ be a group with nontrivial Floyd boundary.
Then, there exists $C\geq 0$ and $\delta>0$ such that for any $g,h\in \Gamma$, there exists $w\in \Gamma$ with $d(e,w)\leq C$ such that
$\delta_g(e,gwh)\geq \delta$ and $\delta_g(e,gw)\geq \delta$
and
$|d(e,gwh)-d(e,g)-d(e,h)|\leq C$.
Moreover, every geodesic from $e$ to $gwh$ in the Cayley graph of $\Gamma$ passes within $C$ of $g$ and $gw$.
\end{lemma}

We also have the following lemma.

\begin{lemma}\label{boundedtoone}
For every $K_1,K_2>0$, there is a constant $N>0$ such that for all $k,l\in \mathbb{N}$, and $z\in \Gamma$ with $k+l-K_1\leq |z|\leq k+l+K_1$ there are at most $N$ triples $(x,a,y)$ with $d(e,x)=k$, $d(e,y)=l$, $d(e,a)\leq K_1$ such that $xay=z$ and $x$ is at distance at most $K_2$ from a geodesic from $e$ to $z$.
\end{lemma}

\begin{proof}
Consider such $x,y,z,a\in \Gamma$ with $xay=z$. Let $\alpha$ be a geodesic from $e$ to $z$ such that $x$ is within $K_2$ of $\alpha$.
Then, $x$ is within $2K_2$ of $\alpha(d(e,x))$, so there is a uniformly finite number of possibilities for such $x$.
Since $d(e,a)\leq K_1$, there also is a finite number of possibilities for such $a$.
This proves that there also is a finite number of possibilities for such $y$, which concludes the proof.
\end{proof}

\begin{proof}[Proof of Proposition \ref{doubleGreenfiniteonspheres}]
We follow the proof of \cite[Lemma 2.5]{Gouezel-local}. First fix $r<R$.
Write
$$u_{k}(r)=\sum_{x\in \Gamma,d(e,x)=k}G^{2}_{r}(e,x).$$
Using Lemma~\ref{extensionlemmaFloyd}, to any $x,y$ with $d(e,x)=k$, $d(e,y)=l$ we can associate a $\Psi(x,y)=xay$ with $l+k-C\leq d(e,\Psi(x,y))<l+k+C.$ 
We have 
\begin{align*}
    G^{2}_{r}(e, x)G^{2}_{r}(e, y) &\lesssim G^{2}_{r}(e, x)G^{2}_{r}(e, a)G^{2}_{r}(e, y) \\
    & = G^{2}_{r}(e, x)G^{2}_{r}(x, xa)G^{2}_{r}(xa, xay)
\end{align*}
and so \cite[(2.4)]{Gouezel-local} yields
$$ G^{2}_{r}(e, x)G^{2}_{r}(e, y)\lesssim G^{2}_{r}(e, xay).$$

According to Lemma~\ref{boundedtoone}, any $g\in \Gamma$ has a finite number of preimage by $\Psi$, so that
\begin{align*}
    u_{k}(r)u_{l}(r)=\sum_{x\in S_{k},y\in S_{l}}&G^{2}(e,x)G^{2}_{r}(e,y)\lesssim \sum_{x\in S_{k},y\in S_{l}} G^{2}_{r}(e,\Phi(x,y))\\
    & \lesssim \sum^{k+l+C}_{i=k+l-C}\sum_{z\in S_i}G^{2}_{r}(e,z)\lesssim \sum^{k+l+C}_{i=k+l-C} u_{i}(r).
\end{align*}

According to \cite[Proposition~1.9]{GouezelLalley}, $\sum_{g\in \Gamma} G^{2}_r(e,g)$ is finite if and only if the derivative $G'_r(e,e)$ is finite.
Since $r<R$, this derivative is finite.
In particular, the sequence $u_{k}(r)$ is summable, and reaches its maximum $M(r)$ at some index $k_{0}(r)$.
Using the previous equation with $k=l=k_{0}(r)$ we get  $M(r)^{2}\lesssim M(r)$ and thus $M(r)\lesssim 1$.
We have thus proved that for $r<R$,
$u_{k}(r)$ is bounded independently of $k$ and $r$.
Since $r\mapsto G_r(e,e)$ is non-decreasing, we can take the limit as $r$ tends to $R$, which completes the proof.
\end{proof}

\subsection{Super-exponential decay of the probability of missing balls}\label{SectionweakAncona}
We consider a finitely generated group $\Gamma$ and a symmetric admissible finitely supported probability measure $\mu$ on $\Gamma$.
Let $f(n)=a^{-n}$ be an exponentially decaying Floyd function for $\Gamma$.
Let $G=G_{R}$ denote the Green function at the spectral radius $R$.
To simplify notations, also write $d(e,x)=|x|$ for $x\in\Gamma$.
We prove here the following.
Define the Green function restricted to paths that stay in some fixed subset $A\subset \Gamma$ by
$$G(x,y;A)=\sum_{n\geq 0}\sum_{\underset{z_1,...,z_{n-1}\in A}{z_0=x,z_1,...,z_n=y}}\mu(z_0^{-1}z_1)\mu(z_1^{-1}z_2)...\mu(z_{n-1}^{-1}z_n).$$
Notice that we do not require $x$ and $y$ to be in $A$.
In particular, $G(x,y;A^c)$ is the Green function restricted to paths from $x$ to $y$ avoiding $A$.

\begin{proposition}\label{propsuperexponentialdecay}
With these notations,
there exists $c>1$ and $\eta_0\geq 0$ such that the following holds.
For every $\delta>0$, for every $x,y,z\in \Gamma$ with $\delta^{f}_{z}(x,y)>\delta$ we have for every $\eta\geq \eta_0$
$$G(x,y;B_{\eta}(z)^{c})\leq \mathrm{e}^{-\delta c^\eta}.$$
\end{proposition}

\begin{proof}
By equivariance of the Green function we can assume without loss of generality that $z=e$.
If the Floyd metric is a point the result trivially holds because for each $\delta>0$ there are only finitely many $x\in \Gamma$ with $\delta^{f}_{z}(x,y)>\delta$ (see the proof of \cite[Theorem 4.1]{GGPY}. Therefore we assume the Floyd boundary is infinite.

We fix $\eta>0$.
We consider some $\epsilon>0$ which will be chosen small enough later, independently of $\eta$.
Let $N$ be the smallest integer less than $\delta a^{\epsilon \eta}/2$.
We will find subsets $A_{0},...,A_{N+1}\subset \Gamma$ with $A_0=\{x\}$ and
$A_{N+1}=\{y\}$ satisfying the following.

\begin{enumerate}[a)]
    \item Any admissible path in $\Gamma$ from $x$ to $y$ which does not intersect $B_{\eta}(y)$ passes successively through $A_1,...,A_{N+1}$.
    \item For each $i=0,..,N$, we have
\begin{equation}\label{prooffirstsuperexponentialdecay1}
    \sum_{a\in A_{i}, b\in A_{i+1}}G(a,b)^2<1/4.
\end{equation}
\end{enumerate}
The result will then follow as in \cite[Lemma 2.6]{Gouezel-local}, as we now explain.
Following \cite{Series} and \cite{Ledrappier}, we introduce the operator
$$L_i:\ell^2(A_{i+1})\rightarrow \ell^2(A_i)$$
defined by
$$L_if(a)=\sum_{b\in A_{i+1}}G(a,b)f(b).$$
First, Condition~a) above shows that
$$G(x,y;B_{\eta}(z)^{c})\leq\sum_{a_1\in A_1}...\sum_{a_N\in A_N}G(x,a_1)G(a_1,a_2)...G(a_N,y).$$
This can be written as
$$G(x,y;B_{\eta}(z)^{c})\leq \left (L_0...L_N\delta_y\right )(x),$$
so that $G(x,y;B_{\eta}(z)^{c})$ is bounded by $\prod \|L_i\|$.
Then, Cauchy-Schwarz shows that
$$\|L_i\|\leq \left (\sum_{a\in A_{i}, b\in A_{i+1}}G(a,b)^2\right )^{1/2}.$$
Thus, Condition~b) shows that $\|L_i\|\leq 1/2$, so that
$$G(x,y;B_{\eta}(z)^{c})\leq 2^{-N}\leq 2^{-\delta \frac{1}{2}a^{\epsilon\eta}+1}.$$
This will yield the desired inequality, so that we only have to construct the sets $A_1,...,A_N$ to conclude the proof.

Let $K=\max_{g\in \supp \mu}|g|$ be the maximal jump size of the random walk
and let $L$ be such that for all $g\in \Gamma$,
$$Kf(|g|-K)\leq Lf(|g|).$$
For $\theta \in \mathbb{R}$ let 
$$A(\theta)=\left \{g\in \Gamma: |g|>\eta, \delta^f_{e}(g,x)\in \left [\theta - La^{-|g|}, \theta+La^{-|g|}\right ]\right \}.$$
For $i=1,...,N$, we consider the interval
$$I_{i}=[(2i-1)a^{-\epsilon \eta},2i a^{-\epsilon \eta}].$$
We will find $\theta_i \in I_i$, $i=1,...,N$ so that $A_{i}=A(\theta_i)$ satisfy the conditions above.

We first shows that for any choice of $\theta_1,...,\theta_N$, Condition~a) is satisfied.
Consider a path for the random walk from $x$ to $y$.
Since $\theta_N \in I_N$,
in particular we have $0\leq \theta_N \leq \delta^{f}_{e}(x,y)$.
Let $g_{j_{N}}$ be the last point on the path with $\delta^{f}_{e}(x,g_{j_{N}})\leq \theta_N$.
By definition, $\delta^{f}_{e}(x,g_{j_{N}+1})>\theta_N$.
On the other hand,
$$\delta^{f}_{e}(x,g_{j_{N}+1})\leq  \delta^{f}_{e}(x,g_{j_{N}})+\delta^{f}_{e}(g_{j_{N}},g_{j_{N}+1})\leq \theta_N+Ka^{-|g_{j_{N}+1}|+K}\leq \theta_N +La^{-|g_{j_{N}+1}|}$$
so that $g_{j_{N}+1}\in A(\theta_N)$.
This proves that the path passes through $A_N$.

Now we claim that if $\eta$ is large enough and $\epsilon$ is small enough, independently of $\eta$, then if the path passes through $A_{i+1}$, it first passes through $A_{i}$.
This will settle the first condition.
We thus consider $x_i\in A_{i+1}$ such that the path passes through $x_i$.
Then, by definition
$$\delta^f_{e}(x,x_i)\geq \theta_{i+1}-La^{-|x_i|}\geq (2i+1)a^{-\epsilon \eta}-La^{-|x_i|}.$$
Since $x_i\in A_{i+1}$, in particular, $|x_i|>\eta$.
If $\eta$ is large enough, then $a^{\eta/2}\geq L$, so
if $\epsilon<1/2$, then
$$a^{-\epsilon \eta}\geq La^{-\eta}>La^{-|x_i|}.$$
In particular,
$$\delta^f_{e}(x,x_i)> 2ia^{-\epsilon \eta}\geq \theta_i.$$
Consider the sub-path of the random walk from $x$ to $x_i$.
Let $g_{j_{i}}$ be the last point on this sub-path with $\delta^{f}_{e}(x,g_{j_{i}})\leq \theta_i$.
The same proof as above shows that $g_{j_i+1}\in A_i$.
This proves the claim.

We now find $\theta_1,...,\theta_N$, $\theta_i\in I_1$ satisfying Condition~b). 
Let $\Leb$ denote the Lebesgue measure on $\mathbb{R}$ and
let $m_j$ be the restriction of the measure $N\Leb$ to $I_j$.
Consider the product space $\Omega=\prod^{N}_{j=1}I_{j}$, endowed with the product measure $m=\otimes^{N}_{j=1} m_j$.
Define
$$f_{i}(\theta_{1},...,\theta_{N})=\sum_{g\in A(\theta_i), h \in A(\theta_{i+1})}G^{2}(g,h).$$
We will show that for each $i$, 
$$\int_{\Omega}f_{i} dm<a^{-\eta/4}.$$
It will follow that if $\eta$ is large enough, then
$$\int \sum f_i <1/4.$$
In particular, it will follow that there exists $(\theta_1,...,\theta_N)$ such that for every $i$,
\begin{equation}\label{prooffirstsuperexponentialdecay2}
    f_{i}(\theta_{1},...,\theta_{N})\leq 1/4.
\end{equation}
This will prove~(\ref{prooffirstsuperexponentialdecay1}).
We give the argument for $i=1,...,N-1$. The cases for $f_0$ and $f_{N}$ are similar but simpler.
Let
$X_{j}(g)=\{\theta\in I_{j}:g\in A(\theta)\}$.
By definition,
$$\int f_{i}dm=\sum_{g,h\in \Gamma}G^{2}(g,h)N \Leb(X_{i}(g)) N \Leb(X_{i+1}(h))$$
Note, $\Leb(X_{i}(g))\leq 2La^{-|g|}$ for each $i=1,...,N$ and $g\in \Gamma$.
Furthermore $X_{i}(g)$ is empty outside the set
$$A'_{i}=\left \{g\in \Gamma: |g|>\eta \text{ and }\delta^{f}_{e}(x,g)\in [(2i-1)a^{-\epsilon \eta}-La^{-\eta},2ia^{-\epsilon \eta}+La^{-\eta}]\right \}.$$
It follows that
\begin{align*}
    \int f_{i}dm &\leq 4L^{2}N^{2} \sum_{g\in A'_i, h\in A'_{i+1}} G^{2}(g,h) a^{-|g|-|h|}\\
    & \lesssim  e^{2\epsilon \eta} \sum_{g\in A'_{i}, h\in A'_{i+1}} G^{2}(g,h) a^{-|h|-|g|}.
\end{align*}

The last expression is bounded above by
$$e^{2\epsilon \eta} \sum_{u\in \Gamma}N(u)G^{2}(e,u)a^{-|u|}$$
where $N(u)$ is the number of ways to decompose $u=g^{-1}h$ where $g\in A'_i, h\in A'_{i+1}.$
Karlsson's visibility lemma \cite[Lemma~1]{Karlsson} implies that letting $d$ be the maximum distance of a geodesic segment $[g,h]$ to $e$ we have  
$$\delta^{f}_{e}(g,h)\leq 4df(d)+2\sum^\infty_{j=d}f(j)=4da^{-d}+2a^{-d}/(1-a^{-1})\lesssim a^{-d/2}.$$
The implied constant depends only on $a$.
Note, if $g\in A'_i, h\in A'_{i+1}$, then 
$$\delta^{f}_{e}(g,h)\geq a^{-\epsilon \eta}-2La^{-\eta}\geq \frac{1}{2} a^{-\epsilon \eta}$$
if $\epsilon$ is small enough.
It follows that $d\leq 2\epsilon \eta+C_a$ for a constant $C_a$ depending only on $a$.

Let $\alpha$ be a geodesic segment from $e$ to $u$ in $\Gamma$, then $g\alpha$ is a geodesic segment from $g$ to $h$.
There exists a time $j$ such that $g\alpha(j)\in B(e, 2\epsilon \eta+C_a)$.
Thus,
$$g \in \bigcup^{|u|}_{j=0}\alpha(j)^{-1}B(e, 2\epsilon \eta+C_a)$$
which gives at most $(|u| + 1)D_ae^{2\epsilon \eta v}$
possibilities for $g$ for some $v$, since balls $\Gamma$ grow at most exponentially.
The constant $D_a$ only depends on $a$.
Arguing similarly for $h$ we obtain 
$$N(u)\lesssim (|u| + 1)^{2}e^{4\epsilon \eta v}.$$
Furthermore, note $\delta^{f}_{e}(g,h)\leq f(\eta-|u|/2)|u|$ whenever $|g|,|h|>\eta$.
It follows that $a^{-\epsilon \eta}\lesssim a^{|u|/2-\eta}|u|$ so $|u|>\eta$ assuming $\epsilon$ is small enough.
Thus,
\begin{equation}\label{prooffirstsuperexponentialdecay3}
    \int f_{i}dm \lesssim e^{2\epsilon \eta} \sum_{|u|>\eta}(|u| + 1)^{2}e^{4\epsilon \eta v}G^{2}(e,u)a^{-|u|}
\end{equation}


Proposition~\ref{doubleGreenfiniteonspheres}, together with~(\ref{prooffirstsuperexponentialdecay3}) show that
$$\int f_i dm \lesssim e^{2\epsilon(1+2v) \eta}\sum^{\infty}_{n=\eta}(n+1)^2 a^{-n}\lesssim e^{2\epsilon(1+2v) \eta} a^{-\eta/2}.$$
Thus, letting $\epsilon$ be small enough (not depending on $\eta$)
we obtain $$\int f_i dm \leq a^{-\eta/4},$$
showing~(\ref{prooffirstsuperexponentialdecay2}) and thus completing the proof.
\end{proof}

Following the arguments of \cite[Sections~4,5]{GGPY}, we obtain as a corollary weak Floyd-Ancona inequalities at the spectral radius.

\begin{theorem}\label{thmweakAncona}
Let $\Gamma$ be a finitely generated group and let $\mu$ be a symmetric probability measure whose finite support generates $\Gamma$.
For any $\delta>0$, there exists $C\geq 0$ such that for any $x,y,z\in \Gamma$ with $\delta_{y}^f(x,z)>\delta$, we have for every $r\leq R$
$$\frac{1}{C}G_r(x,y)G_r(y,z)\leq G_r(x,z)\leq C G_r(x,y)G_r(y,z).$$
\end{theorem}

We do not give the proof here and we refer to \cite[Theorem~5.1]{GGPY}.
We will prove anyway a stronger result in the following subsection.

Following \cite{GGPY}, we also obtain the following corollary, which is a consequence of Proposition~\ref{propsuperexponentialdecay}, see \cite[Theorem~5.2]{GGPY} for more details.

\begin{corollary}\label{Anconaepsilon}
Let $\Gamma$ be a finitely generated group and let $\mu$ be a symmetric probability measure whose finite support generates $\Gamma$.
For every $\delta>0$ and $\epsilon>0$ there exists $\eta$ such that for every $x,y,z$ with $\delta_y^f(x,z)\geq \delta$,
$$G(x,z;B_{\eta}(y)^c)\leq \epsilon G(x,z).$$
\end{corollary}

\subsection{Floyd-Ancona inequalities}\label{SectionstrongAncona}
We prove a strong version of Floyd-Ancona inequalities.
We start proving a stronger version of Proposition~\ref{propsuperexponentialdecay}
and we consider a finitely generated group $\Gamma$.
Recall that $G$ denotes the Green function $G_R$ at the spectral radius $R$.
We define the function $h_t(\eta)$ by
$$h_t(\eta)=-\log \sup \left \{G(x,y;B_\eta(e)^c),\delta_e^f(x,y)\geq t\right \}.$$
Proposition~\ref{propsuperexponentialdecay} shows that
$h_t(\eta)\geq tc^\eta$,
where $c>1$.
Let $h(\eta)=h_{1/2\eta}(\eta)$.
Then, $h$ is super-linear, that is
\begin{equation}\label{strengthenedsuperexponentialdecay1}
\frac{h(\eta)}{\eta}\to \infty, \eta\to \infty.
\end{equation}
Moreover, for every $x,y$ such that $\delta_e^f(x,y)\geq \frac{1}{2\eta}$, we have
\begin{equation}\label{strengthenedsuperexponentialdecay2}
G(x,y;B_\eta(e)^c)\leq \mathrm{e}^{-h(\eta)}.
\end{equation}

We now define the notion of regularly entering a ball centered at $e$ for a path as follows.
What we call a path here is a finite sequence of elements of $\Gamma$, not necessarily adjacent in the Cayley graph.
We fix $\delta>0$ and $x,y\in \Gamma$  such that $\delta^f_e(x,y)\geq \delta$.
Let $E_\eta(x)$, respectively $E_\eta(y)$, be the set of $z\in B_\eta(e)$ such that $\delta_e^f(x,z)\leq \frac{1}{\eta}$, respectively $\delta_e^f(z,y)\leq \frac{1}{\eta}$.

\begin{definition}\label{definitionregular}
Let $x,y\in \Gamma$ be such that $\delta^f_e(x,y)\geq \delta$.
Call a path from $x$ to $y$ $\eta$-regular if
\begin{enumerate}
    \item it enters $B_\eta(e)$,
    \item its entrance point in $B_\eta(e)$ is in $E_\eta(x)$,
    \item its exit point in $B_\eta(e)$ is in $E_\eta(y)$.
\end{enumerate}
\end{definition}

Beware that the definition depends on $\delta$.
Denote by $\reg(\eta,\delta)$ the set of $\eta$-regular paths from $x$ to $y$, where $\delta_e(x,y)\geq \delta$.
For technical reasons, we will need a slight modification of $\reg(\eta,\delta)$ in the following.

\begin{definition}\label{definitionreg'}
Let $x,y$ be such that $\delta_e^f(x,y)\geq \delta$.
We say that a path from $x$ to $y$ is in $\reg'(\eta,\delta)$ if it satisfies the assumption of Definition~\ref{definitionregular}, replacing $E_{\eta}(x)$ and $E_{\eta}(y)$ by $E_{2\eta}(x)$ and $E_{2\eta}(y)$ respectively.
\end{definition}

Note that we keep $B_{\eta}(e)$ unchanged and only replace $E_{\eta}(x)$ and $E_{\eta}(y)$ with $E_{2\eta}(x)$ and $E_{2\eta}(y)$ in Definition~\ref{definitionreg'}.
We will actually both need to deal with paths in $\reg(\eta,\delta)$ and with paths in $\reg'(\eta,\delta)$ in our proofs below.

Let $G_r(x,y;\reg'(\eta,\delta)^c)$ be the Green function at $r$ restricted to paths that are not $\eta$-regular and
$G_r(x,y;\reg'(\eta,\delta))$ be the Green function at $r$ restricted to $\eta$-regular paths.
Precisely,
$$G_r(x,y;\reg'(\eta,\delta)^c)=\sum_{n\geq 0}\sum_{\underset{\notin \reg'(\eta,\delta)}{(z_1,...,z_{n-1})} }r^n\mu(x^{-1}z_1)\mu(z_1^{-1}z_2)...\mu(z_{n-1}^{-1}y)$$
and
$$G_r(x,y;\reg'(\eta,\delta))=\sum_{n\geq 0}\sum_{\underset{\in \reg'(\eta,\delta)}{(z_1,...,z_{n-1})} }r^n\mu(x^{-1}z_1)\mu(z_1^{-1}z_2)...\mu(z_{n-1}^{-1}y).$$
We have the following enhanced version of Proposition~\ref{propsuperexponentialdecay}.

\begin{proposition}\label{propstrengthenedsuperexponentialdecay}
For each $\delta>0$ there exists a super-linear function $h_\delta$ such that for all $x,y,z\in \Gamma$ with $\delta^{f}_{z}(x,y)>\delta$, we have
$$G(x,y;\reg'(\eta,\delta)^{c})\leq \mathrm{e}^{-h_\delta(\eta)}.$$
\end{proposition}

\begin{proof}
Consider a path from $x$ to $y$ which is not in $\reg'(\eta,\delta)$.
Then, either it does not enter $B_{\eta}(e)$ or denoting by $w_1$ and $w_2$ its entrance and exit points in $B_{\eta}(e)$,
either $\delta_e^f(x,w_1)>1/2\eta$ or $\delta_e^f(w_2,y)>1/2\eta$.
Denote by $G(x,y;\reg'_1(\eta,\delta)^{c})$ the Green function restricted to paths satisfying the condition on $w_1$ and by
$G(x,y;\reg'_2(\eta,\delta)^{c})$ the Green function restricted to paths satisfying the condition on $w_2$.
We thus have
$$G(x,y;\reg'(\eta,\delta)^{c})\leq G(x,y;B_{\eta}(e)^c)+G(x,y;\reg'_1(\eta,\delta)^{c})+G(x,y;\reg'_2(\eta,\delta)^{c}).$$

Proposition~\ref{propsuperexponentialdecay} shows that
$$G(x,y;B_{\eta}(e)^c)\leq \mathrm{e}^{-h_0(\eta)},$$
for some super-linear function $h_0$.

We now deal with $G(x,y;\reg'_1(\eta,\delta)^{c})$.
Conditioning on the first visit to $B_{\eta}(e)$, we have
$$G(x,y;\reg'_1(\eta,\delta)^{c})\leq \sum_{\underset{\delta_e^f(x,w_1)>1/2\eta}{w_1\in B_{\eta}(e)}}G(x,w_1;B_\eta(e)^c)G(w_1,y).$$
Then,~(\ref{strengthenedsuperexponentialdecay1}) shows that
$$G(x,w_1;B_\eta(e)^c)\leq \mathrm{e}^{-h_1(\eta)}$$
for some super-linear function $h_1$.
Since balls grow at most exponentially and since $G(w_1,y)$ is uniformly bounded,
$\sum_{w_1\in B_{\eta}(e),\delta_e^f(x,w_1)>1/2\eta}G(w_1,y)$ grows at most exponentially in $\eta$.
Since $h_1$ is super-linear, up to changing $h_1$, we get
$$G(x,y;\reg'_1(\eta,\delta)^{c})\leq \mathrm{e}^{-h_1(\eta)}.$$
Similarly, we have
$$G(x,y;\reg'_2(\eta,\delta)^{c})\leq \mathrm{e}^{-h_2(\eta)}.$$
Letting $h_3$ be the infimum of $h_0,h_1,h_2$ and $h=h_3-\log 3$, we get
$$G(x,y;\reg'(\eta,\delta)^{c})\leq \mathrm{e}^{-h(\eta)}.$$
Then, $h$ is super-linear and only depends on $\delta$, which concludes the proof.
\end{proof}

Note that the same proof shows that
$$G(x,y;\reg(\eta,\delta)^{c})\leq \mathrm{e}^{-h_\delta(\eta)}$$
for some super-linear function $h_\delta$, although we will not need this in the following.

\begin{definition}\label{defstarlike}
Let $\Omega\subset \Gamma$ and let $x\in \Gamma$.
For $\lambda\geq 1$ and $c\geq 0$, we say that $\Omega$ is $(\lambda,c)$-starlike around $x$ if for every $y\in \Omega$, there exists a path $\alpha$ from $y$ to $x$ that stay inside $\Omega$ and whose length is at most $\lambda d(y,x)+c$.
\end{definition}

We deduce from Proposition~\ref{propstrengthenedsuperexponentialdecay} a strengthened version of weak Floyd-Ancona inequalities.

\begin{proposition}\label{strengthenedweakAncona}
Let $\Gamma$ be a finitely generated group and let $\mu$ be a symmetric admissible finitely supported probability measure on $\Gamma$.
For every $\lambda\geq 1,c\geq 0$ and $\delta\geq 0$,
there exist $C\geq 0$ such that the following holds.
Let $x,y,z$ be three points in $\Gamma$ with $\delta^f_z(x,y)\geq \delta$.
Let $\Omega\subset \Gamma$ be a $(\lambda,c)$-starlike set around $z$.
Then, for every $r\leq R$, we have
$$\frac{1}{C}G_r(x,z;\Omega)G_r(z,y;\Omega)\leq G_r(x,y;\Omega)\leq C G_r(x,z;\Omega)G_r(z,y;\Omega).$$
\end{proposition}

Note that taking $\Omega=\Gamma$, we recover Theorem~\ref{thmweakAncona}.
\begin{proof}
We fix $r\leq R$.
To simplify the notations, we do not refer to $r$ in the following.
In particular, we will write $G=G_r$ for the Green function.
Notice that we can assume for simplicity that $z=e$.
Also, the inequality
$$G(x,e;\Omega)G(e,y;\Omega)\leq C G(x,y;\Omega)$$
is always satisfied (see for example \cite[(2.4)]{Gouezel-local}) so we just need to prove the other inequality.

We fix $\eta_0$ that will be chosen large enough later.
Consider a path for the random walk from $x$ to $y$.
Then, it is either $\eta_0$-regular or not.
Hence,
$$G(x,y;\Omega)\leq G(x,y;\reg(\eta_0,\delta)\cap \Omega)+G(x,y;\reg(\eta_0,\delta)^c\cap \Omega),$$
where $G(x,y;\reg(\eta_0,\delta)\cap \Omega)$ denotes the Green function restricted to $\eta_0$-regular paths that stay inside $\Omega$ and $G(x,y;\reg(\eta_0,\delta)^c\cap \Omega)$ denotes the Green function restricted to paths that stay inside $\Omega$ and that are not $\eta_0$-regular.

If a path is $\eta_0$-regular, in particular, it enters $B_{\eta_0}(e)$, so that conditioning by the first visit to $B_{\eta_0}(e)$, we get
$$G(x,y;\reg(\eta_0,\delta)\cap \Omega)\leq \sum_{w\in B_{\eta_0}(e)}G(x,w;B_{\eta_0}(e)^c\cap\Omega)G(w,y;\Omega).$$
Since $\Omega$ is $(\lambda,c)$-starlike around $e$, if $w\in B_{\eta_0}(e)\cap \Omega$, there is a path from $w$ to $e$ that stays inside $\Omega$ and whose length is bounded by $\lambda \eta_0+c$.
In particular, we have
$G(x,w;\Omega)\leq C_{\eta_0} G(x,e;\Omega)$ and $G(w,y;\Omega)\leq C_{\eta_0} G(e,y;\Omega)$.
Summing over all possible $w\in B_{\eta_0}$, we obtain
\begin{equation}\label{firstequationweakAncona}
    G(x,y;\reg(\eta_0,\delta)\cap \Omega)\leq C'_{\eta_0}G(x,e;\Omega)G(e,y;\Omega).
\end{equation}

We now find an upper bound of $G(x,y;\reg(\eta_0,\delta)^c\cap \Omega)$.
Among paths from $x$ to $y$ that are not $\eta_0$-regular, some of them are $2\eta_0$ regular and the other ones are not.
We thus have
\begin{equation}\label{secondequationweakAncona}
\begin{split}
    G(x,y;\reg(\eta_0,\delta)^c\cap \Omega)&\leq G(x,y;\reg(\eta_0,\delta)^c\cap \reg(2\eta_0,\delta)\cap \Omega)\\
    &+G(x,y;\reg(\eta_0,\delta)^c\cap \reg(2\eta_0,\delta)^c\cap \Omega).
\end{split}
\end{equation}
Consider a path $\alpha$ from $x$ to $y$ which is not $\eta_0$-regular, but which is $2\eta_0$-regular.
Let $u$, respectively $v$ be the entrance, respectively exit points in $B_{2\eta_0}$.
If $\eta_0$ is large enough, then $\delta_{e}^f(u,v)\geq \delta/2$, since $\delta_e^f(x,u)\leq 1/2\eta_0$, $\delta_e^f(v,y)\leq 1/2\eta_0$ and $\delta_e^f(x,y)\geq \delta$.
Let $\beta$ be the sub-path of $\alpha$ from $u$ to $v$ and assume first that $\beta$ enters $B_{\eta_0}(e)$.
Let $w_1$ and $w_2$ be the entrance and exit point of $\beta$ in $B_{\eta_0}(e)$.
Since $B_{\eta_0}(e)\subset B_{2\eta_0}(e)$, $w_1$ and $w_2$ also are the entrance and exit point of $\alpha$ in $B_{\eta_0}(e)$.
By definition, $\alpha$ is not $\eta_0$-regular, so that either
$\delta_e^f(x,w_1)>\frac{1}{\eta_0}$ or $\delta_e^f(w_2,y)>\frac{1}{\eta_0}$.
In the first case, we get
$\delta_e^f(u,w_1)>\frac{1}{2\eta_0}$, in the second case, we get
$\delta_e^f(w_2,v)>\frac{1}{2\eta_0}$.
In both cases, $\beta$ is in $\reg'(\eta_0,\delta/2)$.
Note that this is also the case if $\beta$ does not enter $B_{\eta_0}(e)$ at all.
Conditioning on the first and last visit to $B_{2\eta_0}(e)$, we have
\begin{align*}
    G(x,y;\reg(\eta_0,\delta)^c\cap \reg(2\eta_0,\delta)\cap \Omega)\leq \sum_{u,v\in B_{2\eta_0}(e)}&G(x,u;\Omega)\\
    &G(u,v;\reg'(\eta_0)^c)G(v,y;\Omega).
\end{align*}
Then, Proposition~\ref{propstrengthenedsuperexponentialdecay} yields
$$G(u,v;\reg'(\eta_0)^c)\leq \mathrm{e}^{-h_{\delta/2}(\eta_0)}$$
for some super-linear function $h_{\delta/2}$.
Since $\Omega$ is $(\lambda,c)$-starlike around $e$, if $u$ and $v$ are in $\Omega$, then
$G(u,e;\Omega)\geq p^{d(e,u)}\geq p^{2\eta_0}$ for some $p<1$ and similarly, $G(e,v;\Omega)\geq p^{2\eta_0}$.
Up to changing $h_{\delta/2}(\eta_0)$ into some other super-linear function, we thus have
$$G(u,v;\reg'(\eta_0)^c)\leq \mathrm{e}^{-h_{\delta/2}(\eta_0)}G(u,e;\Omega)G(e,v;\Omega).$$
Also,
$$G(x,u;\Omega)G(u,e;\Omega)\leq C G(x,e;\Omega)$$
and
$$G(e,v;\Omega)G(v,y;\Omega)\leq C G(e,y;\Omega).$$
This shows that
$$G(x,y;\reg(\eta_0,\delta)^c\cap \reg(2\eta_0,\delta)\cap \Omega)\leq C^2\mathrm{e}^{-h_{\delta/2}(\eta_0)}\sum_{u,v\in B_{2\eta_0}(e)}G(x,e;\Omega)G(e,y;\Omega).$$
Since balls grow at most exponentially, changing $h_{\delta/2}$ again, we get
$$G(x,y;\reg(\eta_0,\delta)^c\cap \reg(2\eta_0,\delta)\cap \Omega)\leq \mathrm{e}^{-h_{\delta/2}(\eta_0)} G(x,e;\Omega)G(e,y;\Omega).$$

We deal with the second term in~(\ref{secondequationweakAncona}).
We have by definition
$$G(x,y;\reg(\eta_0,\delta)^c\cap \reg(2\eta_0,\delta)^c\cap \Omega)\leq G(x,y;\reg(2\eta_0,\delta)^c\cap \Omega).$$
A path which is not $2\eta_0$-regular can be $4\eta_0$-regular or not.
We thus have
\begin{align*}
    G(x,y;\reg(2\eta_0,\delta)^c\cap \Omega)&\leq G(x,y;\reg(2\eta_0,\delta)^c\cap \reg(4\eta_0,\delta)\cap \Omega)\\
    &+G(x,y;\reg(2\eta_0,\delta)^c\cap \reg(4\eta_0,\delta)^c\cap \Omega).
\end{align*}
We again first deal with $ G(x,y;\reg(2\eta_0,\delta)^c\cap \reg(4\eta_0,\delta)\cap \Omega)$.
Decomposing a path in $\reg(2\eta_0,\delta)^c\cap \reg(4\eta_0,\delta)$ according to its entrance and exit point in $B_{4\eta_0}(e)$, we similarly get
$$G(x,y;\reg(2\eta_0,\delta)^c\cap \reg(4\eta_0,\delta)\cap \Omega)\leq \mathrm{e}^{-h_{\delta/2}(2\eta_0)} G(x,e;\Omega)G(e,y;\Omega).$$

Doing a similar decomposition for path which are not $2^i\eta_0$-regular but which are $2^{i+1}\eta_0$-regular, we finally get
\begin{equation}\label{thirdequationweakAncona}
    G(x,y;\reg(\eta_0,\delta)^c\cap \Omega)\leq \sum_{i\geq 0}\mathrm{e}^{-h(2^i\eta_0)}G(x,e;\Omega)G(e,y;\Omega)
\end{equation}
for some super-linear function $h$.
The sum $\sum_{i\geq 0}\mathrm{e}^{-h(2^i\eta_0)}$ is thus finite, so that combining~(\ref{firstequationweakAncona}) with~(\ref{thirdequationweakAncona}), we have
$$G(x,y;\Omega)\leq CG(x,e;\Omega)G(e,y;\Omega).$$
This concludes the proof.
\end{proof}

We can now prove that strong Floyd-Ancona inequalities hold.
Recall the following definition from the introduction.

\begin{definition}\label{deffellowtravel}
Let $x,y$ and $x',y'$ be four points in $\Gamma$.
Let $\delta\geq 0$.
We say that the pairs $(x,y)$ and $(x',y')$ $\delta$-fellow-travel for a time at least $n$ if there exist distinct points $z_1,...,z_n\in \Gamma$ such that for all $i$, $\delta_{z_i}^f(x,y)\geq \delta$ and $\delta_{z_i}^f(x',y')\geq \delta$.
\end{definition}

\begin{theorem}\label{thmstrongAncona}
Let $\Gamma$ be a finitely generated group and let $\mu$ be a symmetric admissible finitely supported probability measure on $\Gamma$.
For every $\delta \geq 0$, there exist $K\geq 0$ and $0<\rho<1$ such that the following holds.
For every $x,y,x',y'$ such that the pairs $(x,y)$ and $(x',y')$ $\delta$-fellow travel for a time at least $n$, we have for every $r\leq R$
$$\left |\frac{G_r(x,y)G_r(x',y')}{G_r(x,y')G_r(x',y)}-1\right |\leq K\rho^n.$$
\end{theorem}

Before proving this theorem, we introduce the following terminology.
\begin{definition}\label{deftransitionFloyd}
Let $\alpha$ be a geodesic in $\Cay(\Gamma,S)$ and let $\delta>0$.
A $\delta$-transition point on $\alpha$ is a point $y$ on $\alpha$ such that for every $x,z$ on $\alpha$ such that $x,y,z$ are aligned in this order,
we have
$\delta_y^f(x,z)\geq \delta$.
\end{definition}

Beware that the terminology transition point is also commonly used for relatively hyperbolic groups, as noted in the introduction.
However, if $\Gamma$ is relatively hyperbolic, then a transition in sense of Definition~\ref{deftransitionusual} is also a transition point in the sense of Definition~\ref{deftransitionFloyd}, according to Proposition~\ref{Floydgeo}.
We will denote by $\Tr_{\delta}(\alpha)$ the set of $\delta$-transition points on a geodesic $\alpha$.

\begin{lemma}\label{Globaltransitionimplieslocaltransition}
For every $\delta$ there exists $D(\delta)\geq0$ and $\delta'\leq \delta$ such that the following holds.
Let $\alpha$ be a geodesic with endpoints $\alpha_-$ and $\alpha_+$ and let $w$ be a point in $\Gamma$ satisfying
$\delta_{w}^f(\alpha_-,\alpha_+)\geq \delta$.
Then there exists a $\delta'$-transition point $w'$ on $\alpha$ such that $d(w,w')\leq D(\delta)$.
\end{lemma}

\begin{proof}
According to Karlsson's lemma \cite[Lemma~1]{Karlsson}, $w$ is within a bounded distance of a point $w'$ on $\alpha$.
By definition of the Floyd distance, there exists $\delta'$ such that $\delta_{w'}(\alpha_-,\alpha_+)\geq \delta'$.
Thus, we only need to prove that whenever $w$ is on a geodesic $\alpha$ and satisfies $\delta_{w}^f(\alpha_-,\alpha_+)\geq \delta$, then $w$ is a $\delta'$-transition point on $\alpha$ for some $\delta'$.

Assume this is not the case.
Then, there exist a sequence $w_n$ on geodesics $\alpha^{(n)}$ and a sequence of points $x_n,y_n$ on $\alpha^{(n)}$ such that $\delta_w(\alpha^{(n)}_-,\alpha^{(n)}_+)\geq \delta$ and $\delta_w(x_n,y_n)$ tends to 0.
By left invariance, we can assume that $w_n=w$ is fixed.
Necessarily, $x_n$ and $y_n$ tend to infinity.
Hence, up to changing $x_n$ and $y_n$, the subgeodesics from $\alpha^{(n)}_-$ to $x_n$ and from $\alpha^{(n)}_+$ to $y_n$ are arbitrarily far away from $w$.
In particular, the Floyd length of these subgeodesic, seen from $w$, goes to 0.
Moreover, there is a sequence of paths from $x_n$ to $y_n$ whose Floyd length seen from $w$ goes to 0.
Concatenating these paths yields a path from $\alpha^{(n)}_-$ to $^{(n)}_+$ whose Floyd length goes to 0, contradicting the fact that $\delta_w(\alpha^{(n)}_-,\alpha^{(n)}_+)\geq \delta$.
\end{proof}

\begin{lemma}\label{thintriangles}
For every $\delta>0$ there exists $\Delta(\delta)\geq 0$ and $\delta_0\leq \delta$ such that the following holds.
Let $x,y,z\in \Gamma$ and consider a geodesic triangle with vertices $x,y,z$.
Then any point $w$ on $[x,y]$ satisfying $\delta_w^f(x,y)\geq \delta$ is within a distance at most $\Delta(\delta)$ from $\Tr_{\delta_0}([x,z])\cup \Tr_{\delta_0}([y,z])$.
\end{lemma}

\begin{proof}
Since $\delta_w^f(x,y)\geq \delta$, either $\delta_w^f(x,z)\geq \delta/2$ or $\delta_w^f(y,z)\geq \delta/2$.
Thus, the result follows from Lemma~\ref{Globaltransitionimplieslocaltransition}.
\end{proof}

We now prove Theorem~\ref{thmstrongAncona}.

\begin{proof}
Let $D=D(\delta)$ and $\delta'$ be the numbers given by Lemma~\ref{Globaltransitionimplieslocaltransition}. Then, for every $i$, there exist $\delta'$-transition points $\hat{z}_i$ on a geodesic $[x,y]$ such that
$d(z_i,\hat{z}_i)\leq D$.
Up to reindexing the $z_i$, we may assume that $d(\hat{z}_i,x)\leq d(\hat{z}_{i+1},x)$.
We fix some integer $k$ that we will choose later and we consider the points $z_k,z_{2k},...,z_{mk}$, with $m=\lfloor n/k\rfloor$.
By definition, for every $i$, $\hat{z}_{(i+1)k}$ is a $\delta'$-transition point on the subgeodesic of $[x,y]$ from $\hat{z}_{ik}$ to $y$.
Let $\Delta=\Delta(\delta')$ and $\delta_0\leq \delta'$ be the numbers given by Lemma~\ref{thintriangles} for $\delta'$.
We define sets $\Omega_i$ by
$$\Omega_i=\{w\in \Gamma,d(z_{(m-j+1)k},\Tr_{\delta_0}([w,y]))\leq \Delta+D\},$$
for some geodesic $[w,y]$ from $w$ to $y$.
We first prove the following.
\begin{lemma}
If $k$ is large enough, then $\Omega_{i+1}\subset \Omega_i$.
Moreover, $x,x'\in \Omega_m$ and $y,y'\notin \Omega_1$.
\end{lemma}

\begin{proof}
Let $w\in \Omega_{i+1}$ and consider a geodesic $[w,y]$ from $w$ to $y$ and a point $\tilde{w}$ on $\Tr_{\delta'}([w,y])$ with
\begin{equation}\label{omegaidecreasing1}
    d(\tilde{w},z_{(m-i)k})\leq \Delta+D.
\end{equation}
Let $[w,\hat{z}_{(m-i)k}]$ be a geodesic from $w$ to $\hat{z}_{(m-i)k}$ and $[\hat{z}_{(m-i)k},y]$ the subgeodesic of $[x,y]$ from $\hat{z}_{(m-i)k}$ to $y$.
Consider the triangle composed by these three geodesics.
Since $\hat{z}_{(m-i+1)k}$ is a $\delta'$-transition point on $[\hat{z}_{(m-i)k},y]$, Lemma~\ref{thintriangles} shows that there is a $\delta_0$-transition point $\tilde{z}$ which is either on $[w,y]$ or on $[w,\hat{z}_{(m-i)k}]$ satisfying
\begin{equation}\label{omegaidecreasing2}
    d(\tilde{z},\hat{z}_{(m-i+1)k})\leq \Delta.
\end{equation}
If we prove that $\tilde{z}$ is on $[w,y]$, then $w\in \Omega_i$ since $d(z_{(m-i+1)k},\hat{z}_{(m-i+1)k})\leq D$.
To show that $\Omega_{i+1}\subset \Omega_i$, we just need to prove that if $k$ is large enough, then $\tilde{z}$ cannot be on $[w,\hat{z}_{(m-i)k}]$.

Since $\tilde{w}$ in on $[w,y]$ and $d(\tilde{w},\hat{z}_{(m-i)k})$ is uniformly bounded by~(\ref{omegaidecreasing1}), there exists $c$ such that
$$d(w,y)\geq d(w,\hat{z}_{(m-i)k})+d(\hat{z}_{(m-i)k},y)-c.$$
Also, $\hat{z}_{(m-i+1)k}$ is by definition on a geodesic from $\hat{z}_{(m-i)k}$ to $y$, so we get
$$d(w,y)\geq d(w,\hat{z}_{(m-i)k})+d(\hat{z}_{(m-i)k},\hat{z}_{(m-i+1)k})+d(\hat{z}_{(m-i+1)k},y)-c.$$
Finally, by the triangle inequality, $d(w,y)-d(\hat{z}_{(m-i+1)k},y)\leq d(w,\hat{z}_{(m-i+1)k})$ so we finally get
\begin{equation}\label{equationinclusionsetsOmega1}
d(w,\hat{z}_{(m-i+1)k})\geq d(w,\hat{z}_{(m-i)k})+d(\hat{z}_{(m-i)k},\hat{z}_{(m-i+1)k})-c.
\end{equation}
Now assume by contradiction that $\tilde{z}$ in on $[w,\hat{z}_{(m-i)k}]$.
Then, $\hat{z}_{(m-i+1)k}$ is within a bounded distance of $[w,\hat{z}_{(m-i)k}]$ by~(\ref{omegaidecreasing2}) and so there exists $c'$ such that
\begin{equation}\label{equationinclusionsetsOmega2}
    d(w,\hat{z}_{(m-i)k})\geq d(w,\hat{z}_{(m-i+1)k})+d(\hat{z}_{(m-i)k},\hat{z}_{(m-i+1)k})-c'.
\end{equation}
Summing~(\ref{equationinclusionsetsOmega1}) and~(\ref{equationinclusionsetsOmega2}), we get
$$2d(\hat{z}_{(m-i)k},\hat{z}_{(m-i+1)k})\leq c+c',$$
which cannot happen if $d(\hat{z}_{(m-i)k},\hat{z}_{(m-i+1)k})$ is large enough, which is true if $k$ is large enough.

This concludes the first part of the lemma.
We now prove that $x,x'\in \Omega_m$ and $y,y'\notin \Omega_1$.
By symmetry in the definition of $\Omega_i$, we just need to prove that $x\in \Omega_m$ and $y\notin \Omega_1$.
By definition $\hat{z}_k$ is in $\Tr_{\delta_0}([x,y])$ and since $d(z_k,\hat{z}_k)\leq D$, $x\in \Omega_m$.
Also, $\Tr_{\delta_0}([y,y])$ is reduced to $\{y\}$ and if $k$ is large enough, then $d(\hat{z}_m,y)>\Delta+D$, so $y\notin\Omega_1$.
This concludes the proof of the lemma.
\end{proof}

Let $h_1(w)=\frac{G_r(w,y)}{G_r(x',y)}$ and $h_2(w)=\frac{G_r(w,y')}{G_r(x',y')}$.
Then, $h_1$ is $r$-harmonic everywhere except at $y$ and $h_2$ is $r$-harmonic everywhere except at $y'$.
In particular, both are $r$-harmonic on $\Omega_1$.
We will construct functions $h_1^j$ and $h_2^j$ that are $r$-harmonic on $\Omega_{j+1}$, with $h_1^0=h_1$ and $h_2^0=h_2$.
We will also construct functions $\varphi^j$ satisfying the following conditions
\begin{enumerate}[a)]
    \item $h_i^{j}=h_i^{j+1}+\varphi^{j+1}$ on $\Omega_{j+2}$,
    \item $h_i^{j}\geq \varphi^{j+1}\geq \beta_0 h_i^{j}$ for some $\beta_0>0$.
\end{enumerate}
Before constructing them, let us show how we can conclude.
By construction, we have $h_i^{j+1}\leq (1-\beta_0)h_i^{j}$ and so $h_i^{m-1}\leq (1-\beta_0)^{m-2}h_i$.
By definition, on $\Omega_m$, we have
$h_i=\sum_j\varphi^j+h_i^{m-1}$.
Thus, since $x\in \Omega_m$,
$$|h_1(x)-h_2(x)|=|h_1^{m-1}(x)-h_2^{m-1}(x)|\leq (1-\beta_0)^{m-2}(h_1(x)+h_2(x)).$$
Proposition~\ref{strengthenedweakAncona} shows that $\frac{G_r(x,y)}{G_r(x,z_k)G_r(z_k,y)}$ and $\frac{G_r(x',y)}{G_r(x',z_k)G_r(z_k,y)}$ are bounded away from 0 and infinity.
The ratio of these two quantities is $\frac{h_1(x)G_r(x',z_k)}{G_r(x,z_k)}$ so that this quantity also is bounded away from 0 and infinity.
Similarly, $\frac{h_2(x)G_r(x',z_k)}{G_r(x,z_k)}$ is bounded away from 0 and infinity.
Finally, we see that $\frac{h_1(x)}{h_2(x)}$ is bounded away from 0 and infinity.
We thus get
$$\left |\frac{h_1(x)}{h_2(x)}-1\right |\leq C(1-\beta_0)^{m-2}\leq \frac{C}{(1-\beta_0)^{3}}\left ((1-\beta_0)^{1/k}\right)^n.$$
This will prove the theorem, taking $K=\frac{C}{(1-\beta_0)^{3}}$ and $\rho=(1-\beta_0)^{1/k}$.

Let us construct $h_i^j$ and $\varphi^j$ by induction to conclude.
Assume that the functions $h_i^j$ and $\varphi^j$ were constructed.
Since $h_1^j$ is $r$-harmonic on $\Omega_{j+1}$,
for $w\in \Omega_{j+2}\subset \Omega_{j+1}$, we have
\begin{equation}\label{strongAnconaequation1}
    h_1^j(w)=\sum_{w'\in \Omega_{j+1}^c}G(w,w';\Omega_{j+1})h_1^j(w').
\end{equation}
We want to apply Proposition~\ref{strengthenedweakAncona} to $G(w,w';\Omega_{j+1})$.
We thus have to prove that $\Omega_{j+1}$ is starlike in the sense of Definition~\ref{defstarlike}.

\begin{lemma}
There exist $\lambda$ and $c$ such that for every $j$, the set $\Omega_{j+1}$ is $(\lambda,c)$-starlike around $z_{(m-j)k}$.
\end{lemma}

\begin{proof}
Let $w\in \Omega_{j+1}$.
Then there exists a $\delta_0$-transition point $w'$ on a geodesic $[w,y]$ such that $d(z_{(m-j)k},w')\leq \Delta+D$.
Consider the path $\alpha$ obtained by concatenating the subgeodesic $\alpha_1$ of $[w,y]$ from $w$ to $w'$ and a geodesic $\alpha_2$ from $w'$ to $\hat{z}_{(m-j)k}$.
Then, $\alpha_1$ stays in $\Omega_{j+1}$.
Moreover, for any $w''\in \Gamma$, $w''$ is a $\delta_0$-transition point on any geodesic $[w'',y]$ from $w''$ to $y$ by definition.
Hence, if $w''$ is on $\alpha_2$, then $d(z_{(m-j)k},w'')\leq\Delta+D$ and so $w''\in \Omega_{j+1}$.
Thus, $\alpha$ itself is contained in $\Omega_{j+1}$.
Moreover, since $\alpha_1$ is a geodesic, the length of $\alpha$ is at most $d(w,w')+\Delta\leq d(w,z_{(m-j)k})+2\Delta+2D$.
This concludes the proof of the lemma.
\end{proof}

To apply Proposition~\ref{strengthenedweakAncona}, we also prove the following.

\begin{lemma}
If $k$ is large enough, then there exists $\delta_1$, only depending on $k$ and $\delta$, such that for $w\in \Omega_{j+2}$, $w'\notin \Omega_{j+1}$, we have $\delta_{z_{(m-j)k}}^f(w,w')\geq \delta_1$.
\end{lemma}

\begin{proof}
We will use the thin-triangle property twice.
Let $w'\notin \Omega_{j+1}$ and $w\in \Omega_{j+2}$.
Consider the triangle consisting of geodesics $[\hat{z}_{(m-j-1)k},w']$ from $\hat{z}_{(m-j-1)k}$ to $w'$, $[w',y]$ from $w'$ to $y$ and the subgeodesic $[\hat{z}_{(m-j-1)k},y]$ of $[x,y]$ from $\hat{z}_{(m-j-1)k}$ to $y$.

Then $\hat{z}_{(m-j)k}$ is a $\delta'$-transition point on $[\hat{z}_{(m-j-1)k},y]$.
Recall that we denote by $\Delta=\Delta(\delta')$ and $\delta_0\leq \delta'$ the numbers given by Lemma~\ref{thintriangles} for $\delta'$.
Then, according to this lemma, there is a $\delta_0$-transition point $\tilde{z}$ which lies either on $[w',y]$ or on $[\hat{z}_{(m-j-1)k},w']$ and satisfying
\begin{equation}\label{omegaidecreasing3}
    d(\tilde{z},\hat{z}_{(m-j)k})\leq \Delta.
\end{equation}
Since $w'\notin \Omega_{j+1}$, $\tilde{z}$ has to lie on $[\hat{z}_{(m-j-1)k},w']$.

Now, consider the triangle given by the geodesic $[\hat{z}_{(m-j-1)k},w']$ and geodesics $[w,w']$ from $w$ to $w'$ and $[w,\hat{z}_{(m-j-1)k}]$ from $w$ to $\hat{z}_{(m-j-1)k}$.
Then, using again Lemma~\ref{thintriangles}, there exist numbers $\Delta(\delta_0)\geq 0$ and $\delta_0'\leq \delta_0$ and there exists a $\delta_0'$-transition point $\tilde{z}'$ which lies either on $[w,w']$ or on $[w,\hat{z}_{(m-j-1)k}]$ such that $d(\tilde{z},\tilde{z}')\leq \Delta(\delta_0)$.
We prove by contradiction that if $k$ is large enough, then $\tilde{z}'$ cannot lie on $[w,\hat{z}_{(m-j-1)k}]$.
Indeed, if it were the case, then $d(\tilde{z}',\hat{z}_{(m-j)k})$ would be uniformly bounded by~(\ref{omegaidecreasing3}) and so we would have
\begin{equation}\label{equationseparationOmega1}
    d(w,\hat{z}_{(m-j-1)k})\geq d(w,\hat{z}_{(m-j)k})+d(\hat{z}_{(m-j)k},\hat{z}_{(m-j-1)k})-c.
\end{equation}
However, since $w\in \Omega_{j+2}$, $\hat{z}_{(m-j-1)k}$ is within a bounded distance of a point on a geodesic from $w$ to $y$.
Thus,
$$d(w,y)\geq d(w,\hat{z}_{(m-j-1)k})+d(\hat{z}_{(m-j-1)k},y)-c',$$
so that
\begin{equation}\label{equationseparationOmega2}
d(w,\hat{z}_{(m-j)k})\geq d(w,\hat{z}_{(m-j-1)k})+d(\hat{z}_{(m-j-1)k},\hat{z}_{(m-j)k})-c'.
\end{equation}
Summing~(\ref{equationseparationOmega1}) and~(\ref{equationseparationOmega2}), we would get
$$2d(\hat{z}_{(m-j-1)k},\hat{z}_{(m-j)k})\leq c+c'.$$
This cannot happen if $d(\hat{z}_{(m-i)k},\hat{z}_{(m-i+1)k})$ is large enough, which is true if $k$ is large enough.
This proves that $\tilde{z}'$ lies on $[w,w']$.
In particular, this proves that $\hat{z}_{(m-j)k}$ is within a bounded distance of a transition point on $[w,w']$.
Since $z_{(m-j)k}$ is within a bounded distance of $\hat{z}_{(m-j)k}$, this concludes the proof of the lemma.
\end{proof}

We now fix $k$ large enough so that all conditions above are satisfied.
Proposition~\ref{strengthenedweakAncona} shows that for
$w\in \Omega_{j+2}$ and $w'\in \Omega_{j+1}$, we have
$$G(w,w';\Omega_{j+1})\asymp G(w,z_{(m-j-1)k};\Omega_{j+1})G(z_{(m-j-1)k},w';\Omega_{j+1}).$$
Then,~(\ref{strongAnconaequation1}) shows that for $w\in \Omega_{j+2}$,
$$h_1(w)\asymp \sum_{w'\in \Omega_{j+1}}G(w,z_{(m-j-1)k};\Omega_{j+1})G(z_{(m-j-1)k},w';\Omega_{j+1}).$$
In other words,
$$h_1(w)\asymp G(w,z_{(m-j-1)k};\Omega_{j+1}) h_1(z_{(m-j-1)k}).$$
Since $x'\in \Omega_{j+2}$,
we also have
$$\frac{h_1^j(w)}{h_1^j(x')}\asymp \frac{G_r(w,z_{(m-j-1)k};\Omega_{j+1})}{G_r(x',z_{(m-j-1)k};\Omega_{j+1})}.$$
For small enough $\beta$, we set
$$\varphi^{j+1}=\beta h_1^j(x')\frac{G_r(w,z_{(m-j-1)k};\Omega_{j+1})}{G_r(x',z_{(m-j-1)k};\Omega_{j+1})}.$$
If $\beta$ is small enough,
then $h_1^{j}\geq \varphi^{j+1}\geq \beta_0 h_1^{j}$, for some fixed $\beta_0>0$.
We then set $h_1^{j+1}=h_1^j-\varphi^{j+1}$.
Notice that $\varphi^{j+1}$ does not depend on $h_1^j$, but only on its value at $x'$.
By induction, it only depends on $h_1^0(x')=1$, so it does not depend on $h_1$.
We also have $h_2^{j}\geq \varphi^{j+1}\geq \beta_0 h_2^{j}$, so that the functions $h_i^{j}$ and $\varphi^{j+1}$ thus constructed satisfy both Conditions~a) and~b) above.
This concludes the proof of the theorem.
\end{proof}

Theorem~\ref{maintheoremAncona} follows from Proposition~\ref{strengthenedweakAncona} and Theorem~\ref{thmstrongAncona}.

\section{Description of the Martin boundary}\label{SectiondescriptionMartin}
This section if devoted to the proofs of Theorem~\ref{maintheorem} and Theorem~\ref{minimality}.
We consider a non-elementary relatively hyperbolic group $\Gamma$ and a symmetric probability measure $\mu$ whose finite support generates $\Gamma$.
Recall the definition of the $r$-geometric boundary given in Section~\ref{Sectionspectraldegeneracy}.
Our goal is to prove that the $r$-geometric boundary coincides with the $r$-Martin boundary.
We start proving the following.

\begin{proposition}\label{propconvergencegeometricMartin}
Let $\Gamma$ be a non-elementary relatively hyperbolic group with respect to a collection of virtually abelian subgroups $\Omega$.
Let $\mu$ be a symmetric probability measure whose finite support generates $\Gamma$ and let $R$ be its spectral radius.
Let $r\leq R$.
Let $g_n$ be a sequence of elements of $\Gamma$.
Assume that $g_n$ converges to a point in the $r$-geometric boundary.
Then, $g_n$ converges to a point in the $r$-Martin boundary.
\end{proposition}

We will deal separately with conical limit points and parabolic limit points.
First, it is proved in \cite{GGPY} that as soon as weak Ancona inequalities (Theorem~\ref{thmweakAncona}) are satisfied, then if $g_n$ converges to a conical limit point, it also converges to the Martin boundary.
It is also proved that the corresponding point in the Martin boundary is minimal and that two distinct conical limit points yield two distinct points in the Martin boundary.
Note that this is also a direct consequence of strong Ancona inequalities.
In particular, we have the following proposition.

\begin{proposition}\label{propconvergenceconical}
Let $\Gamma$ be a non-elementary relatively hyperbolic group.
Let $\mu$ be a symmetric probability measure whose finite support generates $\Gamma$ and let $R$ be its spectral radius.
Let $r\leq R$.
Let $g_n$ be a sequence of elements of $\Gamma$ which converges to a conical limit point.
Then, $g_n$ converges to a minimal point in the $r$-Martin boundary.
\end{proposition}

We thus are left with parabolic limit points.

\subsection{Parabolic limit points: the degenerate case}\label{Sectionparaboliclimitpointsdegenerate}
We focus here on degenerate parabolic points and we prove the following.
If $g_n$ converges to a degenerate-parabolic limit point in the Bowditch boundary, then $g_n$ converges in the $R$-Martin boundary.
We assume that the parabolic subgroups are virtually nilpotent.

\begin{proposition}\label{propconvergenceparabolic1}
Let $\Gamma$ be a non-elementary relatively hyperbolic group with respect to a collection of virtually nilpotent subgroups.
Let $\mu$ be a symmetric probability measure whose finite support generates $\Gamma$ and let $R$ be its spectral radius.
Let $\xi$ be a degenerate parabolic limit point.
Let $g_n$ be a sequence of elements of $\Gamma$ converging to $\xi$.
Then, $g_n$ converges to a point in the $R$-Martin boundary.
\end{proposition}

Let $H$ be a parabolic subgroup, so that $H$ is virtually nilpotent.
There is a finite index subgroup of $H$ which is isomorphic to a nilpotent group $\mathcal{N}$.
Any section $H/\mathcal{N}\to H$ identifies $H$ as a set with $\mathcal{N}\times E$, where $E$ is finite.
Let $\eta\geq 0$.
Then, the $\eta$-neighborhood of $H$, which we denote by $N_{\eta}(H)$, is identified with $H\times E_{\eta}$, where $E_{\eta}$ is finite.
Thus, $N_{\eta}(H)$ can also be identified with $\mathcal{N}\times E'_{\eta}$, where $E'_{\eta}$ is finite.
We will use the notation $E'_{\eta}=\{1,...,N_{\eta}\}$.

Actually, as a set, $\Gamma$ can be $H$-equivariantly identified with $H\times \N$.
Indeed, $H$ acts by left multiplication on $\Gamma$ and the quotient is countable.
We order elements in the quotient according to their distance to $H$.
It follows that
\begin{enumerate}
    \item $N_{\eta}(H)$ can be $\mathcal{N}$-equivariantly identified with $\mathcal{N}\times \{1,...,N_\eta\}$ as above,
    \item if $\eta\leq \eta'$, then $N_\eta\leq N_{\eta'}$. In other words, the set $\mathcal{N}\times \{1,...,N_\eta\}$, identified with $N_{\eta}(H)$ is a subset of $\mathcal{N}\times \{1,...,N_{\eta'}\}$ identified with $N_{\eta'}(H)$.
\end{enumerate}

Recall from Section~\ref{Sectionspectraldegeneracy} that we denote by $p_{H,r}$ the first return kernel to $H$ for $r\mu$.
More generally, we define the first return kernel to $N_\eta(H)$ for $r\mu$ as
$$p_{H,\eta,r}(h,h')=\sum_{n\geq 0}\sum_{\underset{\notin N_\eta(H)}{g_1,...,g_{n-1}}}r^n\mu(h^{-1}g_1)\mu(g_1^{-1}g_2)...\mu(g_{n-2}^{-1}g_{n-1})\mu(g_{n-1}^{-1}h').$$
To simplify the notation, whenever $H$, $r$ and $\eta$ are fixed, identifying $h,h'$ with $(x,k)$,$(x',k')$, $x,x'\in \mathcal{N}$ and $k,k'\in \{1,...,N_\eta\}$, we write
$$p_{H,\eta,r}((x,k),(x',k'))=p_{k,k'}(x,x').$$
Let $G_{H,\eta,r}$ be the Green function associated with $p_{H,\eta,r}$ and let $K_{H,\eta,r}$ be the corresponding Martin kernel.
The following result is proved in \cite[Lemma~5.8]{DGGP} for $r=1$, but the same proof works in general.
We provide it for convenience.

\begin{lemma}\label{sameGreen}
For every $h,h'\in N_\eta(H)$, for every $r\leq R$ and for every $\eta\geq 0$, we have $G_{H,\eta,r}(h,h')=G_r(h,h')$.
\end{lemma}

In other words, the Green function $G_{H,\eta,r}$ at 1 coincides with the restriction of the initial Green function at $r$ on $N_\eta(H)$.

\begin{proof}
Given a transition kernel $p$ with convolution powers $p^{(n)}$, the weight of a trajectory $(x_1,...,x_k)$ is given by
$$W(x_1,...,x_n)=p(x_1,x_2)...p(x_{n-1},x_n).$$
Note that the Green function from a point $x$ to a point $y$ is then the sum of the weights of every trajectory from $x$ to $y$.

Every trajectory from $h$ to $h$ for the initial walk driven by $r\mu$ defines a trajectory from $h$ to $h'$ for $p_{H,\eta,r}$ having the same weight, by conditioning the trajectory on the successive passages through $N_\eta(H)$.
Every trajectory for $p_{H,\eta,r}$ is uniquely obtained in such a way.
Summing over all trajectories, the two Green functions coincide.
\end{proof}
Let $\mathcal{A}=\mathcal{N}/\mathcal{N}'$ be the abelianization of $\mathcal{N}$.
Then, $\mathcal{A}$ is abelian of rank $d$, so it is isomorphic to $\Z^d\times F$, where $F$ is finite.
Let $x\in \mathcal{N}$.
Let $\pi(x)$ be the projection of $x$ on $\Z^d$.
That is, $\pi$ is the composition of the first projection of $A=\Z^d\times F$ on $\Z^d$ and the projection of $\mathcal{N}$ on its abelianization $\mathcal{A}$.
For $u\in \R^d$, we define for $j,k\in \{1,...,N_{\eta}\}$,
$$F_{j,k}(u)=\sum_{x\in \mathcal{N}}p_{j,k}(e,x)\mathrm{e}^{\pi(x)\cdot u}.$$
We then let $F(u)$ be the matrix with entries $F_{j,k}(u)$.
This matrix is analogous to the matrix $F(u)$ in \cite{Dussaule} and \cite{DGGP}, which is only defined for abelian groups (that is, we do not need the projection $\pi$ in there).

By definition, the entries of the $n$th power $F(u)^n$ of $F(u)$ are given by
$$F(u)^n_{j,k}=\sum_{x\in \mathcal{N}}p_{j,k}^{(n)}(e,x)\mathrm{e}^{\pi(x)\cdot u},$$
where $p_{j,k}^{(n)}(e,x)$ is the $n$th power convolution of $p_{j,k}(e,x)$.
Recall that we assume that $\mu(e)>0$.
In particular, the matrix $F(u)$ is strongly irreducible, meaning that there exists $n$ such that all the entries of $F(u)^n$ are positive.
The Perron-Frobenius Theorem (see \cite[Theorem~1.1]{Seneta}) shows that $F(u)$ has a dominant eigenvalue, which we denote by $\lambda(u)$.

According to \cite[Proposition~3.5]{Dussaule}, the function $u\mapsto \lambda(u)$ is strongly convex where it is defined.
The proof is only given for the abelian case (that is $\mathcal{N}$ is abelian and we de not need to use the projection $\pi$).
However, the exact same proof applies in our situation.

In \cite{Margulis}, Margulis shows that positive harmonic functions on $\mathcal{N}$ are constant on left cosets of $\mathcal{N}'$ and thus define positive harmonic functions on the abelianization $\mathcal{A}$ for the induced transition kernel.
This result is still true for graphs with polynomial growth, according to \cite[Section~25]{Woess-book}.
Moreover, the minimal harmonic function on $\mathcal{N}$ are in one-to-one correspondence with the minimal harmonic function on $\mathcal{A}$.
We can actually be very precise about the minimal harmonic functions on $\mathcal{N}$.

\begin{lemma}\label{minimalharmonicnilpotent}
The $t$-minimal harmonic functions for the transition kernel $p_{H,\eta,r}$ on $N_\eta(H)=\mathcal{N}\times \{1,...,N_{\eta}\}$ are exactly the functions of the form
$$(x,k)\in \mathcal{N}\times \{1,...,N_{\eta}\}\mapsto C_k\mathrm{e}^{\pi(x)\cdot u},$$
where $u\in \R^d$ satisfies $\lambda(u)=1/t$ and where the vector with entries $C_k$ is an eigenvector associated with $\lambda(u)$.
\end{lemma}

\begin{proof}
To simplify notations, if $h$ is a function on $\mathcal{N}\times \{1,...,N_\eta\}$, we set
$$h_j(x)=h((x,j)).$$
If $(x_0,j_0)$ is fixed and if $h$ is $t$-harmonic, then
\cite[Theorem~25.8]{Woess-book} states that the function
$x\in \mathcal{N}\mapsto \frac{h_{j_0}(xx_0)}{h_{j_0}(x_0)}$ is an exponential,
that is, there exists $u\in \R^d$ such that for all $x\in\mathcal{N}$
$$\frac{h_{j_0}(xx_0)}{h_{j_0}(x_0)}=\mathrm{e}^{\pi(x)\cdot u}.$$
In particular, for every $j$, for every $x$, we have
$$h_{j}(x)=h_j(e)\mathrm{e}^{\pi(x)\cdot u}$$
for some $u\in \R^d$.
To find all possible $u$, let us write the $t$-harmonicity condition.
We have
\begin{equation}\label{harmonicitynilp}
\begin{split}
    h_j(e)&=t\sum_{x\in \mathcal{N}}\sum_{1\leq k \leq N}p_{j,k}(e,x)h_k(e)\mathrm{e}^{\pi(x)\cdot u}\\
    &=t\sum_{1\leq k \leq N}\left (\sum_{x\in \mathcal{N}}p_{j,k}(e,x)\mathrm{e}^{\pi(x)\cdot u}\right )h_k(e).
\end{split}
\end{equation}
Thus, the vector of $\R^{N_\eta}$ with entries $h_j(e)$ is an eigenvector for $F(u)$ associated with the eigenvalue $1/t$.
Since this vector has non-negative entries, the corresponding eigenvalue is necessarily the dominant eigenvalue $\lambda(u)$, according to the Perron-Frobenius Theorem \cite[Theorem~1.1]{Seneta}.
Thus, $\lambda(u)=1/t$.
We proved that as a set, the $t$-minimal Martin boundary is included in this set of functions.
Moreover, the topology of pointwise convergence coincide with the Euclidean topology on $\{u,\lambda(u)=1/t\}$.
Thus, the $t$-minimal Martin boundary can be seen as a Borelian subset of $\{u,\lambda(u)=1/t\}$.

Conversely, let $h$ be a positive function as in the statement of the lemma.
Then, we necessarily have $h_j(e)=C_j$.
Moreover,~(\ref{harmonicitynilp}) is satisfied, so
$$h_j(x)=th_j(e)\mathrm{e}^{ \pi(x)\cdot u}=t\sum_{x'\in \mathcal{N}}\sum_{1\leq k \leq N}p_{j,k}(e,x')h_k(e)\mathrm{e}^{\pi(xx')\cdot u}.$$
Let us sum over $x''=xx'$ to obtain
$$h_j(x)=t\sum_{x''\in \mathcal{N}}\sum_{1\leq k \leq N}p_{j,k}(e,x^{-1}x'')h_k(e)\mathrm{e}^{ \pi(x'')\cdot u}=t\sum_{x,k}p_{j,k}(x,x'')h_k(x'').$$
Thus, $h$ is $t$-harmonic.
Let us prove that it is minimal.
To do so, we use the Martin representation theorem and we write
$$h_j(x)=\int_{\{u,\lambda(u)=1/t\}}C_j(v)\mathrm{e}^{\pi(x)\cdot v}d\nu(v)$$
for some finite measure $\nu$ on $\{u,\lambda(u)=1/t\}$.
For every $x\in \mathcal{N}$,
$$1=\int_{\{u,\lambda(u)=1/t\}}\frac{C_j(v)}{h_j(e)}\mathrm{e}^{\pi(x)\cdot (v-u)}d\nu(v).$$
If the support of $\nu$ were not reduced to $u$, then letting $x$ vary, we would have that this integral is not bounded, which would be a contradiction.
Thus, the support of $\nu$ is reduced to $u$ and so $h$ is minimal.
This concludes the proof.
\end{proof}

We arbitrarily choose a word distance $d_{\mathcal{N}}$ on $\mathcal{N}$.
Following \cite{DGGP}, for $M\geq 0$, the transition kernel $p_{H,\eta,r}$ is said to have exponential moments up to $M$ if for every $k,k'\in \{1,...,N_{\eta}\}$,
$$\sum_{x\in \mathcal{N}}p_{k,k'}(e,x)\mathrm{e}^{Md_{\mathcal{N}}(e,x)}<+\infty.$$
It is proved in \cite{DGGP} that up to enlarging $\eta$, $p_{H,\eta,1}$ has exponential moments up to $M$, where $\eta$ depends on $M$.
The proof crucially uses the fact that $1<R$.
It actually applies for $r<R$, but we will need the same result at the spectral radius, so we have to produce a new proof.
This new proof will use the new deviation inequalities from Section~\ref{Sectiondeviationinequalities}.
We insist on the fact that we do not need to assume that parabolic subgroups are virtually nilpotent in this lemma.
This could be useful in another context.

\begin{lemma}\label{exponentialmomentsinducedchain}
Let $M\geq 0$.
There exists $\eta_M\geq 0$ such that for every $\eta\geq \eta_M$, for every $r\leq R$, $p_{H,\eta,r}$ has exponential moments up to $M$.
\end{lemma}

\begin{proof}
We will use the following claim, given by \cite[Corollary~5.13]{DGGP}.
For $g_1,g_2\in \Gamma$, we let $g_1'$ and $g_2'$ be there projection on $H$.
Then,
\begin{equation}\label{contractionprojection}
    d(g'_1,g'_2)\leq d(g_1,g_2) + c_0.
\end{equation}

Let $x\in \mathcal{N}$ and $j,j'\in \{1,...,N_{\eta}\}$.
Consider a trajectory for the random walk starting at $g_s=(e,j)$, ending at $g_e=(x,j)$ and staying outside $N_{\eta}(H)$.
We fix some $a_0$ that will be chosen large enough later.
Let $g_0$ be the last point on this trajectory whose projection $x_0$ on $\mathcal{N}$ is at distance between $a_0$ and $2a_0$ from $x$.
Choosing a geodesic from $g_0$ to $x_0$ in $\Gamma$, let $\tilde{g}_0$ be the point at distance $\eta+1$ from $x_0$ on this geodesic.
Our goal is first to give an upper-bound for the contribution in $p_{H,\eta,r}$ of all trajectories with fixed $g_0$.

We thus fix $g_0$ and we also fix $\tilde{g}_0$ for now, with the same notations as above.
Let $\alpha$ be a trajectory as above and let $\alpha'$ be the sub-trajectory from $g_s$ to $g_0$.
Let $g_1$ be the last point of $\alpha'$ which stays inside the ball $B(x_0,2\eta)$.
Notice that $d(g_s,x_0)\leq 2a_0+\eta$, so that if $\eta$ is large enough, then $g_s\in B(x_0,2\eta)$.

Also, if $a_0$ is large enough, then \cite[Lemme~1.13~(2), Lemme~1.15~(2)]{Sisto} shows that given any two points $\hat{g},\hat{g}'$ in $\Gamma$, projecting on $\mathcal{N}$ on points $\hat{x}$ and $\hat{x}'$ satisfying $d(\hat{x},\hat{x}')\geq a_0/2$, any relative geodesic from $\hat{g}$ to $\hat{g}'$ passes within a uniformly bounded distance of $\hat{x}$ and of $\hat{x}'$.
We did not recall the definition of relative geodesic for sake of clarity, but according to \cite[Proposition~8.13]{Hruska}, a point on a relative geodesic is within a uniformly bounded distance of a transition point on an actual geodesic.
In particular, Proposition~\ref{Floydgeo} implies that $\delta_{\hat{x}}(\hat{g},\hat{g}')\geq \delta$ and $\delta_{\hat{x}'}(\hat{g},\hat{g}')\geq \delta$ for some uniform $\delta>0$.
We deduce that if $\delta^f_{x_0}(g_1,\tilde{g}_0)\leq 1/2\eta$ and if $\eta$ is large enough, then any projection of $g_1$ on $\mathcal{N}$ is at distance at most $a_0/2$ from $x_0$.

We denote by $W_0$ the set of trajectories as above satisfying $\delta^f_{x_0}(g_1,\tilde{g}_0)\leq 1/2\eta$ and by $W_0'$ those satisfying $\delta^f_{x_0}(g_1,\tilde{g}_0)>1/2\eta$.

We first give an upper-bound for the contribution of trajectories in $W_0$.
Every trajectory in $W_0$ passes through a point $g_0'$ within distance at most $2\eta$ of a point $x_0'$ in $\mathcal{N}$ which is itself within distance at most $a_0/2$ of $x_0$.
Since all the trajectories we are considering stay outside $\mathcal{V}_{\eta}(H)$, the distance formula \cite[Theorem~3.1]{Sisto} shows that they stay outside $B(x_0,\eta/\lambda)$ for some $\lambda$.
Hence, according to Proposition~\ref{propsuperexponentialdecay}, we can bound the total contribution of the sub-trajectories from $g_s$ to $g_0$ by $\mathrm{e}^{-h_0(\eta)}$ for some super-linear function $h_0$.
Since the Green function $G_r(g_s,g_0')$ is decaying at most exponentially fast in $\eta$, we can bound the whole contribution of these sub-trajectories by $\mathrm{e}^{-h_1(\eta)}G(g_s,g_0')$, where $h_1$ is also super-linear.
Decomposing a trajectory according to its first visit at $g_0'$, as in Proposition~\ref{strengthenedweakAncona}, we bound the total contribution of trajectories in $W_0'$ by
$$\sum_{g_0'}\mathrm{e}^{-h_1(\eta)}G_r(g_s,g_0')G_r(g_0',g_0)\lesssim \mathrm{e}^{-h_1(\eta)}G_r(g_s,g_0).$$

We now deal with trajectories in $W_0'$.
Consider such a trajectory and let $g_2$ be its last point inside $B(x_0,4\eta)$.
Either $\delta^f_{x_0}(g_2,\tilde{g}_0)\leq 1/4\eta$ or $\delta^f_{x_0}(g_2,\tilde{g}_0)>1/4\eta$.
We denote by $W_1$, respectively $W_1'$, the subset of $W_0'$ consisting of trajectories satisfying the first, respectively second condition.
The triangle inequality implies that trajectories of $W_1$ satisfy $\delta_{x_0}^f(g_1,g_2)\geq 1/4\eta$.
According to~(\ref{strengthenedsuperexponentialdecay2}), 
whenever $g_1$ and $g_2$ are fixed, the whole contribution of these sub-trajectories from $g_1$ to $g_2$ is bounded by
$\mathrm{e}^{-h_2(2\eta)}$ for some super-linear function $h_2$ and since the Green function $G_r(g_1,g_2)$ is decaying at most exponentially fast in $\eta$, we bound the whole contribution of these sub-trajectories by
$\mathrm{e}^{-h_3(2\eta)}G_r(g_1,g_2)$, where $h_3$ is also super-linear.
Decomposing a trajectory from $g_s$ to $g_0$ according to its visits at $g_1$ and $g_2$ as we did above with $g_0'$, we bound the total contribution of trajectories in $W_1$ by
$\mathrm{e}^{-h_3(2\eta)} G_r(g_s,g_0)$.
Finally, note that balls of radius $r$ in $\Gamma$ have cardinality at most exponential in $r$.
Consequently, summing over every possible $g_1$ and $g_2$, we get a contribution which is at most $\mathrm{e}^{-h_4(2\eta)}G(g_s,g_0)$ for some super-linear function $h_4$.

We now focus on trajectories of $W_1'$.
There are again two possibilities.
Let $g_3$ be the last point in $B(x_0,8\eta)$.
Either $\delta^f_{x_0}(g_3,\tilde{g}_0)\leq 1/8\eta$ and we denote by $W_2$ the corresponding subset of $W_1'$, or $\delta^f_{x_0}(g_3,\tilde{g}_0)>1/8\eta$ and we denote by $W_2'$ the corresponding subset of $W_1'$.
We bound the contribution of trajectories of $W_2$ by $\mathrm{e}^{-h_4(4\eta)}G_r(g_s,g_0)$ as above.
We go on and define by induction a sequence of sets $W_k$ and $W_k'$ as well as a sequence of points $g_k$ as follows.
We let $g_{k+1}$ be the last possible point of a trajectory in the ball $B(e,2^{k+1}\eta)$. We let $W_k$ be the set of trajectories in $W_{k-1}'$ such that $\delta^f_{x_0}(g_{k+1},\tilde{g}_0)\leq 2^{-k}\eta$.
Finally, we let $W_k'$ be the set of trajectories in $W_{k-1}'$ such that $\delta^f_{x_0}(g_{k+1},\tilde{g}_0)> 2^{-k}\eta$.
The contribution of trajectories in $W_k$ is then at most $\mathrm{e}^{-h_4(2^k\eta)}G_r(g_s,g_0)$.
Also notice that $W_0'=\cup_{k\geq 1}W_k$.
Hence, the whole contribution of trajectories of $W_0'$ is at most
$$\sum_{k\geq 1}\mathrm{e}^{-h_4(2^k\eta)}G_r(g_s,g_0)\leq \mathrm{e}^{-h_5(\eta)}G_r(g_s,g_0).$$
for some super-linear function $h_5$.

We thus proved that every trajectory from $g_s$ to $g_e$ staying outside $N_{\eta}(H)$, whether it is in $W_0$ or in $W_0'$, projects on $\mathcal{N}$ on a point $x_0$ satisfying that the distance $d(e,x_0)$ is between $a_0/2$ and $2a_0$.
Moreover, whenever $g_0$ is fixed, the contribution of every sub-trajectory from $g_s$ to $g_0$ is at most $\mathrm{e}^{-h(\eta)}G_r(g_s,g_0)$ for some super-linear function $h$ (which is the infimum between $h_1$ and $h_5$).

We now set $g^{(0)}=g_0$.
Similarly, every trajectory passes through a point whose projection on $\mathcal{N}$ is at a distance between $5a_0/2$ and $4a_0$ of $e$.
We denote by $g^{(1)}$ the last such point and by $x_1$ the corresponding projection on $\mathcal{N}$.
Whenever $g^{(0)}$ and $g^{(1)}$ are fixed, the contribution of sub-trajectories from $g^{(0)}$ to $g^{(1)}$ is at most $\mathrm{e}^{-h(\eta)}G_r(g^{(0)},g^{(1)})$.
We construct a similar finite sequence of points $g^{(0)},...,g^{(l)}$ in $\Gamma$ projecting on points $x_1,...,x_l$ on $\mathcal{N}$ satisfying that $d(e,x_k)$ is between $(2k+1/2)a_0$ and $2(k+1)a_0$ and where $g^{(k)}$ is the last such point on the trajectory.
Since $\mathcal{N}$ is quasi-isometrically embedded in $\Gamma$, we necessarily have $l\geq \Lambda d_{\mathcal{N}}(e,x)$.
Moreover, the contribution of sub-trajectories from $g^{(k)}$ to $g^{(k+1)}$ is at most $\mathrm{e}^{-h(\eta)}G_r(g^{(k)},g^{(k+1)})$.

Decomposing any trajectory from $g_s$ to $g_e$ according to its positions at the points $g^{(k)}$, we get
$$G_r(g_s,g_e)=\sum_{g^{(0)},...,g^{(l)}}G_r(g_s,g^{(0)})G_r(g^{(0)},g^{(1)})...G_r(g^{(l)},g_e),$$
where the sum is over all possible $g^{(0)},...,g^{(l)}$ as above.
We bound the total contribution of trajectories staying outside $N_{\eta}(H)$ by
$$\mathrm{e}^{-lh(\eta)}\sum_{g_0,...,g_l}G_r(g_s,g^{(0)})G_r(g^{(0)},g^{(1)})...G_r(g^{(l)},g_e)\leq  \mathrm{e}^{-lh(\eta)}G_r(g_s,g_e).$$
Up to choosing $\eta$ large enough, this last quantity is bounded by
$\mathrm{e}^{-\alpha d_{\mathcal{N}}(g,g')}$
where $\alpha$ can be chosen arbitrarily small.
We thus get
$$\sum_{x\in \mathcal{N}}p_{j,j'}(e,x)\mathrm{e}^{Md_{\mathcal{N}}(e,x)}\leq \sum_{n\geq 0}\sum_{g,d_{\mathcal{N}}(e,g)=n}\mathrm{e}^{(-\alpha+M)n}.$$
Again, balls of radius $r$ in $\mathcal{N}$ have cardinality at most exponential in $r$, so that we can choose $\alpha$ small enough so that this last sum is finite.
In other words, we can choose $\eta$ large enough so that $p_{H,\eta,r}$ has exponential moments up to $M$.
\end{proof}

\begin{lemma}\label{Martinharmonicnilpotent}
There exists $\eta_0$ such that for every $\eta\geq \eta_0$ and for every $r\leq R$, every point in the 1-Martin boundary of the transition kernel $p_{H,\eta,r}$ is a positive harmonic function for $p_{H,\eta,r}$.
\end{lemma}

\begin{proof}
In \cite{GGPY}, the authors prove the following.
Whenever a transition kernel has super-exponential moments, every Martin kernel is a positive harmonic function.
In our situation, up to choosing $\eta$ large enough, the transition kernel we study has exponential moments up to $M$, for arbitrarily large $M$.
We show that this is enough to conclude, adapting the proof of \cite{GGPY}.

First,
Harnack inequalities show that there exists some uniform $C>1$ such that for every $g,g'\in N_\eta(H)$,
$$K_{H,\eta,R}(g,g')=K_R(g,g')\leq C^{d(e,\gamma)}.$$
Let us fix $M>2\log C$ and let us fix $\eta_0$ such that $p_{\eta,H,R}$ has exponential moments up to $M$ for every $\eta \geq \eta_0$.
For every $g \in N_{\eta}(H)$,
$$d(e,g)\leq \lambda d_{\mathcal{N}}(e,g)+\alpha \eta,$$
for some constants $\lambda$ and $\alpha$.
Up to changing $M$, we see that for every $g\in N_{\eta}(H)$,
\begin{equation}\label{HarnackMartinharmonic}
    K_R(g,g')\leq C_{\eta}C^{d_{\mathcal{N}}(e,g)}
\end{equation}
and we still have $M>2\log C$.

We introduce the linear operator $P$ associated with the transition kernel $p_{H,\eta,R}$.
That is, for every function $\omega$ on $N_{\eta}(H)$,
$$P\omega (g)=\sum_{g'\in N_{\eta}(H)}p_{H,\eta,R}(g,g')\omega (g').$$
We want to prove that for every $\xi$ in the 1-Martin boundary of $p_{H,\eta,R}$, we have $PK(\cdot,\xi)=K(\cdot,\xi)$.
Let $g_n$ be a sequence of $N_{\eta}(H)$ converging to $\xi$.
We need to prove that $PK(\cdot,g_n)$ both converges to $K(\cdot,\xi)$ and to $PK(\cdot,\xi)$.

By definition, the Green function $G_{H,\eta,R}(\cdot,g')$ is 1-harmonic for $p_{H,\eta,R}$ everywhere except at $g'$.
Fixing $g$, for $n$ large enough, $PG_{H,\eta,R}(g,g_n)=G_{H,\eta,R}(g,g_n)$, since $g_n$ tends to infinity.
Thus, for large enough $n$, $PK_{H,\eta,R}(g,g_n)=K_{H,\eta,R}(g,g_n)$ and so $PK(g,g_n)$ converges to $K(\gamma,\xi)$.

We now prove that $PK(\cdot,g_n)$ also converges to $PK(\cdot,\xi)$.
Let
$$\tilde{K}_n=|K(\cdot,g_n)-K(\cdot,\xi)|.$$
We just need to show that $P\tilde{K}_n$ converges to 0.
Let us fix $\epsilon>0$.
According to~(\ref{HarnackMartinharmonic}),
$$\tilde{K}_n(g)\leq 2C_{\eta}C^{d_{\mathcal{N}}(e,x)},$$
where $x$ is the projection of $g$ on $\mathcal{N}$.
Let us fix $g$ and let us fix $L>d_{\mathcal{N}}(e,x)$.
Decompose the transition kernel $p_{H,\eta,R}$ as $p_{H,\eta,R}=p_L+q_L$, where $p_L$ is its restriction to $B_{\mathcal{N}}(e,L)=\{g \in N_{\eta}(H),d_{\mathcal{N}}(e,x)\leq L\}$.
Denote by $P_L$ and $Q_L$ the associated linear operators, so that $P=P_L+Q_L$.
Then,
$$Q_L\tilde{K}_n(g)\leq \sum_{g',d_{\mathcal{N}}(x,x')> L}p_{H,\eta,R}(g,g')2C_{\eta}C^{d_{\mathcal{N}}(e,x')}.$$
Changing the sum over $g'$ with a sum over $g^{-1}g'$, we get
$$Q_L\tilde{K}_n(g)\leq \sum_{g',d_{\mathcal{N}}(e,x')>L}p_{H,\eta,R}(e,g')2C_{\eta}C^{d_{\mathcal{N}}(e,x)+d_{\mathcal{N}}(e,x')}$$
Recall that $d_{\mathcal{N}}(e,x)<L$.
Consequently,
$$Q_L\tilde{K}_n(g)\leq 2C_{\eta}\sum_{\gamma',d_{\mathcal{N}}(e,g')> L}p_{H,\eta,R}(e,\gamma')C^{2d_{\mathcal{N}}(e,g')}.$$
Since $M>2\log C$ and since $p_{H,\eta,R}$ has exponential moments up to $M$, this last sum goes to 0 when $L$ tends to infinity.
We can choose $L$ large enough so that $Q_L\tilde{K}_n(g)\leq \epsilon$.
Now that $L$ is fixed,
$P_L\tilde{K}_n(g)$ tends to 0 when $n$ tends to infinity, since the sum in the definition of $P_L$ is finite.
For $n$ large enough, $P_L\tilde{K}_n(g)\leq \epsilon$, so that $P\tilde{K}_n(g)\leq 2\epsilon$.
This proves that $P\tilde{K}_n(g)$ converges to 0 when $n$ tends to infinity and this is true for every $g$.
This concludes the proof.
\end{proof}

Let $R_{\eta,r}(H)$ be the spectral radius of the transition kernel $p_{H,\eta,r}$.
In particular, $R_{0,R}(H)=R_H$.

\begin{lemma}\label{spectralradiusnilpotent}
There exists $\eta_0$ such that for every $\eta\geq \eta_0$ and for every $r\leq R$, if $R_{\eta,r}(H)=1$, then $\min_u\lambda(u)=1$.
\end{lemma}

\begin{proof}
According to \cite[Lemma~7.2]{Woess-book}, if there exists a $t$-harmonic positive function, then $t\leq R_{\eta,r}(H)$.
Assume that $R_{\eta,r}(H)=1$.
Lemma~\ref{Martinharmonicnilpotent} shows that there exists a 1-harmonic function.
This proves that 1 is the maximum of the $t$ such that there exists a $t$-harmonic positive function.

According to the Martin representation theorem, there exists a $t$-harmonic function if and only if the $t$-minimal Martin boundary is non-empty.
Lemma~\ref{minimalharmonicnilpotent} shows that there exists a $t$-minimal harmonic function if and only if there exists $u\in \R^d$ such that $\lambda(u)=1/t$.
Thus, $1=\max_u\frac{1}{\lambda(u)}$.
This proves the lemma.
\end{proof}

Assuming that $\mu$ is spectrally degenerate along $H$, we have $R_{0,R}(H)=1$.
We show that we also have $R_{\eta,R}(H)=1$ for every $\eta\geq 0$.

\begin{lemma}\label{lemmaassumption2firstpart}
Let $0\leq \eta \leq \eta'$ and let $r\leq R$.
If $R_{\eta,r}(H)=1$, then $R_{\eta',r}(H)=1$.
\end{lemma}

\begin{proof}
Let $G_{H,\eta,r}(\cdot,\cdot|t)$ be the Green function at $t$, associated with the transition kernel $p_{H,\eta,r}$.
We show that
\begin{equation}\label{inequalityGreenneighborhoods}
    G_{H,\eta,r}(e,e|t)\leq G_{H,\eta',r}(e,e|t).
\end{equation}
By definition,
$$G_{H,\eta,r}(e,e|t)=\sum_{n\geq 0}t^np_{H,\eta,r}^{(n)}(e,e).$$
In other words, $G_{H,\eta,r}(e,e|t)$ is the sum of the weight of every path from $e$ to $e$ that stays inside $N_\eta(H)$, multiplied with $t$ at the power the length of this path.
Here, the weight is the power of convolution of $p_{H,\eta,r}$.

By definition of this transition kernel, it is equivalent to sum the weight of every path from $e$ to $e$ (not necessarily staying inside $N_{\eta}(H)$, multiplied by $t$ at the power the number of points on this path which are in $N_\eta(H)$.
Here, the weight is the power of convolution of $r\mu$.
Formally, we write this as
$$G_{H,\eta,r}(e,e|t)=\sum_{n\geq 0}\sum_{g_1,...,g_{n-1}}t^{\mathrm{Card}\{j,g_j\in N_{\eta}(H)\}}r^n\mu(g_1)\mu(g_1^{-1}g_2)\cdots \mu(g_{n-1}^{-1}).$$
Since $N_\eta(H)\subset N_{\eta'}(H)$, we have~(\ref{inequalityGreenneighborhoods}) and the lemma follows.
\end{proof}

We can now use all the lemmas of this section to prove the following proposition.

\begin{proposition}\label{propsumupdegenerateparabolic}
Assume that $\mu$ is spectrally degenerate along $H$.
Then, there exists $\eta_0$ such that for every $\eta \geq \eta_0$,
the 1-Martin boundary of the transition kernel $p_{H,\eta,R}$ is reduced to a point.
\end{proposition}

\begin{proof}
We first prove that the minimal Martin boundary of  $p_{H,\eta,R}$ is reduced to a point.
Indeed, according to Lemma~\ref{lemmaassumption2firstpart}, for every $\eta$, $R_{\eta,R}(H)=1$.
Thus, Lemma~\ref{spectralradiusnilpotent} shows that $\min \lambda(u)=1$.
Since $\lambda$ is strictly convex, this proves that $\lambda(u)=1$ has only one solution $u_0$.
Finally, Lemma~\ref{minimalharmonicnilpotent} shows that there is only one positive minimal harmonic function and so the minimal Martin boundary is reduced to a point.

According to Lemma~\ref{Martinharmonicnilpotent}, if $\eta \geq \eta_0$, then the Martin kernels are positive harmonic functions for $p_{H,\eta,R}$.
Since every positive harmonic function can be represented as an integral of the Martin kernel over the minimal boundary for some finite measure, this proves that all Martin kernels are proportional.
Since they are all equal to 1 at $e$, this proves there is only one Martin kernel, which concludes the proof.
\end{proof}

We can now prove Proposition~\ref{propconvergenceparabolic1}.
\begin{proof}
Let us fix $\eta \geq \eta_0$, so that we can apply Proposition~\ref{propsumupdegenerateparabolic}.
Let $g_n$ converge to a degenerate parabolic limit point $\xi$.
Let $H$ be the stabilizer of $\xi$ and let $H_0\in \Omega_0$ be such that $H=gH_0g^{-1}$.
Then, by definition, $\mu$ is spectrally degenerate along $H_0$.

Let $\pi_n=\pi_{gH_0}(g_n)$ be the projection of $g_n$ on $gH_0$.
Since $g_n$ converges to $\xi$, the sequence $g^{-1}\pi_n$ goes to infinity in $H_0$.
In particular, according to Proposition~\ref{propsumupdegenerateparabolic}, $g^{-1}\pi_n$ converges in the 1-Martin boundary of $p_{H,\eta,R}$.
That is, for every $h\in H$,
$$K_{H,\eta,R}(h, g^{-1}\pi_n)=\frac{G_{H,\eta,R}(h,g^{-1}\pi_n)}{G_{H,\eta,R}(e,g^{-1}\pi_n)}$$
converges to some limit.

More generally, if $h\in \Gamma$ is fixed, then $h\in N_\eta(H)$ for some $\eta$, so that Proposition~\ref{propsumupdegenerateparabolic} also shows that
$K_{H,\eta,R}(h, g^{-1}\pi_n)$ converges to some limit.
Recall that the Green function is invariant by left multiplication, so this also proves that for every $h\in \Gamma$,
$K_{H,\eta,R}(h,\pi_n)$ converges to some limit.
Moreover, according to Lemma~\ref{sameGreen}, the Martin kernel associated with $p_{H,\eta,R}$ at 1 and the Martin kernel associated with $\mu$ at $R$ coincide.
This proves that the limit of $K_{H,\eta,R}(h,\pi_n)$ does not depend on $H$ and $\eta$.
We denote this limit by $K_R(h,\xi)$.
This also proves that $K_R(h,\pi_n)$ converges to $K_R(h,\xi)$.

To conclude,
we want to prove that $K_R(h,g_n)$ also converges to the same limit.
We fix $h\in \Gamma$ and we fix $\epsilon>0$.
Since $\pi_n$ goes to infinity, the distance between the projection of $h$ on $H_0$ and $\pi_n$ also goes to infinity.
As above, \cite[Lemma~1.13~(2), Lemma~1.15~(2)]{Sisto} and \cite[Proposition~8.13]{Hruska} show that $\pi_n$ is within a bounded distance of a transition point, in the sense of Definition~\ref{deftransitionusual}, on a geodesic from $h$ to $g_n$.
Proposition~\ref{Floydgeo} then shows that $\delta^{f}_{\pi_n}(h,g_n)>\delta$ for some uniform $\delta$ that does not depend on $n$.
Corollary~\ref{Anconaepsilon} shows that there exists $S$ such that
$$G_R(h,g_n;B_S(\pi_n)^c)\leq \epsilon G_R(h,g_n).$$
Similarly, for large enough $n$, we have
$$G_R(e,g_n;B_S(\pi_n)^c)\leq \epsilon G_R(e,g_n).$$
We decompose a path from $h$ to $g_n$ according to its last visit to $B_S(\pi_n)$.
We have
$$G_R(h,g_n)=G_R(h,g_n;B_S(\pi_n)^c)+\sum_{u\in B_S(e)}G_R(h,\pi_nu)G_R(\pi_nu,g_n),$$
so that
\begin{equation}\label{equationconvergencedegenerateparabolic1}
    (1-\epsilon)G_R(h,g_n)\leq \sum_{u\in B_S(e)}G_R(h,\pi_nu)G_R(\pi_nu,g_n)
\end{equation}
and similarly
\begin{equation}\label{equationconvergencedegenerateparabolic2}
    (1-\epsilon)G_R(e,g_n)\leq \sum_{u\in B_S(e)}G_R(e,\pi_nu)G_R(\pi_nu,g_n).
\end{equation}
We also have
\begin{equation}\label{equationconvergencedegenerateparabolic3}
    G_R(h,g_n)\geq \sum_{u\in B_S(e)}G_R(h,\pi_nu)G_R(\pi_nu,g_n)
\end{equation}
and
\begin{equation}\label{equationconvergencedegenerateparabolic4}
    G_R(e,g_n)\geq \sum_{u\in B_S(e)}G_R(e,\pi_nu)G_R(\pi_nu,g_n).
\end{equation}
Moreover, for every $u\in B_S(e)$, the sequence $\pi_nu$ lies in $N_{\eta}(H)$ and converges to $\xi$.
The above discussion shows that $K_R(h,\pi_nu)$ converges to $K_R(h,\xi)$.
Thus for large enough $n$, for every $u\in B_S(e)$,
$$(K_R(h,\xi)-\epsilon)G_R(e,\pi_nu)\leq G_R(h,\pi_nu)\leq (K_R(h,\xi)+\epsilon)G_R(e,\pi_nu).$$
In particular,~(\ref{equationconvergencedegenerateparabolic1}) shows that
\begin{equation}\label{equationconvergencedegenerateparabolic5}
    (1-\epsilon)G_R(h,g_n)\leq (K_R(h,\xi)+\epsilon) \sum_{u\in B_S(e)}G_R(e,\pi_nu)G_R(\pi_nu,g_n).
\end{equation}
Combining~(\ref{equationconvergencedegenerateparabolic4}) and~(\ref{equationconvergencedegenerateparabolic5}), we get
$$(1-\epsilon)G_R(h,g_n)\leq (K_R(h,\xi)+\epsilon)G_R(e,g_n)$$
and so
$$K_R(h,g_n)\leq \frac{1}{1-\epsilon}(K_R(h,\xi)+\epsilon).$$
Similarly,~(\ref{equationconvergencedegenerateparabolic2}) and~(\ref{equationconvergencedegenerateparabolic3}) yield
$$K_R(h,g_n)\geq (1-\epsilon)(K_R(h,\xi)-\epsilon).$$
To sum-up, for large enough $n$, we have
$$(1-\epsilon)(K_R(h,\xi)-\epsilon)\leq K_R(h,g_n)\leq  \frac{1}{1-\epsilon}(K_R(h,\xi)+\epsilon).$$
Since $\epsilon$ is arbitrary, this concludes the proof.
\end{proof}

\subsection{Parabolic limit points: the non-degenerate case}\label{Sectionparaboliclimitpointsnondegenerate}
We now focus on non-degenerate parabolic points.
We do assume that the parabolic subgroups are virtually abelian here.
We prove the following.
If $g_n$ converges to a point in the geometric boundary of a non-degenerate parabolic subgroups, then $g_n$ converges to a point in the $R$-Martin boundary.
We also prove that for $r<R$, $g_n$ converges to a point in the $r$-Martin boundary, whether the parabolic subgroup is non-degenerate or not.
The formal statement is a bit technical.

\begin{proposition}\label{propconvergenceparabolic2}
Let $\Gamma$ be a non-elementary relatively hyperbolic group with respect to a collection $\Omega$ of virtually abelian subgroups.
Choose a finite set $\Omega_0$ of representatives of conjugacy classes of elements of $\Omega$.
Let $\mu$ be a symmetric probability measure whose finite support generates $\Gamma$ and let $R$ be its spectral radius and $r\leq R$.
Let $\xi$ be a parabolic limit point and let $g\in \Gamma$ and $H\in \Omega_0$ be such that $gHg^{-1}$ is the stabilizer of $\xi$.
Let $g_n$ be a sequence of elements of $\Gamma$ such that the projection $\pi_{gH}(g_n)$ of $g_n$ on $gH$ satisfies that
$g^{-1}\pi_{gH}(g_n)$ converges to the geometric boundary of $H$.
Finally, assume either that $r<R$ or that $\mu$ is spectrally non-degenerate along $H$.
Then, $g_n$ converges to a point in the $r$-Martin boundary.
\end{proposition}

Our goal is to prove this proposition.
Let $H$ be a parabolic subgroup, so that $H$ is virtually abelian.
There is a finite index subgroup of $H$ which is isomorphic to $\Z^d$.
As in Section~\ref{Sectionparaboliclimitpointsdegenerate}, any section $H/\Z^d\to H$ identifies $H$ as a set with $\Z^d\times E$, where $E$ is finite.
Moreover, for any $\eta\geq0$,
\begin{enumerate}
    \item $N_{\eta}(H)$ can be $\Z^d$-equivariantly identified with $\Z^d\times \{1,...,N_\eta\}$,
    \item if $\eta\leq \eta'$, then $N_\eta\leq N_{\eta'}$. In other words, the set $\Z^d\times \{1,...,N_\eta\}$, identified with $N_{\eta}(H)$ is a subset of $\Z^d\times \{1,...,N_{\eta'}\}$ identified with $N_{\eta'}(H)$.
\end{enumerate}

We still denote by $p_{H,r}$ the first return kernel to $H$ for $r\mu$
and by $p_{H,\eta,r}$ the first return kernel to $N_{\eta}(H)$ for $r\mu$.
Recall that
$$p_{H,\eta,r}(h,h')=\sum_{n\geq 0}\sum_{\underset{\notin N_\eta(H)}{g_1,...,g_{n-1}}}r^n\mu(h^{-1}g_1)\mu(g_1^{-1}g_2)...\mu(g_{n-2}^{-1}g_{n-1})\mu(g_{n-1}^{-1}h').$$
Again, to simplify notations, whenever $H$, $r$ and $\eta$ are fixed, identifying $h,h''\in H$ with $(x,k)$,$(x',k')$, $x,x'\in \Z^d$ and $k,k'\in \{1,...,N_\eta\}$, we write
$$p_{H,\eta,r}((x,k),(x',k'))=p_{k,k'}(x,x').$$

For $u\in \R^d$, we define for $j,k\in \{1,...,N_{\eta}\}$,
$$F_{j,k}(u)=\sum_{x\in \Z^d}p_{j,k}(e,x)\mathrm{e}^{x\cdot u}.$$
We then let $F(u)$ be the matrix with entries $F_{j,k}(u)$.
This is the same matrix as in Section~\ref{Sectionparaboliclimitpointsdegenerate}, except that we do not need to project $x$ on $\Z^d$, since we are not assuming that the parabolic subgroups are virtually nilpotent, but virtually abelian.

The entries of the $n$th power $F(u)^n$ of $F(u)$ are still given by
$$F(u)^n_{j,k}=\sum_{x\in \Z^d}p_{j,k}^{(n)}(e,x)\mathrm{e}^{x\cdot u},$$
where $p_{j,k}^{(n)}(e,x)$ is the $n$th power convolution of $p_{j,k}(e,x)$.
Moreover, the matrix $F(u)$ is still strongly irreducible and thus has a dominant eigenvalue, which we still denote by $\lambda(u)$.
According to \cite[Proposition~3.5]{Dussaule}, the function $u\mapsto \lambda(u)$ is strongly convex where it is defined.
More precisely, let $\mathcal{F}$ be the interior of the set where all the entries of $F(u)$ are finite.
Then, $u\mapsto\lambda(u)$ is strictly convex on $u$.

In \cite{Dussaule} and \cite{DGGP}, there are two technical assumptions made on $\lambda$, in order to identify the 1-Martin boundary.
Let
$$\mathcal{D}=\{u\in \mathcal{F},\lambda(u)\leq 1\}.$$
Notice that $F$, $\lambda$, $\mathcal{F}$ and $\mathcal{D}$ all depend on $\eta$, but we omit this dependence in the notations.
The two assumptions are as follows.

\begin{hyp}
The set $\mathcal{D}$ is compact.
\end{hyp}

\begin{hyp}
The minimum of $\lambda$ is strictly smaller than 1.
\end{hyp}

We will prove that for large enough $\eta$, these two assumptions are satisfied.
This will allow us to prove Proposition~\ref{propconvergenceparabolic2}.
The first assumption is satisfied as soon as $p_{H,\eta,r}$ has sufficiently large exponential moments.
According to Lemma~\ref{exponentialmomentsinducedchain}, this is true as soon as $\eta$ is large enough.
Precisely, we have the following.

\begin{lemma}\label{lemmaassumption1}\cite[Proposition~4.6]{DGGP}
Let $r\leq R$.
If $\eta$ is large enough, then the transition kernel $p_{H,\eta,r}$ on $\Z^d\times \{1,...,N_{\eta}\}$ satisfies Assumption~1.
\end{lemma}

This settles Assumption~1.
Note that we did not have to assume that $\mu$ is not spectrally degenerate along $H$.
The fact that Assumption~2 is also satisfied for spectrally non-degenerate parabolic subgroups will be derived from the next lemmas.
Recall that we denote by $R_{\eta,r}(H)$ the spectral radius of $p_{H,\eta,r}$.

\begin{lemma}\label{spectralradiusabelian}
There exists $\eta_0$ such that for every $\eta\geq \eta_0$ and for every $r\leq R$, we have $R_{\eta,r}(H)=1$ if and only if $\min_u\lambda(u)=1$.
\end{lemma}

\begin{proof}
The "only if" part is given by Lemma~\ref{spectralradiusnilpotent}.
Let us prove the "if" part.
Assume that $\min_u\lambda(u)=1$ and that $\lambda$ reaches its minimum at $u_0$.
Since $\lambda$ is strictly convex, we have $\nabla \lambda(u_0)=0$
Thus, \cite[Proposition~3.14]{Dussaule} shows that for large enough $\eta$, $p^{(n)}_{H,\eta,r}(e,e)\sim Cn^{-d/2}$.
In particular, $R_{\eta,r}(H)=1$.
\end{proof}

\begin{lemma}\label{lemmaassumption2preparation}
Let $\eta \geq \eta_0$.
Assume that $R_{\eta,r}(H)>1$.
Then, the transition kernel $p_{H,\eta,r}$ satisfies Assumption~2.
\end{lemma}

\begin{proof}
This is a direct consequence of Lemma~\ref{spectralradiusabelian}.
\end{proof}

\begin{lemma}\label{lemmaassumption2r<R}
Let $\eta\geq 0$ and let $r<R$.
Then $R_{\eta,r}(H)>1$.
\end{lemma}

\begin{proof}
Let $\epsilon>0$ be such that $r(1+\epsilon)<R$.
By definition, if $g\neq g'$,
$$p_{H,\eta,r}(g,g')=\sum_{n\geq 1}\sum_{\underset{\notin N_\eta(H)}{g_1,...,g_{n-1}}}r^n\mu(g^{-1}g_1)\mu(g_1^{-1}g_2)...\mu(g_{n-1}^{-1}g').$$
We deduce that
$$(1+\epsilon)p_{H,\eta,r}(g,g')\leq p_{H,\eta,r(1+\epsilon)}(g,g'),$$
so that
$$(1+\epsilon)^np_{H,\eta,r}^{(n)}(g,g')\leq p_{H,\eta,r(1+\epsilon)}^{(n)}(g,g').$$
Denoting by $G_{H,\eta,r}(g,g'|t)$ the Green function associated with $p_{H,\eta,r}$ evaluated at $t$, we thus have
$$G_{H,\eta,r}(g,g'|1+\epsilon)\leq G_{H,\eta,(r+r\epsilon)}(g,g').$$
Since $G_{H,\eta,(r+r\epsilon)}(g,g')=G_{r+r\epsilon}(g,g')<+\infty$ according to Lemma~\ref{sameGreen}, we finally get
that $G_{H,\eta,r}(g,g'|1+\epsilon)$ is finite.
\end{proof}

\begin{lemma}\label{lemmaassumption2secondpart}
Let $\eta\geq 0$.
Then, $R_{\eta,R}(H)=1$ if and only if $R_{0,R}(H)=1$.
\end{lemma}

\begin{proof}
The "if" part is given by Lemma~\ref{lemmaassumption2firstpart}.
Let us prove the converse and assume that $R_{\eta,R}(H)=1$.
Then, according to Lemma~\ref{lemmaassumption2firstpart}, for every $\eta'\geq \eta$, we also have $R_{\eta',R}(H)=1$.
Proposition~\ref{propsumupdegenerateparabolic} shows that for large enough $\eta'$, the 1-Martin boundary of $p_{H,\eta',R}$ is reduced to a point.

We want to prove that the 1-Martin boundary of $p_{H,0,R}$ is also reduced to a point.
Fix $g\in H$ and let $g_n$ be a sequence of $H$ going to infinity.
Then, in particular, $g,g_n\in N_{\eta'}(H)$, so that $K_{H,\eta',R}(g,g_n)$ converges to some limit which does not depend on $g_n$, where $K_{H,\eta',R}$ is the Martin kernel associated with $p_{H,\eta,R}$.
According to Lemma~\ref{sameGreen}, $K_{H,\eta',R}=K_{H,0,R}$.
Consequently, $K_{H,0,R}(g,g_n)$ also converges to some limit, independently of $g_n$.
We deduce that the 1-Martin boundary of $p_{H,0,R}$ is indeed reduced to a point.

According to Lemma~\ref{minimalharmonicnilpotent}, there is only one solution to $\lambda(u)=1$ and since $\lambda$ is strictly convex, 1 is the minimum of $\lambda(u)$.
According to Lemma~\ref{spectralradiusnilpotent}, $R_{0,R}(H)=1$, which concludes the proof.
\end{proof}

\begin{corollary}\label{coroassumption2}
Let $H\in \Omega_0$ be a parabolic subgroup and assume either that $r<R$ or that $\mu$ is spectrally non-degenerate along $H$.
Let $\eta\geq \eta_0$.
Then, the transition kernel $p_{H,\eta,r}$ on $N_\eta(H)$ satisfies Assumption~2.
\end{corollary}

We can finally prove Proposition~\ref{propconvergenceparabolic2}.
\begin{proof}
This is a consequence of \cite[Corollary~4.8]{DGGP}.
However, the formulation is a bit different (the statement refers to sub-Markov chains rather than transition kernels satisfying Assumption~2).
We give the full proof for completeness.

If $\eta$ is large enough, the transition kernel $p_{H,\eta,r}$, which can be seen as a $\Z^d$-invariant kernel on $\Z^d\times \{1,...,N_\eta\}$, satisfies both Assumption~1 and~2.
According to \cite[Lemma~3.24]{Dussaule}, the set $\{u,\lambda(u)=1\}$ is homeomorphic to the sphere $\mathbb{S}^{d-1}$.
An explicit homeomorphism is given by
$$u\in \{u\in \mathbb{R}^d,\lambda(u)=1\}\mapsto \frac{\nabla \lambda(u)}{\|\nabla \lambda (u)\|}.$$
This homeomorphism provides a homeomorphism $\varphi$ between the geometric boundary of $\Z^d\times \{1,...,N_\eta\}$ and $\mathbb{S}^{d-1}$ constructed as follows.
Let $(x_n,j_n)$ be a sequence in $\Z^d\times \{1,...,N_\eta\}$ converging to a point $\tilde{x}$ in the geometric boundary $\partial (\Z^d\times \{1,...,N_\eta\})$.
That is, $x_n$ tends to infinity and $\frac{x_n}{\|x_n\|}$ converges to a point $\theta$ in the unit sphere $\mathbb{S}^{d-1}$.
There exists a unique $u\in \{u\in \mathbb{R}^d,\lambda(u)=1\}$ such that $\theta=\frac{\nabla \lambda(u)}{\|\nabla \lambda (u)\|}$.
Then, define $\varphi(\tilde{x})=u$.

Let $g\in N_\eta(H)$ and let $g_n$ be a sequence of $N_\eta(H)$ converging to a point in the geometric boundary of $H$.
Write $g=(x,j)$ and $g_n=(x_n,j_n)$.
Then, by definition, $x_n$ converges to some $\tilde{x}$ in the geometric boundary of $\Z^d\times \{1,...,N_\eta\}$.
Let $u=\varphi(\tilde{x})$.
Then, \cite[Proposition~3.27]{Dussaule} shows that
$K_{r}(g,g_n)$ converges to $C_{j}\mathrm{e}^{u\cdot x}$.
Here, $C_{j}$ only depends on the coordinate $j$ of $g=(x,j)$.
In particular, $C_j$ does not depend on $j_n$.
Note that since $K_{r}(g,g_n)$ does not depend on $\eta$, its limit $C_{j}\mathrm{e}^{u\cdot x}$ also is independent.
Moreover, letting $x$ vary, we see that $u$ is independent of $\eta$, so that $C_j$ also is independent of $\eta$.

To prove Proposition~\ref{propconvergenceparabolic2}, we now consider a sequence $g_n$ such that its projection $h_n$ on $H$ converges to $\xi$ in the geometric boundary of $H$.
The above discussion shows that for every fixed $g\in \Gamma$ and for every $L>0$, for every $u$ in the ball $B(e,L)$, we have
$$K_r(g,h_nu)\to C_g\mathrm{e}^{u\cdot x(g)},$$
where $u$ only depends on $\xi$ and where $C_g$ only depends on $d(g,H)$ and $x(g)\in \Z^d$ is the projection of $g$ on $\Z^d$.
Exactly like in the proof of Proposition~\ref{propconvergenceparabolic1}, we can use the deviation inequalities from Section~\ref{Sectiondeviationinequalities} to show that
\begin{equation}\label{limitMartinparabolicnondegenerate}
    K_r(g,g_n) \to C_g\mathrm{e}^{u\cdot x(g)}.
\end{equation}
This concludes the proof.
\end{proof}

\subsection{Topology of the Martin boundary}\label{Sectionendproofmaintheorems}
We end the proof of our main theorem.
We thus consider a non-elementary relatively hyperbolic group $\Gamma$ with respect to virtually abelian subgroups.
We also consider a symmetric probability measure $\mu$ whose finite support generates $\Gamma$.
We let $R$ be its spectral radius.

\medskip
We first note that Proposition~\ref{propconvergencegeometricMartin} is a consequence of Proposition~\ref{propconvergenceconical}, Proposition~\ref{propconvergenceparabolic1} and Proposition~\ref{propconvergenceparabolic2}.

This proposition gives a map $\xi \in \Gamma \cup \partial_r\Gamma\mapsto \tilde{\xi}\in \Gamma \cup \partial_{r\mu}\Gamma$.
We will denote by $K_r(\cdot, \tilde{\xi})$ the corresponding limit Martin kernel.
Our goal is to prove Theorem~\ref{maintheorem}.
We first prove the following lemma.

\begin{lemma}\label{lemmaone-to-oneMartin}
Fix $r\leq R$.
Let $\xi$ and $\xi'$ be two distinct points in the $r$-geometric boundary.
Then, there exists a sequence $g_n$ such that
\begin{enumerate}
    \item either $K_r(g_n,\tilde{\xi})$ stays bounded away from 0 whereas $K_r(g_n,\tilde{\xi}')$ converges to 0,
    \item or $K_r(g_n,\tilde{\xi})$ tends to infinity whereas $K_r(g_n,\tilde{\xi}')$ stays bounded away from infinity.
\end{enumerate}
\end{lemma}

\begin{proof}
By definition of the geometric boundary, if $g_n\in \Gamma$ converges to $\xi$, then it converges to a point in the Bowditch boundary.
In other words, $\xi$ and $\xi'$ define two points $\zeta,\zeta'$ in the Bowditch boundary.

We first prove that if $g_n$ converges to $\zeta$ and if $\zeta'\neq \zeta$, then $K(g_n,\tilde{\xi}')$ converges to 0.
Indeed, consider a sequence $g'_m$ converging to $\zeta'$ in the Bowditch boundary.
Since the Floyd boundary covers the Bowditch boundary, according to the results of Gerasimov \cite{Gerasimov}, there exists some uniform $\delta>0$ such that $\delta_e^f(g_n,g_m')\geq \delta$.
According to Theorem~\ref{thmweakAncona}, there exists $C\geq 0$ such that for every $n$ and $m$, we have
$$G_r(g_n,g'_m)\leq CG_r(g_n,e)G_r(e,g'_m),$$
so that
$$K_r(g_n,g'_m)\leq CG_r(g_n,e).$$
Letting $m$ tend to infinity, we get
$$K_r(g_n,\tilde{\xi}')\leq CG_r(g_n,e).$$
Since $g_n$ goes to infinity, $G_r(g_n,e)$ converges to 0 and we get the desired result.

Next, we prove that we can find a sequence $g_n$ converging to $\zeta$ such that $K(g_n,\tilde{\xi})$ stays bounded away from 0.
Assume first that $\zeta$ is conical and fix a geodesic $[e,\zeta)$ from $e$ to $\zeta$.
Since $\zeta$ is conical, there exists an infinite sequence $g_n$ of transition points on $[e,\zeta)$ converging to $\zeta$, see \cite[Lemma~2.20]{YangPS}.
Proposition~\ref{Floydgeo} shows that for $n\leq m$, $\delta_{g_n}^f(e,g_m)\geq \delta$ for some uniform $\delta>0$ and so
Theorem~\ref{thmweakAncona} shows that there exists $C\geq 0$ such that
$$G_r(e,g_m)\leq C G_r(e,g_n)G_r(g_n,g_m).$$
Letting $m$ tend to infinity, we see that $K_r(g_n,\tilde{\xi})\geq \frac{1}{C}\frac{1}{G_r(e,g_n)}$, so that $K_r(g_n,\tilde{\xi})$ tends to infinity when $n$ tends to infinity.
Assume now that $\zeta$ is parabolic. For simplicity assume first that its stabilizer $H$ is in $\Omega_0$.
Then, there exists $u\in \R^d$ such that
$$K_r(g,\tilde{\xi})=C_g\mathrm{e}^{\pi(g)\cdot u},$$
where $\pi(g)$ is the projection of $g$ on $H$ and where $C_g$ only depends on the distance from $g$ to $H$.
Indeed, either the 1-Martin boundary of $p_{H,\eta,r}$ is reduced to a point and then we can use Lemma~\ref{minimalharmonicnilpotent}, or it is homeomorphic to a sphere and we use the explicit form of the Martin kernel given by \cite[Proposition~3.29]{Dussaule}.
If $u=0$, then for any sequence $g_n$ going to infinity in $H$, $K_r(g_n,\tilde{\xi})$ stays bounded.
If not, then choose $\theta\in \mathbb{S}^{d-1}$ such that $\theta \cdot u>0$ and let
$g_n=(x_n,j_n)$ in $H=\Z^d\times\{1,...,N_{\eta}\}$ such that $g_n$ goes to infinity and
$\frac{g_n}{\|g_n\|}$ converges to $\theta$.
Then $K_r(g_n,\tilde{\xi})$ tends to infinity.
In any case, we found a sequence $g_n$ such that $K_r(g_n,\tilde{\xi})$ is bounded away from 0.
If the stabilizer of $\zeta$ is not in $\Omega_0$, then we get the same result, multiplying everything on the left by some $g\in \Gamma$ such that $gHg^{-1}$ is in $\Omega_0$.

We now end the proof.
We begin with the case where $\zeta\neq \zeta'$.
Then, we choose a sequence $g_n$ converging to $\zeta$ such that $K(g_n,\tilde{\xi})$ stays bounded away from 0.
Since $K(g_n,\tilde{\xi}')$ converges to 0, we are done.
We thus assume that $\zeta=\zeta'$.
The only possibility is that $\zeta$ is parabolic.
As above, we first assume that its stabilizer $H$ is in $\Omega_0$.
Thus, there exists $u\neq u'\in \R^d$ such that
$$K_r(g,\tilde{\xi})=C_g\mathrm{e}^{\pi(g)\cdot u},$$
$$K_r(g,\tilde{\xi}')=C_g\mathrm{e}^{\pi(g)\cdot u'}.$$
We have to be a little more precise than above.
If $u=0$, we choose $\theta$ such that $\theta \cdot u'<0$.
Otherwise, we choose $\theta$ such that $\theta \cdot u>0$ and $\theta\cdot u'\leq 0$.
This is always possible, whether $u'=0$ or not.
We then choose again a sequence $g_n=(x_n,j_n)$ in $H=\Z^d\times\{1,...,N_{\eta}\}$ such that $g_n$ goes to infinity and
$\frac{g_n}{\|g_n\|}$ converges to $\theta$.
Then, in the first case, $K_r(g_n,\tilde{\xi})$ stays bounded away from 0 whereas $K_r(g_n,\tilde{\xi}')$ converges to 0.
In the second case, $K_r(g_n,\tilde{\xi})$ tends to infinity, whereas $K_r(g_n,\tilde{\xi}')$ stays bounded away from infinity.
If $H$ is not in $\Omega_0$, then we get the same result, multiplying everything on the left by some $g\in \Gamma$ such that $gHg^{-1}$ is in $\Omega_0$.
We dealt with every case, so this concludes the proof.
\end{proof}

We can finish the proof of Theorem~\ref{maintheorem}.
\begin{proof}
First, the map $\xi \in \Gamma \cup \partial_r\Gamma\mapsto \tilde{\xi}\in \Gamma \cup \partial_{r\mu}\Gamma$ is continuous.
Indeed since the geometric compactifications are metrizable and since $\Gamma$ is dense in these compactifications, it is enough to prove that whenever $g_n\in \Gamma$ converges to $\xi$, then $\tilde{g}_n$ converges to $\tilde{\xi}$, which is given by Propositions~\ref{propconvergenceconical},~\ref{propconvergenceparabolic1} and~\ref{propconvergenceparabolic2}.

We now prove that it is also one-to-one.
We just need to prove that whenever $\xi\neq \xi'\in \partial_r\Gamma$, then $\tilde{\xi}\neq \tilde{\xi}'\in \partial_{r\mu}\Gamma$.
This is given by Lemma~\ref{lemmaone-to-oneMartin}.

The Martin compactification also is metrizable and $\Gamma$ is dense in this compactification.
Let $\tilde{\xi}\in \partial_{r\mu}\Gamma$.
Then, there exists $g_n\in \Gamma$ converging to $\tilde{\xi}$.
Up to taking a subsequence, $g_n$ converges to $\xi$ in the $r$-geometric compactification, so $\tilde{\xi}$ is the image of $\xi$ by the above map.

Thus, the map we constructed is one-to-one, onto and continuous.
Since both spaces are metrizable and compact, this map is a homeomorphism.
\end{proof}

\section{Minimality and stability of the Martin boundary}\label{Sectionstability}
We prove here Theorem~\ref{mainthmstability} and Theorem~\ref{minimality} and so we consider a relatively hyperbolic group with respect to virtually abelian subgroups.
We also consider a symmetric probability measure $\mu$ whose finite support generates $\Gamma$.

\subsection{Minimality}
To prove minimality of the Martin boundary, we will use the following lemmas.
The first one is a general result in potential theory.

\begin{lemma}\label{lemmaminimalMartin}\cite[Proposition~II.1.6]{AnconaLecturenotes}.
Let $p$ be a finitely supported transition kernel on a countable space $E$ which is transient and irreducible.
Let $\phi$ be a non-negative harmonic function on $E$ and $\nu_{\phi}$ the corresponding measure on the minimal Martin boundary $\partial^{min}E$.
For $\mu_{\phi}$-almost every point $\tilde{y}$ in $\partial^{min}E$, $\frac{G(y_n,x_0)}{\phi(y_n)}$ converges to 0 when $y_n$ converges to $\tilde{y}$.
\end{lemma}

We will also need the following result, which is a refinement of
Lemma~\ref{lemmaone-to-oneMartin}.

\begin{lemma}\label{one-to-onepreciseparabolic}
Consider a non-degenerate parabolic subgroup $H$ and let $\xi_1\neq \xi_2$ be two points in the geometric boundary of $H$ and $\tilde{\xi}_1$ and $\tilde{\xi}_2$ the corresponding points in the Martin boundary.
There exists a compact neighborhood $\mathcal{U}$ of $\xi_1$ in the geometric boundary of $H$ and a sequence $g_n$ in $H$ such that
\begin{enumerate}
    \item either $K_r(g_n,\tilde{\xi})$ tends to infinity, uniformly over $\xi\in \mathcal{U}$ and $K_r(g_n,\tilde{\xi}_2)$ stays bounded away from infinity,
    \item or $K_r(g_n,\tilde{\xi})$ stays uniformly bounded away from 0 over $\mathcal{U}$ and $K_r(g_n,\tilde{\xi}_2)$ converges to $0$.
\end{enumerate}
\end{lemma}

\begin{proof}
This is exactly \cite[Proposition~6.1]{DGGP}, but the authors only deal with the 1-Martin boundary there.
We rewrite the proof for convenience.
For every point in the geometric boundary of $H$, there exists $u\in \R^d$, with $\lambda (u)=1$ and such that
$$K_r(g,\tilde{\xi})=C\mathrm{e}^{\pi(g)\cdot u}.$$
The induced map $\partial H\to \{u,\lambda(u)=1\}$ is a homeomorphism.
Let $u_1,u_2$ be the points corresponding to $\xi_1,\xi_2$.
Since $\xi_1\neq \xi_2$, there exists a compact neighborhood $\mathcal{U}$ of $u_1$ such that $u_2\notin \mathcal{U}$.
We can see $\mathcal{U}$ as a compact neighborhood of $\xi_1$ in $\partial H$.

First, assume that $u_1\neq 0$.
Then, we can choose $\mathcal{U}$ such that $0\notin \mathcal{U}$.
We then choose $\theta \in \mathbb{S}^{d-1}$ such that $\theta \cdot u>0$ for every $u\in \mathcal{U}$ and such that $\theta\cdot u_2\leq 0$.
Let $g_n=(x_n,j_n)$ be a sequence in $H=\Z^d\times \{1,..,N_\eta\}$ such that $x_n$ goes to infinity and $\frac{x_n}{\|x_n\|}$ converges to $\theta$.
Then,
$K_r(g_n,\tilde{\xi})$ tends to infinity whereas $K_r(g_n,\tilde{\xi}_2)$ stays bounded away from infinity.

Assume then that $u_1=0$.
Then $u_2\neq 0$.
Choose $\theta$ such that $\theta \cdot u\geq 0$ for every $u\in \mathcal{U}$ and $\theta \cdot u_2<0$.
Let $g_n=(x_n,j_n)$ be a sequence in $H=\Z^d\times \{1,..,N_\eta\}$ such that $x_n$ goes to infinity and $\frac{x_n}{\|x_n\|}$ converges to $\theta$.
Then,
$K_r(g_n,\tilde{\xi})$ stays bounded away from 0 whereas $K_r(g_n,\tilde{\xi}_2)$ converges to 0.
This concludes the proof.
\end{proof}

We can now prove Theorem~\ref{minimality}.
\begin{proof}
Let $\tilde{\xi}_0\in \partial_{r\mu}\Gamma$.
Then, $K_r(\cdot,\tilde{\xi}_0)$ is an $r$-harmonic function.
Let $\nu$ be the corresponding measure in the $r$-minimal Martin boundary, that is
$$K_r(\cdot,\tilde{\xi}_0)=\int_{\partial_{r\mu}^{min}\Gamma}K_r(\cdot,\tilde{\xi})d\nu(\tilde{\xi}).$$
To prove that $\tilde{\xi}_0$ is minimal, we just need to prove that the support of $\nu$ is reduced to $\{\tilde{\xi}_0\}$.
We will actually prove that any other point cannot be in the support of $\nu$.

Let $\tilde{\xi}_0'$ be another point.
Assume first that $\tilde{\xi}_0$ and $\tilde{\xi}_0'$ correspond to different points $\zeta_0$ and $\zeta_0'$ in the Bowditch boundary.
Then, there exists a compact neighborhood $\mathcal{U}$ of $\tilde{\xi}_0'$ such that for every $\tilde{\xi}$ in $\mathcal{U}$, $\delta_e^f(\zeta_0,\zeta_0')\geq \delta$, so that
Theorem~\ref{thmweakAncona} shows that if $g_n$ converges to $\tilde{\xi}\in \mathcal{U}$, then
$$K_r(g_n,\tilde{\xi}_0)\leq CG_r(g_n,e).$$
This proves that $\frac{G_r(g_n,e)}{K_r(g_n,\tilde{\xi}_0)}$ cannot converge to 0 and so Lemma~\ref{lemmaminimalMartin} shows that $\tilde{\xi}_0'$ cannot be in the support of $\nu$.

Assume now that $\tilde{\xi}_0$ and $\tilde{\xi}_0'$ correspond to the same point in the Bowditch boundary.
This point necessarily is a non-degenerate parabolic limit point.
Lemma~\ref{one-to-onepreciseparabolic} shows that there is a neighborhood $\mathcal{U}$ of $\tilde{\xi}_0'$ such that
\begin{enumerate}
    \item either $K_r(g_n,\tilde{\xi})$ tends to infinity, uniformly over $\xi\in \mathcal{U}$ and $K_r(g_n,\tilde{\xi}_0)$ stays bounded away from infinity,
    \item or $K_r(g_n,\tilde{\xi})$ stays uniformly bounded away from 0 over $\mathcal{U}$ and $K_r(g_n,\tilde{\xi}_0)$ converges to $0$.
\end{enumerate}
By definition,
$$K_r(g_n,\tilde{\xi}_0)=\int_{\partial_{r\mu}^{min}\Gamma}K_r(g_n,\tilde{\zeta})d\nu(\tilde{\zeta})\geq \int_{U}K_r(g_n,\tilde{\zeta})d\nu(\tilde{\zeta}).$$
Thus, in the first case, $K_r(g_n,\tilde{\xi}_0)\geq \alpha_n\nu(U)$ for $n$ large enough, where $\alpha_n$ tends to infinity and so $\mu(U)=0$.
In the second case, $K_r(g_n,\tilde{\xi}_0)\geq C\mu(U)$ for some constant $C$, for $n$ large enough and so again, $\mu(U)=0$.
Hence, $\tilde{\xi}_0'$ is not in the support of $\nu$.

We proved that in every case, $\tilde{\xi}'_0$ is not in the support of $\nu$, so that the support of $\nu$ is reduced to $\{\tilde{\xi}_0\}$, hence $\tilde{\xi}_0$ is minimal.
\end{proof}

\subsection{Stability}
We still consider a relatively hyperbolic group with respect to virtually abelian subgroups and a symmetric probability measure $\mu$ whose finite support generates $\Gamma$.
We assume that $\mu$ is not spectrally degenerate.
According to Theorem~\ref{maintheorem}, the $r$-Martin compactifications are all homeomorphic to each other for $r\leq R$.
Our goal is to prove that the map $(r,x,\xi)\mapsto K_r(x,\xi)$ is continuous in $(r,x,\xi)\in (0,R]\times \Gamma \times \partial_{\mu}\Gamma$.

Theorem~\ref{maintheorem} also gives a homeomorphism from the $r$-geometric compactification to the $r$-Martin compactification.
We first deal with conical limit points, that is we assume that $\xi$ is the image of a conical point by this homeomorphism.
We do not need to assume spectral non-degeneracy for those points.

\begin{proposition}\label{stabilityconical}
Let $0<r_0\leq R$ and $x_0\in \Gamma$.
Let $\xi_0$ be a conical limit point and identify $\xi_0$ with a point in $\partial_{r_0\mu}\Gamma$.
Then, $(r,x,\xi)\in (0,R]\times \Gamma \times \Gamma \cup \partial_\mu\Gamma \mapsto K_r(x,\xi)$ is continuous at $(x_0,r_0,\xi_0)$.
\end{proposition}

\begin{proof}
Let $(r_n,x_n,\xi_n)$ be a sequence converging to $(r_0,x_0,\xi_0)$.
Since $\Gamma$ is discrete, we can assume that $x_n=x_0$ is fixed.
Let $z_m$ be a sequence in $\Gamma$ converging to $\xi_0$ and $z'_{n,m}$ be a sequence in $\Gamma$ converging to $\xi_n$, when $m$ tends to infinity.
Since $\xi_0$ is conical, \cite[Lemma~2.2, Proposition~2.3]{hlv} show there exists $\delta>0$ such that the following holds.
For every $k$, there exists $n$ large enough and $m$ large enough (depending on $n$) such that one can find $k$ points $w_1,...,w_k$ in $\Gamma$ satisfying
$\delta_{w_i}^f(x_0,z_m)\geq \delta$ and $\delta_{w_i}^f(e,z'_{n,m})\geq \delta$.
Using Theorem~\ref{thmstrongAncona}) and letting $m$ tend to infinity, we thus get that for every $r$,
$$\left |\frac{K_{r}(x_0,\xi_n)}{K_{r}(x_0,\xi_0)} -1\right |\leq C\rho^{k}.$$
In particular, for every $\epsilon>0$, for every large enough $n$, we have
$$\left |\log K_{r_n}(x,\xi_n)-K_{r_0}(x,\xi_0)\right |\leq \epsilon.$$
This shows that $(r,x,\xi)\mapsto \log K_r(x,\xi)$ is continuous at $(r_0,x_0,\xi_0)$, which concludes the proof.
\end{proof}

As explained in the introduction, this proposition also shows that the Martin boundary is strongly stable in hyperbolic groups, although this is implicit in \cite{Gouezel-local}.

\medskip
We can now deal with parabolic limit points and we fix a parabolic subgroup $H$.
Theorem~\ref{maintheorem} shows that for $\xi$ in the geometric boundary $\partial H$ of $H$, if $g_n$ in $\Gamma$ converges to $\xi$, then $K_r(\cdot, g_n)$ converges to a point $K_r(\cdot, \xi)$ for every $r\leq R$.

We use the notations of Section~\ref{Sectionparaboliclimitpointsdegenerate} and Section~\ref{Sectionparaboliclimitpointsnondegenerate}.
We denote by $p_{H,\eta,r}$ the first return kernel induced by $r\mu$ to the $\eta$-neighborhood $N_\eta(H)$ of $H$.
Note that if $r\leq r'$,
then $p_{H,\eta,r}(g,g')\leq p_{H,\eta,r'}(g,g')$ for every $g,g'\in N_\eta(H)$.
We also consider the matrix $F^{(r)}(u)$ defined by
$$F^{(r)}_{j,k}(u)=\sum_{x\in \Z^d}p_{j,k}(e,x)\mathrm{e}^{x\cdot u}$$
as in Section~\ref{Sectionparaboliclimitpointsnondegenerate}.
We added the exponent $r$ to insist on the dependency on this parameter.
Also note that this matrix depends on $\eta$.

According to Lemma~\ref{exponentialmomentsinducedchain}, for every $M$, there exists $\eta$ such that $p_{H,\eta,r}$ has exponential moments up to $M$, independently of $r$.
In particular, if $\eta$ is large enough, the set of $u\in \R^d$ such that the coefficients of $F^{(r)}(u)$ are finite contains a set $\{u\in \R^d,\|u\|\leq M\}$ which can be taken independently of $r$.
We fix such a large $\eta$.
Then, all matrices $F^{(r)}(u)$ have the same size, say $K\times K$, where $K$ only depends on $\eta$.
Let $F^{(r)}$ the function $u \in \{u,\|u\|\leq M\}\mapsto F^{(r)}(u)$.
We endow $M_K(\R^d)$ with a matrix norm.
For fixed $r$, the function $F^{(r)}$ is continuous in $u$.
We then endow the space of continuous functions from $\{u,\|u\|\leq M\}$ to $M_K(\R^d)$ with the norm $\|\cdot \|_\infty$.
We prove the following.

\begin{lemma}\label{lemme1pourpropstabilite3}
For these norms, the function $r\mapsto F^{(r)}$ is continuous.
\end{lemma}

\begin{proof}
Let $u\in \{u,\|u\|\leq M\}$, $r,r'\leq R_\mu$.
We need to show that $\|F^{(r)}-F^{(r')}\|_\infty$ tends to 0, as $r$ converges to $r'$.
It is enough to show that every entry of $F^{(r)}$ converges to the corresponding entry of $F^{(r')}$.
We thus fix $k$ and $j$.
We also fix $\epsilon>0$.
Then,
$$F^{(r)}_{j,k}(u)-F^{(r')}_{j,k}(u)=\sum_{x\in \Z^d}p_{j,k}^{(r)}(e,x)\mathrm{e}^{x\cdot u}-\sum_{x\in \Z^d}p_{j,k}^{(r')}(e,x)\mathrm{e}^{x\cdot u}.$$
Note that
$$\sum_{x\in \Z^d}p_{j,k}^{(r)}(e,x)\mathrm{e}^{x\cdot u}\leq \sum_{x\in \Z^d}p_{j,k}^{(r)}(e,x)\mathrm{e}^{M\|x\|}$$
and similarly for $r'$.
Since both $p_{\eta,H,r}$ and $p_{\eta,H,r'}$ have exponential moments up to $M$, there exists $D$ which does not depend on $r$ and which satisfies that
$$\sum_{x,\|x\|> D}p_{j,k}^{(r)}(e,x)\mathrm{e}^{x\cdot u}\leq \epsilon/4, \hspace{1cm} \sum_{x,\|x\|> D}p_{j,k}^{(r')}(e,x)\mathrm{e}^{x\cdot u}\leq \epsilon/4.$$
Consequently,
$$\left |F^{(r)}_{j,k}(u)-F^{(r')}_{j,k}(u)\right |\leq \epsilon/2+\sum_{x,\|x\|\leq D}\left |p_{j,k}^{(r)}(e,x)-p_{j,k}^{(r')}(e,x)\right |\mathrm{e}^{x\cdot u}.$$
We then set
$$\Lambda=\sum_{x,\|x\|\leq D}\mathrm{e}^{M\|x\|}.$$
For a  fixed $\gamma$, the map $r\mapsto p_{\eta,H,r}(\gamma)$ is continuous, so
there exists $\alpha$ such that for every $r$ and $r'$ with $|r-r'|\leq \alpha$ and for every $x\in \Z^d$ such that $\|x\|\leq D$,
$$\left |p_{j,k}^{(r)}(e,x)-p_{j,k}^{(r')}(e,x)\right |\leq \frac{\epsilon}{2\Lambda}.$$
We therefore get $\left |F^{(r)}_{j,k}(u)-F^{(r')}_{j,k}(u)\right |\leq \epsilon$ as soon as $|r-r'|\leq \alpha$.
Since $\alpha$ does not depend on $u$,
we deduce that $\|F^{(r)}_{j,k}-F^{(r')}_{j,k}\|_{\infty}\leq \epsilon$ as soon as $|r-r'|\leq \alpha$.
This concludes the proof.
\end{proof}

We denote by $\lambda^{(r)}(u)$ the dominant eigenvalue of $F^{(r)}(u)$ and by $C^{(r)}(u)$, respectively $\nu^{(r)}(u)$, right, respectively left, eigenvectors associated with $\lambda^{(r)}(u)$.
Also, let $\lambda^{(r)}$, respectively $C^{(r)}$, respectively $\nu^{(r)}$ be the functions $u\mapsto \lambda^{(r)}(u)$, respectively
$u\mapsto C^{(r)}(u)$, respectively $u\mapsto \nu^{(r)}(u)$.

Choose arbitrary norms on $\R^K$ and endow the spaces of continuous functions from $\{u,\|u\|\leq M\}$ to $\R^K$ and from $\{u,\|u\|\leq M\}$ to $\R$ with the $\|\cdot \|_\infty$ norm.
Since the maps $F\mapsto \lambda$, $F\mapsto C$ and $F\mapsto \nu$ all are continuous on the open set consisting of matrices $F$ having a simple dominant eigenvalue, we also have the following.

\begin{lemma}\label{lemme1'pourpropstabilite3}
For these norms, the functions $r\mapsto \lambda^{(r)}$, $r\mapsto C^{(r)}$ and $r\mapsto \nu^{(r)}$ are continuous.
\end{lemma}

Recall that there is a homeomorphism between $\partial H$, identified with $\mathbb{S}^d$ and $\{u,\lambda^{(r)}(u)=1\}$.
This homeomorphism is given by
$u\mapsto \frac{\nabla \lambda^{(r)}(u)}{\|\nabla \lambda^{(r)}(u)\|}$ and we denote it by $\phi_r$
Given $\gamma\in \mathcal{V}_{\eta}(H)$, we identify $\gamma$ with $(x,k)$, where $x\in \Z^d$ and $k\in \{1,...,N_{\eta}\}$.
Also identify $e$ with $(x_0,k_0)$.
Then, \cite[Proposition~3.29]{Dussaule} shows that
$$K_r(\gamma,\xi)=\frac{C^{(r)}(u_r)_k}{C^{(r)}(u_r)_{k_0}}\mathrm{e}^{\phi_r(\xi)\cdot (x-x_0)},$$
where $C^{(r)}(u_r)_j$ is the $j$th coordinate of  $C^{(r)}(u_r)$.

\begin{lemma}\label{lemme1''pourpropstabilite3}
With these notations, the map $(\xi,r)\in \partial H\times (0,R_\mu] \mapsto \phi_r(\xi)$ is continuous.
\end{lemma}

\begin{proof}
Let $(\xi_n,r_n)$ converge to $(\xi,r)$.
Set $u_n=\phi_{r_n}(\xi_n)$.
Since we have $\|u_n\|\leq M$, we can assume that $u_n$ converges to $u$.
We thus need to show that $u=\phi_r(\xi)$.

Since the random walk is not spectrally degenerate, the set of $u$ such that $\lambda^{(r)}(u)=1$ is not reduced to a point and since $\nabla^{(r)}$ is strictly convex, on this set we have $\nabla\lambda^{(r)}(u)\neq 0$.
We show exactly like in the proof of Lemma~\ref{lemme1'pourpropstabilite3} that $r\mapsto \frac{\nabla\lambda^{(r)}(\cdot)}{\|\nabla\lambda^{(r)}(\cdot)\|}$ is continuous.
Therefore, $\frac{\nabla\lambda^{(r_n)}(u_n)}{\|\nabla\lambda^{(r_n)}(u_n)\|}$ converges to $\frac{\nabla\lambda^{(r)}(u)}{\|\nabla\lambda^{(r)}(u)\|}$.
By definition, $\frac{\nabla\lambda^{(r_n)}(u_n)}{\|\nabla\lambda^{(r_n)}(u_n)\|}=\xi_n$, which converges to $\xi$.
In particular, $\frac{\nabla\lambda^{(r)}(u)}{\|\nabla\lambda^{(r)}(u)\|}=\xi$.
Also, since $u_n$ converges to $u$, Lemma~\ref{lemme1'pourpropstabilite3} shows that $\lambda^{(r_n)}(u_n)$ converges to  $\lambda^{(r)}(u)$.
Since $\lambda^{(r_n)}(u_n)=1$, we also have $\lambda^{(r)}(u)=1$.
Finally, note that $\phi_r(\xi)$ is the unique point satisfying both $\lambda^{(r)}(\phi_r(\xi))=1$ and $\frac{\nabla\lambda^{(r)}(\phi_r(\xi))}{\|\nabla\lambda^{(r)}(\phi_r(\xi))\|}=\xi$.
We thus have $u=\phi_r(\xi)$, which concludes the proof.
\end{proof}

Note that this lemma still holds for spectrally degenerate random walks.
Indeed, assume that $r=R_\mu$ and that $\mu$ is spectrally degenerate.
We use the same notations and assume that $u_n$ converges to some $u$.
We need to prove that $u=\phi_r(\xi)$.
In this situation, the set of $u$ such that $\lambda^{(r)}(u)=1$ is reduced to $\phi_r(\xi)$.
We again have that $\lambda^{(r_n)}(u_n)$ converges to $\lambda^{(r)}(u)$, so that $\lambda^{(r)}(u)=1$.
This proves again that $u=\phi_r(\xi)$.

\textit{In particular, we always have that for fixed $\gamma$ and $\xi$, the map $r\mapsto K_r(\gamma,\xi)$ is continuous, whether $\mu$ is spectrally degenerate or not.}
However, to get continuity in the variable $(\gamma,\xi,r)$ in what follows, we do need to assume that $\mu$ is not spectrally degenerate.
Using the work of \cite{Dussaule}, we prove the following.

\begin{lemma}\label{uniformconvergenceGreenalongparabolics}
Let $h\in N_{\eta}(H)$ and $h_n$ be a sequence in $N_{\eta}(H)$ converging to $\xi \in \partial H$.
Then, $K_r(h,h_n)$ converges to $K_r(h,\xi)$.
If $h$ is fixed, then the convergence is uniform in $r$ lying in a compact subset of $(0,R]$.
\end{lemma}

\begin{proof}
This is given by \cite[Proposition~3.29]{Dussaule}.
The fact that the convergence is uniform in $r$ follows from \cite[Lemma~3.28]{Dussaule}, applied to the parameter $\tilde{u}=(r,u)$.
To apply this lemma, we need to ensure that $\beta_{\tilde{u}}:=\nabla\lambda^{(r)}(u)\neq 0$.
Since $\mu$ is not spectrally degenerate, the set of $u$ such that $\lambda^{(r)}(u)=1$ is not reduced to a point and since $\lambda^{(r)}$ is strictly convex, we have $\beta_{(r,u)}\neq 0$, which concludes the proof.
\end{proof}

We can now prove the statement analogous to Proposition~\ref{stabilityconical} for parabolic limit points.

\begin{proposition}\label{stabilityparabolic}
Let $0<r_0\leq R$ and $x_0\in \Gamma$.
Let $\xi_0$ be a point in $\partial H$ for some parabolic subgroup $H$ and identify $\xi_0$ with a point in $\partial_{r_0\mu}\Gamma$.
Then, the map $(r,x,\xi)\in (0,R]\times \Gamma \times \Gamma \cup \partial_\mu\Gamma \mapsto K_r(x,\xi)$ is continuous at $(x_0,r_0,\xi_0)$.
\end{proposition}

\begin{proof}
Let $(r_n,x_n,\xi_n)$ converge to $(r_0,x_0,\xi_0)$.
As in the proof of Proposition~\ref{stabilityconical}, we can assume that $x_n=x_0$ for every $n$.
We can thus fix large enough $\eta$ so that $x_n\in N_\eta(H)$ for every $n$.
We can also, up to extracting sub-sequences, assume that either $\xi_n\in \partial H$ for every $n$, or that $\xi_n\notin \partial H$ for every $n$.
In the first case, continuity is given by Lemma~\ref{lemme1''pourpropstabilite3}.

Let us assume then that for every $n$, $\xi_n\notin\partial H$.
Then, the projection $\pi_n$ of $\xi_n$ on $H$ is well defined up to a bounded distance in the Cayley graph.
Moreover, $\pi_n$ converges to $\xi_0$ in $\partial H$.
Fix $\epsilon>0$.
Corollary~\ref{Anconaepsilon} shows that there exists $S$, independent of $n$, such that
$$G_{r_n}(x_0,\xi_n)\leq \frac{1}{1-\epsilon}\sum_{u\in B_S(e)}G_{r_n}(x_0,\pi_nu)G_{r_n}(\pi_nu,\xi_n).$$
Also, we have
$$G_{r_n}(e,\xi_n)\geq \sum_{u\in B_S(e)}G_{r_n}(e,\pi_nu)G_{r_n}(\pi_nu,\xi_n).$$
Since $\pi_nu$ converges to $\xi_0$ in $\partial H$, Lemma~\ref{uniformconvergenceGreenalongparabolics} shows that for large enough $n$,
$$|K_{r_n}(x_0,\pi_nu)-K_{r_n}(x_0,\xi_0)|\leq \epsilon.$$
Then, as in the proof of Proposition~\ref{propconvergenceparabolic1}, we finally get that for large enough $n$,
$$(1-\epsilon)(K_{r_n}(x_0,\xi_0)-\epsilon)\leq K_{r_n}(x_0,\xi_n)\leq  \frac{1}{1-\epsilon}(K_{r_n}(x_0,\xi_0)+\epsilon).$$
Finally, $K_{r_n}(x_0,\xi_0)$ converges to $K_{r_0}(x_0,\xi_0)$, so that $K_{r_n}(x_0,\xi_n)$ also converges to $K_{r_0}(x_0,\xi_0)$,
which concludes the proof.
\end{proof}


\section{Strong stability in low dimension}\label{Sectionspectrallydegeneratemeasures}
We first prove here the following proposition.
We will then use it to prove strong stability of the Martin boundary when the parabolic subgroups have small rank.

\begin{proposition}\label{propdegenerateimplieshighrank}
Let $\Gamma$ be a relatively hyperbolic group and
let $H$ a virtually abelian parabolic subgroup of rank $d$.
Let $\mu$ be a probability measure on $\Gamma$ whose symmetric finite support generates $\Gamma$.
Assume that $\mu$ is spectrally degenerate along $H$.
Then, $d\geq 5$.
\end{proposition}

We begin with the following result proved in \cite{GouezelLalley}.
\begin{lemma}\label{lemmafirstderivativeGreen}\cite[Proposition~1.9]{GouezelLalley}
Let $\Gamma$ be a finitely generated group and let $\mu$ be a probability measure on $\Gamma$, with spectral radius $R$.
Then, for every $g,g'\in \Gamma$ and for every $r\leq R$, we have
$$\frac{d}{dr}(rG_r(g,g'))=\sum_{g''\in \Gamma}G_r(g,g'')G_r(g'',g').$$
\end{lemma}

We now prove the following key result.

\begin{proposition}\label{Greenderivativealongparabolicisfinite}
Let $\Gamma$ be a relatively hyperbolic group and let $\mu$ be a symmetric probability measure whose finite support generates $\Gamma$, with spectral radius $R$.
Let $H$ be a parabolic subgroup.
Then, for every $r\leq R$ and for every $\eta\geq 0$, we have
$$\sum_{g\in N_{\eta}(H)}G_r(e,g)G_r(g,e)<+\infty.$$
\end{proposition}

In other words, the derivative of the Green function along parabolic subgroups is finite at the spectral radius.
Notice that we do not need to assume that the measure is spectrally non-degenerate along $H$.

\begin{proof}
We just need to prove the result for $r=R$, since the Green function is non-decreasing in $r$.
Also, we just need to prove the result for $\eta=0$.
We first fix $r<R$, so that $\frac{d}{dr}G_r(e,e)$ is finite.
According to Lemma~\ref{lemmafirstderivativeGreen}, the sum
$\sum_{g\in \Gamma}G_r(e,g)G_r(g,e)$ is thus finite.

Let $h_1,...,h_m\in H$.
According to Lemma~\ref{extensionlemmaFloyd}, there exists $a_1$ in $B_C(e)$ such that
$\delta_{h_1}(e,h_1ah_2)\geq \delta$ for some uniform $C\geq 0$ and $\delta>0$.
We let $g_1=h_1a_1h_2$.
Then, there exists $a_2\in B(e,C)$ such that $\delta_{g_1}(e,g_2)\geq \delta$, where $g_2=g_1a_2h_3$.
By induction, we see there exist $a_1,...,a_{m-1}$ in $B(e,C)$ such that letting $g_i=h_1a_1h_2...a_ih_{i+1}$, we have
$\delta_{g_i}(e,g_{i+1})\geq \delta$.

Also, whenever $g$ is fixed, Lemma~\ref{boundedtoone} shows there is a finite number of ways of writing
$g$ as $h_1a_1h_2$ as above.
In particular, there exists $C_1$ such that there are at most $C_1^m$ ways of writing $g$ as
$h_1a_1h_2...a_{m-1}h_m$.
We can thus find $C_2$ such that
$$C_2^m\sum_{g\in \Gamma}G_r(e,g)G_r(g,e)\geq \sum_{h_1,...,h_m\in H}G_r(e,g_m)G_r(g_m,e).$$
According to Theorem~\ref{thmweakAncona} (applied $m$ times), we have
\begin{align*}
    C_2^m\sum_{g\in \Gamma}G_r(e,g|r)G_r(g,e)\geq \frac{1}{C_3^{2m}}&\sum_{h_1,...,h_m\in H}G_r(e,h_1)G_r(h_1,e)G_r(e,a_1)G_r(a_1,e)...\\
    &G_r(e,a_{m-1})G_r(a_{m-1},e)G_r(e,h_m)G_r(h_m,e).
\end{align*}
Since $a_j\in B(e,C)$, $G_r(e,a_j)G_r(a_j,e)\geq C_4$ for some $C_4\geq0$.

We thus proved that for every $m\geq 1$,
$$\left (\sum_{h\in H}G_r(e,h)G_r(h,e)\right)^m\leq C_5^m\sum_{g\in \Gamma}G_r(e,g)G_r(g,e)$$
so that
$$\sum_{h\in H}G_r(e,h)G_r(h,e)\leq C_5\left (\sum_{g\in \Gamma}G_r(e,g)G_r(g,e)\right)^{1/m}.$$
Letting $m$ tend to infinity, we thus have
$$\sum_{h\in H}G_r(e,h)G_r(h,e)\leq C_5.$$
Since $r\mapsto G_r(e,h)G_r(h,e)$ is non-decreasing, we can now let $r$ tend to $R$ to obtain
$$\sum_{h\in H}G_R(e,h)G_R(h,e|R)\leq C_5,$$
which concludes the proof.
\end{proof}

Let us recall the following classical result, which is a generalization of the local limit theorem in $\Z^d$.

\begin{lemma}\label{locallimitabelian}\cite[Theorem~13.12]{Woess-book}
Let $H$ be a virtually abelian group of rank $d$.
Let $p$ be an $H$-invariant transition kernel on $H\times \{1,...,N\}$, with spectral radius $R$.
Assume that $p$ is strongly irreducible and has finite exponential moments.
Then, if $R=1$, there exists $C>0$ such that
$$p^{(n)}(e,e)\sim \frac{C}{n^{d/2}},$$
where $p^{(n)}$ denote the $n$th convolution power of $p$.
\end{lemma}

We now prove Proposition~\ref{propdegenerateimplieshighrank}.

\begin{proof}
First, Lemme~\ref{exponentialmomentsinducedchain} shows that if $\eta$ is chosen large enough, then $p_{\eta,H,R_{\mu}}$ has finite exponential moments.
According to Lemma~\ref{lemmafirstderivativeGreen} and Proposition~\ref{Greenderivativealongparabolicisfinite}, for every $r\leq R$, the first derivative of the Green function $G_{\eta,H,r}$ associated with $p_{\eta,H,r}$ is finite at 1.
According to Lemma~\ref{lemmaassumption2firstpart}, if we assume that $\mu$ is spectrally degenerate along $H$, then the spectral radius of $p_{\eta,H,R}$ is 1.
We thus have an $H$-invariant transition kernel $p$ on $H\times \{1,...,N\}$ with finite exponential moments, whose spectral radius is 1 and whose Green function has a finite derivative at 1.
We just need to prove that this can only happen if the rank $d$ of $H$ is at least 5.

Lemma~\ref{locallimitabelian} shows that if $p$ were strongly irreducible, then we would have
$$p^{(n)}\sim \frac{C}{n^{d/2}}.$$
Since the derivative of the Green function at 1 is given by
$\sum_{n\geq 0}np^{(n)}$.
This sum is finite if and only if $\frac{d}{2}-1>1$ or equivalently $d>4$.

To conclude, let us show that we can assume that $p$ is strongly irreducible.
Since the initial random walk on $\Gamma$ is irreducible, we already now that $p$ is irreducible.
We define a new transition kernel $\tilde{p}$ on $N_{\eta}(H)$ by
$$\tilde{p}(h,h')=\alpha \delta_{h,h'}+(1-\alpha)p(h,h'),$$
where $0<\alpha <1$.
Then, $\tilde{p}$ is strongly irreducible.
Moreover, letting $\tilde{G}$ be its Green function, \cite[Lemma~9.2]{Woess-book} shows that
$$\tilde{G}_t(e,e)=\frac{1}{1-\alpha t}G_{\frac{(1-\alpha)t}{1-\alpha t}}(e,e).$$
The radius of convergence of $\frac{1}{1-\alpha t}$ is $\frac{1}{\alpha}>1$.
Also, if $t<\frac{1}{\alpha}$, then
$\frac{(1-\alpha)t}{1-\alpha t}\geq 1$ if and only if $t\geq 1$.
This proves that the radius of convergence of $\tilde{G}$ also is 1
and the formula above shows that the first derivative of $\tilde{G}$ is finite at 1.

We thus constructed a strongly irreducible $H$-invariant transition kernel $p$ on $H\times \{1,...,N\}$ with finite exponential moments, whose spectral radius is 1 and whose Green function has a finite derivative at 1.
This concludes the proof.
\end{proof}

We get Theorem~\ref{stabilitylowdimension} as a corollary.
Indeed, if $\Gamma$ is the fundamental group of a geometrically finite hyperbolic manifold $X$, then it is relatively hyperbolic with respect to the cusp stabilizers, which are virtually abelian groups.
If the dimension of $X$ is at most 5, then the rank of the cusp stabilizers is at most 4, see for example \cite[\S 5.4,\S 5.5]{Ratcliffe} and \cite{Bowditchgeometricallyfinite} for more details on this.
Hence, Proposition~\ref{propdegenerateimplieshighrank} shows that for any symmetric probability measure $\mu$ whose finite support generates $\Gamma$,
$\mu$ cannot be spectrally degenerate along the parabolic subgroups.
Thus, Theorem~\ref{maintheorem} shows that for every $r\leq R$, the $r$-Martin boundary coincides with the 1-geometric boundary.
As noted in \cite[Corollary~1.5]{GGPY}, in this situation, this boundary also coincides with the CAT(0) boundary of $\Gamma$. \qed

\bibliographystyle{plain}
\bibliography{spectral_radius}

\end{document}